\documentclass{article}

\usepackage{makeidx}
\usepackage{amsmath,amsthm,amssymb,graphics,ifthen,a4wide,epsf,tabularx}

\usepackage{amsmath,amsthm,amssymb}
\usepackage{ifthen,a4wide,epsf}
\usepackage{array,tabularx}
\usepackage{epsfig,graphics,graphicx}

\newcommand{\solution}[1]{}

\newtheorem{theorem}{Theorem}[section]
\newtheorem{definition}[theorem]{Definition}
\newtheorem{proposition}[theorem]{Proposition}
\newtheorem{corollary}[theorem]{Corollary}
\newtheorem{lemma}[theorem]{Lemma}
\newtheorem{fact}[theorem]{Remark}
\newtheorem{exemplu}[theorem]{Example}
\newtheorem{exercise}{Exercise}
\newtheorem{notation}[theorem]{Notation}

\newtheorem{problem}[theorem]{Problem}

\newcommand{\bdfn}{\begin{definition}}
\newcommand{\edfn}{\end{definition}}
\newcommand{\bthm}{\begin{theorem}}
\newcommand{\ethm}{\end{theorem}}
\newcommand{\bprop}{\begin{proposition}}
\newcommand{\eprop}{\end{proposition}}
\newcommand{\bcor}{\begin{corollary}}
\newcommand{\ecor}{\end{corollary}}
\newcommand{\blem}{\begin{lemma}}
\newcommand{\elem}{\end{lemma}}
\newcommand{\bfact}{\begin{fact}}
\newcommand{\efact}{\end{fact}}
\newcommand{\bex}{\begin{exemplu}\begin{rm}}
\newcommand{\eex}{\end{rm}\end{exemplu}}
\newcommand{\bxc}{\begin{exercise}}
\newcommand{\exc}{\end{exercise}}
\newcommand{\bntn}{\begin{notation}}
\newcommand{\entn}{\end{notation}}

\newcommand{\be}{\begin{enumerate}}
\newcommand{\ee}{\end{enumerate}}
\newcommand{\bce}{\begin{center}}
\newcommand{\ece}{\end{center}}
\newcommand{\bi}{\begin{itemize}}
\newcommand{\ei}{\end{itemize}}
\newcommand{\bt}{\begin{tabular}}
\newcommand{\et}{\end{tabular}}
\newcommand{\beq}{\begin{equation}}
\newcommand{\eeq}{\end{equation}}
\newcommand{\ba}{\begin{array}} 
\newcommand{\ea}{\end{array}}
\newcommand {\bea} {\begin{eqnarray}}
\newcommand {\eea} {\end {eqnarray}}
\newcommand {\bua} {\begin{eqnarray*}}
\newcommand {\eua} {\end {eqnarray*}}

\newcommand{\ra}{\rightarrow}

\newcommand{\Ra}{\Rightarrow}

\newcommand{\se}{\subseteq}
\newcommand{\ol}{\overline}
\newcommand{\ul}{\underline}

\newcommand{\Lra}{\Leftrightarrow}

\newcommand{\uae}{The following are equivalent}

\newcommand{\ds}{\displaystyle}

\def\R{{\mathbb R}}
\def\N{{\mathbb N}}

\def\Q{{\mathbb Q}}
\def\B{{\mathbb B}}

\newcommand{\Real}{{\mathbb R}}

\newcommand{\limn}{\ds\lim_{n\to\infty}}

\newcommand{\remin}{\mathop{-\!\!\!\!\!\hspace*{1mm}\raisebox{0.5mm}{$
\cdot$}}\nolimits}

\newcommand{\eps}{\varepsilon}

\newcommand{\lambdaxy}{(1-\lambda)x\oplus\lambda y}

\newcommand{\midxy}{\frac12x\oplus\frac12y}

\newcommand{\aquant}{\forall}

\newcommand{\si}{\wedge}

\newcommand{\adw}{{\cal A}^\omega[X,d,W]_{-b}}

\newcounter{ct}

\newcommand{\aomega}{{\cal A}^\omega}
\newcommand{\aomegaxd}{\mathcal{A}^\omega[X,d]_{-b}}
\newcommand{\aomegaxdw}{\mathcal{A}^\omega[X,d,W]_{-b}}

\newenvironment{denseBibitem}[1]
{\bibitem{#1}}

\begin{document}

\title{Proof mining in metric fixed point theory and ergodic theory}
\author{ Lauren\c tiu Leu\c stean
\footnote{This survey was written during the author's stay as an Oberwolfach Leibniz Fellow at Oberwolfach Mathematics Institute whose support is gratefully acknowledged.}\\[0.2cm]
Department of Mathematics, Technische Universit\" at Darmstadt,\\
 Schlossgartenstrasse 7, 64289 Darmstadt, Germany\\[0.1cm] and\\[0.1cm]
Institute of Mathematics "Simion Stoilow'' of the 
Romanian Academy, \\
Calea Grivi\c tei 21, P.O. Box 1-462, Bucharest, Romania\\[0.1cm]
E-mail: leustean@mathematik.tu-darmstadt.de
}
\date{}

\maketitle

\begin{abstract}
\noindent In this survey we present some recent applications of proof mining to the fixed point theory of (asymptotically) nonexpansive mappings and to the metastability (in the sense of Terence Tao) of ergodic averages in uniformly convex Banach spaces.
\end{abstract}

\tableofcontents

\section{Proof Mining}

\noindent By {\em proof mining} we mean the logical analysis, using proof-theoretic tools, of mathematical proofs with the aim of extracting relevant information hidden in the proofs. 
This new information can be both of quantitative nature, such as algorithms and effective bounds, as well as of qualitative nature, such as uniformities in the bounds or weakening the premises. Thus, even if one is not particularly interested in the numerical details of the bounds themselves, in many cases such explicit bounds immediately show the independence of the quantity in question from certain input data. An up-to-date and comprehensive reference for proof mining is Kohlenbach's recent book \cite{Koh08-book}. 

The main proof-theoretic techniques in proof mining are the so-called {\em proof interpretations}. A proof interpretation $I$ maps proofs $p$ in theories $\cal T$ of theorems $A$ into new proofs $p^I$ in theories ${\cal T}^I$ of the interpretation $A^I$ of $A$. In this way, the original mathematical proof is  transformed into  a new enriched proof of a stronger result, from which the desired additional information can be read off. While the soundness of these methods rests on results in mathematical logic, the new proof can again  be written in ordinary mathematics.

This line of research has its roots in Kreisel's program on {\em unwinding of proofs}. Already in the 50's, Kreisel had asked
\bce
{\em "What more do we know if we have proved a theorem by restricted means than if we merely know that it is true?"}
\ece
\noindent Kreisel proposed to apply proof-theoretic techniques - originally developed for foundational purposes - to analyze concrete proofs in mathematics and unwind the extra information hidden in them; see for example \cite{KreMacint82,Luc96} and,  more recently, \cite{Macint05}. Unwinding of proofs has had applications in  number theory \cite{Kre82,Luc89}, algebra \cite{Del96,CosLomRoy01, Coq04,CoqLom06,Lom06} and combinatorics \cite{Bel90,Gir87,KetSol81,Wei04}.

However, the most systematic development of proof mining took place in connection with applications to 
approximation theory \cite{Koh92,Koh93,Koh93a,Oli02,KohOli03a},  metric fixed point theory \cite{Koh01,Koh01a,Koh03,KohLeu03,KohLam04,Koh05a,Ger06,Bri07,Bri07a,Bri07b,KohLeu07,Leu07,Leu07a,Bri08,KohLeu08a,Leu08,Bri08a,Bri08b}, as well as ergodic theory and topological dynamics \cite{AviGerTow08,Ger08a,KohLeu08a,Ger08b}.

Moreover, in the context of these applications, general logical metatheorems were developed by Kohlenbach \cite{Koh05,Koh06}  and Gerhardy-Kohlenbach \cite{GerKoh08}, having the following form:  If certain $\forall\,\exists$-sentences are proved in some formal systems associated to abstract structures $X$ (e.g.  metric, (uniformly convex) normed, Hilbert, $CAT(0)$ or $W$-hyperbolic spaces), then from a given proof one can extract an effective bound which holds in arbitrary such spaces and is uniform for all parameters meeting very weak local boundedness conditions. Adaptations of these metatheorems to other structures ($\R$-trees, Gromov hyperbolic spaces, uniformly convex $W$-hyperbolic spaces) are given in \cite{Leu06}. The importance of the metatheorems is that they can be used to infer new uniform existence results without having to carry out any actual proof analysis.  
The metatheorems apply to formal systems and thus to formalized proofs, they guarantee the extractability of additional information based on a transformation of these formalized proofs. In practice, it is in general not necessary to completely formalize a mathematical proof in order to analyze it. Applications of proof mining often consist of preprocessing the original mathematical proof by putting the statement and the main concepts involved into a suitable logical form and then identifying the key steps in the proof that need to be given a computational interpretation.  As a result, we get direct proofs  for the explicit quantitative versions of the original results, i.e. proofs that no longer rely on any logical tools.
  
Naturally, there are limits to what can be achieved with proof mining. Let us consider the Cauchy property of bounded monotone sequences $(a_n)$ in $\R$, which is a statement of a more complicated  $\forall\,\exists\forall$ logical form:
\beq
\forall\eps>0\exists N\in\N\forall p\in\N\big(|a_{N+p}-a_N|<\eps\big). \label{intro-Cauchy}
\eeq

\noindent By a well-known construction of Specker \cite{Spe49}, there are easily computable sequences already in $[0,1]\cap\Q$ without any computable bound on the $\exists N$, that is which have no computable Cauchy modulus. 
Nevertheless, as we shall see in Section \ref{proof-mining-fpt}, the logical metatheorems guarantee effective uniform bounds on the so-called  {\em Herbrand normal form} of the Cauchy property, that (ineffectively) is equivalent with (\ref{intro-Cauchy}).

The proofs of the logical metatheorems are based on extensions to the new theories of two proof interpretations developed by G\" odel: {\em functional} (or {\em Dialectica}) {\em interpretation} \cite{God58} and {\em double-negation} interpretation \cite{God33}. 

In applications of proof mining, Kohlenbach's {\em monotone} functional interpretation (see \cite{Koh96a} or \cite[Chapter 9]{Koh08-book} for details) is crucially used, since it systematically transforms any statement in a given proof into a new  version for which explicit bounds are provided. As it is argued in \cite{KohOli03}, monotone functional interpretation provides in many cases the {\em right} notion of {\em numerical implication} in analysis.

Recently, Terence Tao \cite{Tao07} arrived at a proposal of so-called {\em hard analysis} (as opposed to {\em soft analysis}), inspired by the finitary arguments used recently by him and Green \cite{GreTao08} in their proof that there are arithmetic progressions of arbitrary length in the prime numbers, as well as by him alone in a series of papers \cite{Tao06,Tao07a,Tao08,Tao08a,Tao08b}. In the essay \cite{Tao07}, Tao illustrates his ideas using two examples: a {\em finite convergence} principle and a {\em finitary} infinite pigeonhole principle. It turns out that both the former and a variant of the latter directly result from monotone functional interpretation \cite{Koh07,GasKoh09}. Hence, Tao's hard analysis could be roughly understood as carrying out analysis on the level of 
uniform bounds in the sense of monotone functional interpretation which in many cases allows one to {\em finitize} analytic assumptions and to arrive at qualitatively stronger results.

\section{Some topics in fixed point theory of nonexpansive mappings}\label{intro-fpt}

In the following we review some topics related with the applications we shall present in Chapter \ref{proof-mining-fpt}. We refer to  \cite[Chapter 3]{KirSim01} or to \cite{GoeKir90,GoeRei84} for a comprehensive treatment of the fixed point theory of nonexpansive mappings.

The notion of nonexpansive mapping can be introduced in the very general setting of metric spaces. Thus, if $(X,d)$ is a metric space, and $C\subseteq X$ a nonempty subset, then a mapping $T:C\to C$ is said to be {\em nonexpansive} if for all $x,y\in C$, 
\[d(Tx, Ty)\le d(x,y).\]

We shall denote with $Fix(T)$ the set of fixed points of $T$.  The subset $C$ is said to have the {\em fixed point property for nonexpansive mappings}, FPP for short, if  $Fix(T)\ne\emptyset$ for any nonexpansive mapping $T:C\to C$.

While an abstract metric space is all that is needed to define the concept of nonexpansive mapping, the most interesting results were obtained in the setting of Banach spaces.

Fixed point theory of contractions is, even from a computational point of view, essentially trivial, due to Banach's Contraction Mapping Principle. Anyway, the picture known for contractions breaks down for nonexpansive mappings, as we indicate below:
\be
\item Nonexpansive mappings need not to have fixed points: just take $T:\R\to\R, \,\,T(x)=x+1$.
\item Even when $C$ is compact (and therefore fixed points exist by the fixed point theorems of Brouwer and Schauder), they are not unique:  take $T:\R\to\R, \,\,T(x)=x$.
\item Even when the fixed point is unique, it will in general not be approximated by the Picard iteration $x_{n+1}=Tx_n$:  if we let $T:[0,1]\to[0,1], \,T(x)=1-x$ and $x_0=0$, then $T$ has a unique fixed point $\frac12$, while $x_n$ alternates between $0$ and $1$.
\ee

Fixed point theory for nonexpansive mappings has been a very active research area in nonlinear analysis beginning with the 60's, when the most widely known result in the theory, the so-called Browder-G\" ohde-Kirk Theorem, was published. 

\bthm\label{BGK-thm}
If $C$ is a bounded closed and convex subset of a uniformly convex Banach space $X$ and $T:C\to C$ is nonexpansive, then $T$ has a fixed point.
\ethm

The above theorem was proved independently by Browder \cite{Bro65} and G\" ohde \cite{Goh65} in the form stated above, and by Kirk \cite{Kir65} in a more general form. Browder and Kirk used the same line of argument, which in fact yields a more general result - while the proof of G\" ohde relies on properties essentially unique to uniformly convex Banach spaces.

\subsection{The approximate fixed point property}

Let $(X,d)$ be a metric space, $C\se X$ and $T:C\to C$. The {\em minimal displacement of } $T$ is defined as
\beq 
r_C(T):=\inf\{d(x,Tx)\mid x\in C\}. \label{def-minimal-disp}
\eeq
A sequence $(x_n)$ in $C$ is called an {\em approximate fixed point sequence} of $T$ if $\ds\limn d(x_n,Tx_n)=0$.  We say that $T$ is {\em approximately fixed} \cite{BorReiSha92}, or that $T$ {\em has approximate fixed points}, if $T$ has an approximate fixed point sequence. 

Given $\varepsilon>0$, a point $x\in C$ is said to be an {\em $\varepsilon$-fixed point} of $T$ if $d(x,Tx)<\eps$. We shall denote with $Fix_\eps(T)$ the set of $\eps$-fixed points of $T$.

It is easy to see that $T$ is approximately fixed if and only if $r_C(T)=0$ if and only if $Fix_\eps(T)\ne\emptyset$ for any $\eps>0$.

A related notion is the following. For $x\in C$ and $b,\eps>0$, let us denote
\[Fix_{\varepsilon}(T,x,b):=\{y\in C\mid d(y,x)\leq b \text{~and~} d(y,Ty)<\eps\}.\]
If $Fix_{\varepsilon}(T,x,b)\ne \emptyset$ for all $\eps>0$, we say that $T$ {\em has approximate fixed points} in a $b$-neighborhood of $x$. 

\blem\label{char-AFP-b-neighborhood}
The following are equivalent.
\be
\item $T$ has a bounded approximate fixed point sequence;
\item for all $x\in C$ there exists  $b>0$ such that $T$ has approximate fixed points in a $b$-neighborhood of $x$;
\item there exist $x\in C$ and $b>0$ such that $T$ has approximate fixed points in a $b$-neighborhood of $x$.
\ee 
\elem

A subset $C$ of  a metric space $(X,d)$ is said to have the {\em approximate fixed point property for nonexpansive mappings}, AFPP for short, if each nonexpansive mapping $T:C\to C$ is approximately fixed. It is well-known that {\em bounded} closed convex subsets of Banach spaces have the AFPP for nonexpansive mappings (see, for example, \cite[Chapter 3, Lemma 2.4]{KirSim01}. 

Goebel and Kuczumow \cite{GoeKuc78} were the first to remark that there exist {\em unbounded} sets in Hilbert spaces that have this property. Namely, they proved that any closed convex set $C$ contained in a block has the AFPP; a set $K\se \ell_2$ is a {\em block} if $K$ is of the form $K=\{x\in \ell_2\mid |<x,e_n>|\le M_n, n=1,2,\ldots, \}$, where $\{e_n\}$ is some orthogonal basis and $(M_n)$ is a sequence of positive reals. More recently, Kuczumow  gave in \cite{Kuc03} an example of an unbounded closed convex subset of $\ell_2$ that has the AFPP, but it is not contained in a block for any orthogonal basis of $\ell_2$.

Goebel and Kuczumow' result was extended by Ray \cite{Ray78} to include all linearly bounded subsets of $\ell_p$, $1<p<\infty$. A subset $C$ of a normed space $X$ is said to be {\em linearly bounded} if it has bounded intersection with all lines in $X$.
Subsequently, Ray obtained  the following characterization of the FPP in Hilbert spaces, answering an open problem of Kirk.

\bthm\cite{Ray80}
A closed convex subset of a real Hilbert space has the FPP for nonexpansive mappings if and only if it is bounded.
\ethm
In \cite{Rei83}, Reich proved the following remarkable theorem.

\bthm\cite{Rei83}\label{Reich-lin-bounded}
A closed convex subset of a reflexive Banach space has the AFPP for nonexpansive mappings if and only if it is linearly bounded.
\ethm   
If the Banach space $X$ is finite-dimensional, then any linearly bounded subset $C$ of $X$ is, in fact, bounded. Thus, in this case, either $C$ is bounded and has the FPP, or $C$ is unbounded and does not even have the AFPP for nonexpansive mappings.

As it was already noted in \cite{Rei83}, the above theorem can not be extended to all Banach spaces: just take $X=\ell_1$, $C=\{x\in \ell_1\mid |x_n|\le 1 \text{ for all } n\}$ and define $T:C\to C$ by $T(x_1,x_2,\ldots)=(1,x_2.x_3,\ldots)$. Then $C$ is linearly bounded and $T$ is an isometry, but $r_C(T)=1$, hence $T$ is not approximately fixed.

In \cite{Sha91}, Shafrir gave a more general geometric characterization of the AFPP which is true in an arbitrary Banach space or even for the more general class of complete hyperbolic  spaces in the sense of \cite{ReiSha90}. In order to do this, he introduced the concept of a {\em directionally bounded set}.

A {\em directional curve} in a metric space $(X,d)$ is a curve $\gamma:[0,\infty)\to X$ for which there exists $b>0$ such that for each $t\geq s\geq 0$,
\[t-s-b\leq d(\gamma(s),\gamma(t))\leq t-s.\]
A convex subset of a Banach space is called {\em directionally bounded} if it contains no directional curve. Since a line is a directional curve with $b=0$, directionally bounded sets are always linearly bounded. Shafrir proved two important results.

\bthm\cite{Sha91}
\be
\item A convex subset of a Banach space has the AFPP if and only if it is directionally bounded.
\item A Banach space $X$ is reflexive if and only if every closed convex linearly bounded subset of $X$ is directionally bounded.
\ee
\ethm
Therefore, the characterization for the AFPP from Reich Theorem \ref {Reich-lin-bounded} is true for a Banach space $X$ if and only if $X$ is reflexive.
 
Answering an open question of Shafrir \cite{Sha91}, in \cite{MatRei03} Matou\v skov\' a and Reich showed that any infinite-dimensional Banach space contains an unbounded convex subset which has the AFPP for nonexpansive mappings; 
Shafrir \cite{Sha91} had proved this only for infinite-dimensional Banach spaces which do not contain an isomorphic copy of $\ell_1$.

\subsection{Krasnoselski-Mann iterations}\label{intro-fpt-KM}

A fundamental theorem in the fixed point theory of nonexpansive mappings is the following result due to Krasnoselski, which shows that, under an additional compactness condition, a fixed point of $T$ can be approximated by a special iteration technique.

\bthm \label{Krasnoselski-thm}\cite{Kra55}
Let $C$ be a closed convex subset of a uniformly convex Banach space $X$, $T$ be a nonexpansive mapping, and suppose that $T(C)$ is contained into a compact subset of $C$. Then for every $x\in C$, the sequence $(x_n)$ defined by
\beq
x_0:=x, \quad x_{n+1}:=\frac12(x_n+Tx_n)\label{def-K-iteration}
\eeq
converges to a fixed point of $T$.
\ethm
\noindent  Schaefer \cite{Sch57} remarked that Krasnoselski Theorem holds for iterations of the form 
\beq
x_0:=x, \quad x_{n+1}:=(1-\lambda)x_n+\lambda Tx_n, \label{def-K-iteration-lambda}
\eeq
 where $\lambda \in (0,1)$. Moreover, Edelstein \cite{Ede70} proved that strict convexity of $X$ suffices. The iteration (\ref{def-K-iteration-lambda}) is today known as the {\em Krasnoselski iteration}. 

For any $\lambda\in(0,1)$, the {\em averaged mapping} $T_\lambda$ is defined by
\[T_\lambda:C\to C, \quad T_\lambda(x)=(1-\lambda)x + \lambda Tx.\]
It is easy to see that $T_\lambda$ is also nonexpansive and that $Fix(T)=Fix(T_\lambda)$. Moreover, the Krasnoselski iteration $(x_n)$ starting with $x\in C$ is the Picard iteration $\big(T_\lambda^n(x)\big)$ of $T_\lambda$. 

A vast extension  of Krasnoselski Theorem was obtained by Ishikawa in his seminal paper \cite{Ish76}.  He showed that Krasnoselski Theorem holds without the assumption of $X$ being uniformly convex and for much more general iterations, defined as follows:
\begin{equation}
x_0:=x, \quad x_{n+1}:=(1-\lambda_n)x_n +\lambda_n Tx_n, \label{def-KM-iteration}
\end{equation}
where $(\lambda_n)$ is a sequence in $[0,1]$ and $x\in C$ is the starting point. This iteration is a special case of the generalized iteration method introduced by Mann \cite{Man53}. Following \cite{BorReiSha92}, we call the iteration (\ref{def-KM-iteration}) the {\em Krasnoselski-Mann iteration}. We remark  that it is often said to be a {\em segmenting Mann iteration} \cite{OutGro69,Gro72,HicKub77}.

\bthm \label{Ishikawa-thm}\cite{Ish76}
Let $C$ be a closed convex subset of a Banach space $X$, $T$ be a nonexpansive mapping, and suppose that $T(C)$ is contained into a compact subset of $C$. Assume that $(\lambda_n)$ is a sequence in $[0,1]$, divergent in sum and bounded away from $1$. 

Then for every $x\in C$, the Krasnoselski-Mann iteration converges to a fixed point of $T$.
\ethm 

\noindent Independently, Edelstein and O'Brien \cite{EdeOBr78} obtained a similar result for constant $\lambda_n=\lambda\in(0,1)$.

The question whether we obtain strong convergence of the Krasnoselski-Mann iterations if the assumption that $T(C)$ is contained into a compact subset of $C$ is exchanged for nicer behavior of $X$ is very natural. The answer to this question is no, and it was given by Genel and Lindenstrauss \cite{GenLin75}. They constructed an example of a bounded closed convex subset $C$ in the Hilbert space $\ell_2$ and a nonexpansive mapping $T:C\to C$ with the property that even the original Krasnoselski iteration (\ref{def-K-iteration}) fails to converge to a fixed point of $T$ for some $x\in C$.

A classical weak convergence result is the following theorem due to Reich \cite{Rei79}. 

\bthm\label{Reich-weak-con-KM-thm}
Let $C$ be a closed  convex subset of a uniformly convex Banach space $X$ with a Fr\' echet differentiable norm and $T:C\to C$ a nonexpansive mapping  with a fixed point.  Assume that $(\lambda_n)$ is a sequence in $[0,1]$ satisfying 
the following condition
\begin{equation}
\sum_{k=0}^\infty \lambda_k(1-\lambda_k)=\infty.\label{intro-hyp-lambda-n-Groetsch}
\end{equation} 
Then for every $x\in C$, the Krasnoselski-Mann iteration converges weakly to a fixed point of $T$.
\ethm

We end this short presentation of Krasnoselski-Mann iterations by emphasizing that a wide variety of iterative procedures used in signal processing and image reconstruction and elsewhere are special cases of the Krasnoselski-Mann iterative procedure, for particular choices of the nonexpansive mapping $T$. We refer to \cite{Byr04,BauBor96} for nice surveys. 

\subsection{Asymptotic regularity}\label{intro-fpt-as-reg}

Asymptotic regularity is a fundamentally important concept in metric fixed-point theory. Asymptotic regularity was already implicit in \cite{Kra55,Sch57,Ede70}, but it was formally introduced by Browder and Petryshyn in \cite{BroPet66}. A mapping $T$ of a metric space $(X,d)$ into itself is said to be asymptotically regular if for all $x\in C$,
\[\displaystyle\lim_{n\to\infty} d(T^n(x),T^{n+1}(x))=0.\]

Let $X$ be a Banach space, $C\se X$ and $T:C\to C$. Then the asymptotic regularity of the averaged mapping  $T_\lambda:=(1-\lambda)I+\lambda T$ is equivalent with  the fact that $\displaystyle\lim_{n\to\infty} \|x_n-Tx_n\|=0$ for all $x\in C$, where $(x_n)$  is the Krasnoselski iteration (\ref{def-K-iteration-lambda}).

Following \cite{BorReiSha92}, we say that the nonexpansive mapping $T$ is {\em $\lambda_n$-asymptotically regular}  (for general $\lambda_n\in[0,1]$) if  for all $x\in C$, 
\[\lim_{n\to\infty} \|x_n-Tx_n\|=0,\]
where $(x_n)$ is the general Krasnoselski-Mann iteration (\ref{def-KM-iteration}). 

The most general assumptions on the sequence $(\lambda_n)$ for which asymptotic regularity has been proved for arbitrary normed spaces are the following, made in Ishikawa's paper \cite{Ish76}: 
\begin{eqnarray}
\sum_{n=0}^\infty \lambda_n=\infty \text{~~and~~}\limsup \lambda_n<1. \label{intro-lambda-n-div-lambda-n-sup}
\end{eqnarray}
Note that if $\lambda_n\in [a,b]$ for all $n\in\N$ and $0<a\leq b<1$, then $(\lambda_n)$ satisfies (\ref{intro-lambda-n-div-lambda-n-sup}).

Ishikawa proved the following result, which was the intermediate step in obtaining Theorem \ref{Ishikawa-thm}.

\begin{theorem}\label{asreg-Ishikawa}\cite{Ish76}
Let $C$ be a convex subset of a Banach space $X$ and $T:C\to C$ be a nonexpansive mapping. Assume that $(\lambda_n)$ satisfies (\ref{intro-lambda-n-div-lambda-n-sup}).
If $(x_n)$ is bounded for some $x\in C$, then $\limn d(x_n,Tx_n)=0$. Thus, if $C$ is bounded, $T$ is $\lambda_n$-asymptotically regular.
\end{theorem}

As observed in \cite{BorReiSha92}, we obtain asymptotic regularity under the weaker assumption that $C$ contains a point $x$ with the property that the Krasnoselski-Mann iteration $(x_n)$ starting with $x$ is bounded. In fact, it is easy to see that if for some $x\in C$, the Krasnoselski-Mann iteration $(x_n)$ starting with $x$ is bounded, then this is true for all $x\in C$.

\begin{theorem}\label{asreg-Ishikawa-bounded-xn}
Let $C$ be a convex subset of a Banach space $X$ and $T:C\rightarrow C$ a nonexpansive mapping. Assume that $(\lambda_n)$ satisfies (\ref{intro-lambda-n-div-lambda-n-sup}) and that $(x_n)$ is bounded for some (each) $x\in C$. 

Then $T$ is $\lambda_n$-asymptotically regular.
\end{theorem}

Using an embedding theorem due to Banach and Mazur \cite{Ban32},  Edelstein and O'Brien \cite{EdeOBr78} also proved the asymptotic regularity for constant $\lambda_n=\lambda\in (0,1)$, and noted that it is uniform for $x\in C$. In \cite{GoeKir83}, Goebel and Kirk  unified Ishikawa's and Edelstein/O'Brien's results, obtaining uniformity with respect to $x$ and to the family of all nonexpansive mappings $T:C\to C$. 

\bthm\label{asreg-uniform-Goebel-Kirk}\cite{GoeKir83}
Let $C$ be a bounded convex subset of a Banach space $X$ and $(\lambda_n)$ satisfying (\ref{intro-lambda-n-div-lambda-n-sup}). Then for every $\eps >0$ there exists a positive integer $N$ such that for all $x\in C$ and all $T:C\to C$ nonexpansive,
\beq
\forall n\geq N \big(\|x_n-Tx_n\| <\eps \big).
\eeq 
\ethm
\noindent We remark that the above theorem was proved in \cite{GoeKir83} for {\em spaces of hyperbolic type};  we refer to Chapter \ref{hyperbolic-spaces} for details on this very general class of spaces.

In 2000, Kirk \cite{Kir00} generalized Theorems \ref{asreg-uniform-Goebel-Kirk} and \ref{asreg-Ishikawa-bounded-xn} to directionally nonexpansive mappings, but only for constant $\lambda_n=\lambda\in(0,1)$.  A mapping $T:C\to C$ is said to be {\em directionally nonexpansive} if $\|Tx-Ty\|\le \|x-y\|$ for all $x\in C$ and all $y\in seg[x,\,Tx]$.

\bthm\label{asreg-uniform-Kirk-dir-ne}\cite{Kir00}
Let $C$ be a convex subset of a Banach space $X$, $T:C\to C$ be directionally nonexpansive and $\lambda \in (0,1)$.
\be
\item If $(x_n)$ is bounded for each $x\in C$, then the averaged mapping $T_\lambda$ is asymptotically regular.
\item If $C$ is bounded, then for every $\eps >0$ there exists $N\in\N$ such that for all $x\in C$ and all $T:C\to C$ directionally nonexpansive,
\beq
\forall n\geq N \big(\|T_\lambda^{n+1}(x)-T_\lambda^n(x)\| < \eps \big).
\eeq 
\ee
\ethm

A very important result is the following theorem due to  Borwein, Reich and Shafrir,  extending Ishikawa Theorem \ref{asreg-Ishikawa-bounded-xn} to unbounded $C$.

\bthm \label{BRS-thm-normed}\cite{BorReiSha92}
Let $C$ be a closed convex subset of a Banach space $X$ and $T:C\rightarrow C$ a nonexpansive mapping. Assume that $(\lambda_n)$ satisfies (\ref{intro-lambda-n-div-lambda-n-sup}).
Then for all $x\in C$, 
\beq
\limn \|x_n-Tx_n\|=r_C(T),
\eeq
where $r_C(T)$ is the minimal displacement of T, defined by (\ref{def-minimal-disp}).
\ethm
Thus, convergence of $(\|x_n-Tx_n\|)$ towards $r_C(T)$ is obtained for $(\lambda_n)$ divergent in sum and bounded away from $1$, while in \cite{ReiSha87, ReiSha90} $(\lambda_n)$ was required also to be bounded away from $0$. In this way, the case of Cesaro and other summability methods is covered \cite{Dot70,Gro72,Man53}. 

As an immediate consequence of  Borwein-Reich-Shafrir Theorem, we get that any approximately fixed nonexpansive mapping is $\lambda_n$-asymptotically regular for $(\lambda_n)$ satisfying (\ref{intro-lambda-n-div-lambda-n-sup}). 

A  straightforward application of Theorems \ref{asreg-Ishikawa-bounded-xn} and \ref{BRS-thm-normed} is the fact that $r_C(T)=0$  whenever $(x_n)$ is bounded for some (each) $x\in C$, in particular for bounded $C$. Let us remark that for unbounded $C$, $r_C(T)$ can be very well strict positive: for example, if $T:\R\to\R, \,Tx=x+1$, then $r_\R(T)=1$ although $T$ is nonexpansive.

In \cite{BaiBru96}, it is conjectured that Ishikawa's Theorem \ref{asreg-Ishikawa} holds true if  (\ref{intro-lambda-n-div-lambda-n-sup}) is replaced by the weaker condition (\ref{intro-hyp-lambda-n-Groetsch}), which is symmetric in $\lambda_n, 1-\lambda_n$. For the case of uniformly convex Banach spaces, this has been proved by Groetsch \cite{Gro72} (see also \cite{Rei79}).

\begin{theorem}\label{intro-Groetsch-thm}
Let $C$ be a convex subset of a uniformly convex Banach space and $T:C\to C$ be a nonexpansive mapping such that $T$ has at least one fixed point. Assume that $(\lambda_n)$ satisfies the following condition:
\begin{equation}
\sum_{k=0}^\infty \lambda_k(1-\lambda_k)=\infty.
\end{equation} 
Then $T$ is $\lambda_n$-asymptotically regular.
\end{theorem}

\subsection{Ishikawa iterations}

Let $C$ be a convex subset of a normed space $X$ and  $T:C\to C$ be nonexpansive.

The {\em Ishikawa iteration} \cite{Ish74} starting with $x\in C$ is defined by 
\beq
x_0:=x, \quad x_{n+1}=(1-\lambda_n)x_n+\lambda_nT\big((1-s_n)x_n+ s_nTx_n\big),
\eeq
where $(\lambda_n),(s_n)$ are sequences in $[0,1]$. By letting $s_n=0$ for all $n\in\N$, we get the Krasnoselski-Mann iteration as  a special case.

An extension of Ishikawa Theorems \ref{asreg-Ishikawa} and \ref{Ishikawa-thm} to these iterations  was proved by Deng \cite{Den96}.

\bthm\cite{Den96}
Let $C$ be a convex subset of a Banach space $X$ and $T:C\to C$ be a nonexpansive mapping. Assume that $(\lambda_n)$ satisfies (\ref{intro-lambda-n-div-lambda-n-sup}) and that $\ds\sum_{n=0}^\infty s_n$ converges.
\be
\item If $(x_n)$ is bounded for some $x\in C$, then $\limn d(x_n,Tx_n)=0$. 
\item  Assume furthermore that $C$ is closed and $T(C)$ is contained into a compact subset of $C$. Then $(x_n)$ converges to a fixed point of $T$.
\ee
\ethm

Tan and Xu \cite{TanXu93} obtained a weak convergence result for Ishikawa iterates  that generalizes Reich Theorem \ref{Reich-weak-con-KM-thm}.

\bthm\label{weak-con-Ishikawa-thm}
Let $C$ be a bounded closed  convex subset of a uniformly convex Banach space $X$ which satisfies Opial's condition or has a Fr\' echet differentiable norm and $T:C\to C$ be a nonexpansive mapping.  Assume that $(\lambda_n), (s_n)$ satisfy 
\beq \sum_{n=0}^\infty\lambda_n(1-\lambda_n) \text{ diverges},\quad \limsup_n s_n<1 \quad \text{ and } \sum_{n=0}^\infty s_n(1-\lambda_n)\text{  converges.}\label{hyp-Ishikawa-TanXu}
\eeq
\be
\item For every $x\in C$, the Ishikawa iteration $(x_n)$ converges weakly to a fixed point of $T$.
\item If, moreover, $T(C)$ is contained into a compact subset of $C$, then the convergence is strong.
\ee
\ethm

As in the case of Krasnoselski-Mann iterations, the  first step towards getting weak or strong convergence is proving asymptotic regularity (with respect to Ishikawa iterates), and this was done by Tan and Xu \cite{TanXu93} for uniformly convex Banach spaces and, recently,  by Dhompongsa and Panyanak \cite{DhoPan08} for $CAT(0)$ spaces.
 
\bthm\label{intro-Ishikawa-as-reg}
Let $X$ be a uniformly convex Banach space or a $CAT(0)$ space, $C\se X$ a  bounded closed convex subset and $T:C\to C$   nonexpansive.  Assume that $(\lambda_n), (s_n)$ satisfy  (\ref{hyp-Ishikawa-TanXu}).

Then $\ds\limn \|x_n-Tx_n\|=0$ for every $x\in C$. 
\ethm

\subsection{Halpern iterations}

Let $C$ be a convex subset of a normed space $X$ and  $T:C\to C$   nonexpansive. The {\em Halpern iteration} was introduced in  \cite{Hal67} as follows: 
\beq
x_0:=x, \quad x_{n+1}:=\lambda_{n+1}x+(1-\lambda_{n+1})Tx_n,
\label{intro-def-Halpern-iterate}
\eeq
where $(\lambda_n)_{n\ge 1}$ is a sequence in $[0,1]$ and $x\in C$ is the starting point.

\bfact \cite{Wit91,Wit92}\, If $T$ is positively homogeneous (i.e. $T(tx)=tT(x)$ for all $t\geq 0$ and all $x\in C$), then
\beq
x_n=\frac{1}{n+1}\,S_nx \label{xn-1-n+1-T-pos-homogeneous}, \quad\text{where}\quad S_0x=x, \,\,\,S_{n+1}x=x+T(S_nx).
\eeq
Furthermore, if $T$ is linear, then $\ds x_n=\frac{1}{n+1}\ds\sum_{i=0}^nT^ix$, so the Halpern iterations could be regarded as nonlinear generalizations of the usual Cesaro averages. We refer to \cite{Wit91,LinWit91} for a a systematic study of the behavior of iterations given by (\ref{xn-1-n+1-T-pos-homogeneous}).
\efact

In \cite{Rei83}, Reich formulated the following problem:

\begin{problem}\cite[Problem 6]{Rei83}\label{problem-Reich}\\
Let $X$ be a Banach space. Is there a sequence $(\lambda_n)$ such that whenever a weakly compact convex subset $C$ of $X$ possesses the FPP for nonexpansive mappings, then $(x_n)$ converges to a fixed point of $T$ for all $x\in C$ and all nonexpansive mappings $T:C\to C$ ?
\end{problem}

Let us consider the following conditions on $(\lambda_n)$.
\[\ba{lll} 
(C1) \quad  \lim \lambda_n=0, & (C2) \quad \ds\sum_{n=1}^\infty\lambda_n=\infty, & (C3)  \quad \ds\sum_{n=1}^\infty|\lambda_{n+1}-\lambda_n|=\infty,\\
(C4) \quad \ds\limn\frac{\lambda_n-\lambda_{n+1}}{\lambda_{n+1}^2}=0, & (C5) \quad  \ds\limn\frac{\lambda_n-\lambda_{n+1}}{\lambda_{n+1}}=0.
\ea
\]
In (C4) and (C5) we assume moreover that $\lambda_n>0$ for all $n\ge 1$.

The study of the convergence of the scheme (\ref{intro-def-Halpern-iterate}) in the Hilbert space setting was initiated by Halpern \cite{Hal67}, who proved  that $(x_n)$ converges to a fixed point of $T$ for $(\lambda_n)$ satisfying certain conditions, two of which are (C1) and (C2). P.-L. Lions \cite{Lio77} improved Halpern's result by showing the convergence of $(x_n)$ if $(\lambda_n)$ satisfies (C1), (C2) and (C4). However, both Halpern's and Lions' conditions exclude the natural choice $\ds\lambda_n=\frac1{n+1}$. 

This was overcome by Wittmann \cite{Wit92}, who obtained one of the most important results on the convergence of Halpern iterations in Hilbert spaces.

\begin{theorem}\label{habil-wittmann-thm}\cite{Wit92}
Let $C$  be a closed convex subset of a Hilbert space $X$ and $T:C\to C$ a nonexpansive mapping such that the set $Fix(T)$ of fixed points of $T$ is nonempty. Assume that $(\lambda_n)$ satisfies  (C1), (C2) and (C3).
Then for any $x\in C$, the Halpern iteration $(x_n)$ converges to the projection $Px$ of $x$ on $Fix(T)$.
\end{theorem}

Thus, all the above partial answers to Reich's problem require that the sequence $(\lambda_n)$ satisfies (C1) and (C2).  Halpern \cite{Hal67} showed in fact that  conditions (C1) and (C2) are necessary in the sense that if, for every closed convex subset $C$ of a Hilbert space $X$ and every nonexpansive mappings $T:C\to C$ such that $Fix(T)\ne\emptyset$, the Halpern iteration $(x_n)$ converges to a fixed point of $T$, then $(\lambda_n)$ must satisfy (C1) and (C2). It however remains an open question whether (C1) and (C2) are sufficient to guarantee the convergence of $(x_n)$. Recently, Chidume and Chidume \cite{ChiChi06} and Suzuki \cite{Suz07} proved that if  the nonexpansive mapping $T$ in (\ref{intro-def-Halpern-iterate}) is averaged, then (C1) and (C2) suffice for obtaining the convergence of $(x_n)$.

Halpern derived  his result as a consequence of a limit theorem for the resolvent, first shown by Browder \cite{Bro67}. This approach has the advantage that this result can be immediately generalized, once the limit theorem for the resolvent has been generalized. This was done by Reich \cite{Rei80}.

\bthm \cite{Rei80} \label{Reich-theorem-Halpern}
Let $C$ be a closed convex subset of a uniformly smooth Banach space $X$, and let $T:C\to C$ be nonexpansive such that $Fix(T)\ne\emptyset$. For each $y\in C$ and $t\in (0,1)$, let $y_t$ denote the unique fixed point of the contraction mapping 
$$T_t(\cdot)=(1-t)y+tT(\cdot).$$
Then $\ds\lim_{t\to 1^-}y_t$ exists and is a fixed point of $T$.
\ethm
A similar result was obtained recently by Kirk \cite{Kir03} for $CAT(0)$ spaces. 
As a consequence of Theorem \ref{Reich-theorem-Halpern}, a partial positive answer to Problem \ref{problem-Reich} was obtained \cite{Rei80} for uniformly smooth Banach spaces and $\ds\lambda_n=\frac1{(n+1)^\alpha}$ with $0<\alpha < 1$. Furthermore, Reich \cite{Rei94} proved the strong convergence of $(x_n)$ in the setting of uniformly smooth Banach spaces that have a weakly sequentially continuous duality mapping for general $(\lambda_n)$ satisfying (C1), (C2) and being decreasing (and hence (C4) holds). Another partial answer in the case of uniformly smooth Banach spaces was obtained by Xu \cite{Xu02,Xu02a} for $(\lambda_n)$ satisfying (C1), (C2) and (C5) (which is weaker than Lions' (C4)). In \cite{ShiTak97}, Shioji and Takahashi extended Wittmann's result  to Banach spaces with uniformly G\^ ateaux differentiable norm and  with the property that $\ds\lim_{t\to 1^-}y_t$ exists and is a fixed point of $T$.

We end this section with the following remark. By inspecting the proof of Theorem \ref{habil-wittmann-thm}, it is easy to see that the first step is to obtain asymptotic regularity (i.e. $\ds \limn \|x_n-Tx_n\|=0$) and this can be done in  a much more general setting.  Thus, the following theorem is essentially contained in \cite{Wit92,Xu02,Xu04} and will be the point of departure for our application in Section \ref{habil-app-fpt-Halpern}.

\begin{theorem}\label{intro-Halpern-ass-reg}
Let $C$ be a convex subset of a normed space $X$ and  $T:C\to C$   nonexpansive. Assume that $(\lambda_n)_{n\geq 1}$ is a sequence in $[0,1]$  satisfying (C1), (C2) and (C3).

Then $\ds\limn \|x_n-Tx_n\|=0$ for every $x\in C$  with the property that $(x_n)$ is bounded.
\end{theorem}

\subsection{Asymptotically nonexpansive mappings}\label{intro-as-ne}

Asymptotically nonexpansive mappings were introduced  by Goebel and Kirk \cite{GoeKir72} as a generalization of the nonexpansive ones. A mapping $T:C\to C$ is said to be {\em asymptotically nonexpansive with sequence}  $(k_n)$ in $[0,\infty)$ if 
$\lim\limits_{n\to\infty} k_n =0$ and 
\[  d(T^n x,T^ny) \leq (1+k_n)d(x,y), \ \hfill \text{forall }n\in\N \text{ and all }x,y\in C.
 \] 
It is obvious that an asymptotically nonexpansive mapping with sequence $(k_n)$ is $(1+k_1)$-Lipschitz. Examples  showing that the class of asymptotically nonexpansive mappings is wider than the class of nonexpansive mappings are given in  \cite{GoeKir72,KirMarShi98}. 

Goebel and Kirk \cite{GoeKir72} extended the Browder-G\" ohde-Kirk Theorem to this class of mappings. 

\bthm\cite{GoeKir72} \label{habil-as-ne-FPP-ucBanach}
Bounded closed convex subsets of uniformly convex Banach spaces have the FPP for asymptotically nonexpansive mappings.
\ethm

Recently \cite{Kir04}, Kirk proved the same result for $CAT(0)$ spaces.

\bthm\cite{Kir04} \label{habil-as-ne-FPP-CAT0}
Bounded closed convex subsets of complete $CAT(0)$ spaces have the FPP for asymptotically nonexpansive mappings.
\ethm

\noindent Kirk proved Theorem \ref{habil-as-ne-FPP-CAT0} using nonstandard methods, inspired by Khamsi's proof that bounded hyperconvex metric spaces have the AFPP for asymptotically nonexpansive mappings \cite{Kha03}.

For asymptotically nonexpansive mappings, the {\em Krasnoselski-Mann iteration} starting from $x\in C$ is defined by 
\begin{equation}
x_0:=x, \quad x_{n+1}:=(1-\lambda_n)x_n +\lambda_n T^nx_n, \label{habil-as-ne-KM-lambda-n-def-hyp}\end{equation}
where $(\lambda_n)$ is a sequence in $[0,1]$. The above iteration was introduced by Schu \cite{Sch91}; it is called {\em modified Mann iteration} in \cite{TanXu94}. 

Asymptotically nonexpansive mappings have been studied mostly in the context of uniformly convex Banach spaces. In fact, for general Banach spaces it is not known whether bounded closed convex subsets have the AFPP (see \cite{KirMarShi98} for a discussion).

In the setting of uniformly convex Banach spaces, the following weak convergence result was proved by Schu \cite{Sch91a}  with the assumption that  Opial's condition is satisfied and by Tan and Xu \cite{TanXu94} in the hypothesis that the space has a Fr\' echet differentiable norm.

\bthm\cite{Sch91a,TanXu94}\label{intro-as-ne-weak}
Let $X$ be a uniformly convex Banach space which  satisfies Opial's condition or has a Fr\' echet differentiable norm,  $C$ be a bounded closed convex subset of $X$ and $T:C\to C$   an asymptotically nonexpansive mapping with sequence $(k_n)$ satisfying $\ds\sum\limits^{\infty}_{i=0} k_i<\infty$. 
Assume  that $(\lambda_n)$ is bounded away from $0$ and $1$. 

Then for all $x\in C$,  the Krasnoselski-Mann iteration $(x_n)$ starting with $x$ converges weakly to a fixed point of $T$.
\ethm

As in the case of nonexpansive mappings, if  $\ds\limn d(x_n,Tx_n)=0$ for all $x\in C$, $T$ is said to be {\em $\lambda_n$-asymptotically regular}. The following asymptotic regularity result is  essentially contained in \cite{Sch91,Sch91a}.

\begin{theorem}\label{intro-as-ne-as-reg}
Let $C$ be a convex subset of a uniformly convex Banach space $X$   and $T:C\to C$ an asymptotically nonexpansive mapping with sequence $(k_n)$ in $[0,\infty )$ satisfying 
$\ds \sum_{i=0}^{\infty} k_i <\infty$.  Let $(\lambda_n)$ be a sequence in $[a,b]$ for $0<a<b<1$.

If $T$ has a fixed point, then $T$ is $\lambda_n$-asymptotically regular.
\end{theorem}

\section{$W$-hyperbolic spaces}\label{hyperbolic-spaces}

We work in the setting of hyperbolic spaces as introduced by Kohlenbach \cite{Koh05}. In order to distinguish them from Gromov hyperbolic spaces \cite{BriHae99} or from other notions of hyperbolic space  that can be found in the literature (see for example \cite{Kir82,GoeKir83,ReiSha90}), we shall call them W-hyperbolic spaces. 

A {\em  $W$-hyperbolic space}  $(X,d,W)$ is a metric space $(X,d)$ together with a {\em convexity mapping} $W:X\times X\times [0,1]\to X$ satisfying 
\begin{eqnarray*}
(W1) & d(z,W(x,y,\lambda))\le (1-\lambda)d(z,x)+\lambda d(z,y),\\
(W2) & d(W(x,y,\lambda),W(x,y,\tilde{\lambda}))=|\lambda-\tilde{\lambda}|\cdot 
d(x,y),\\
(W3) & W(x,y,\lambda)=W(y,x,1-\lambda),\\
(W4) & \,\,\,d(W(x,z,\lambda),W(y,w,\lambda)) \le (1-\lambda)d(x,y)+\lambda
d(z,w).
\end {eqnarray*}

The convexity mapping $W$ was first considered by Takahashi in \cite{Tak70}, where a triple $(X,d,W)$ satisfying $(W1)$ is called a {\em convex metric space}. If $(X,d,W)$ satisfies $(W1)-(W3)$, then we get the notion of {\em space of hyperbolic type} in the sense of Goebel and Kirk \cite{GoeKir83}. $(W4)$ was already considered by Itoh \cite{Ito79} under the name "condition III" and it is used by Reich and Shafrir \cite{ReiSha90} and Kirk \cite{Kir82} to define their notions of hyperbolic space. We refer to \cite[p.384-387]{Koh08-book} for a detailed discussion.

Obviously, any normed space is a $W$-hyperbolic space: just define $W(x,y,\lambda)=(1-\lambda)x+\lambda y$. Furthermore, any convex subset of a normed space is a $W$-hyperbolic space. We shall see in Subsection \ref{W-related-classes} other examples of $W$-hyperbolic spaces. 

Let $(X,d,W)$ be a  $W$-hyperbolic space. If $x,y\in X$ and $\lambda\in[0,1]$, then we use the notation $(1-\lambda)x\oplus \lambda y$ for $W(x,y,\lambda)$. The following holds even for the more general setting of convex metric spaces \cite{Tak70}: for all $x,y\in X$ and  $\lambda\in[0,1]$,
\beq
d(x,\lambdaxy)=\lambda d(x,y)\quad \text{~and~}\quad  d(y,\lambdaxy)=(1-\lambda)d(x,y). \label{habil-prop-xylambda}
\eeq
As an immediate consequence, $1x\oplus 0y=x,\,0x\oplus 1y=y$ and $(1-\lambda)x\oplus \lambda x=\lambda x\oplus (1-\lambda)x=x$. 

Following \cite{Tak70}, we call a $W$-hyperbolic space {\em strictly convex} if for any $x\ne y\in X$ and any $\lambda\in(0,1)$ there exists a unique element $z\in X$ (namely $z=\lambdaxy$) such that 
\[d(x,z)=\lambda d(x,y)\quad \text{~and~}\quad  d(y,z)=(1-\lambda)d(x,y).\]

The following definitions can be given in an arbitrary metric space $(X,d)$. A {\em geodesic path}, {\em geodesic} for short, in $X$ is a map $\gamma:[a,b]\to X$ which is distance-preserving, that is 
\beq 
d(\gamma(s),\gamma(t))=|s-t| \text{~~for all~~} s,t\in [a,b].
\eeq
A {\em geodesic ray} in $X$ is a distance-preserving map $\gamma:[0,\infty)\to X$ and a {\em geodesic line} in $X$ is a distance-preserving map $\gamma:\R\to X$. A {\em geodesic segment} in $X$ is the image of a geodesic in $X$, while a {\em straight line} in $X$ is the image of a geodesic line in $X$. If $\gamma:[a,b]\to\R$ is a geodesic in $X$, $\gamma(a)=x$ and $\gamma(b)=y$, we say that the geodesic  $\gamma$ {\em joins x and y} or that the geodesic  segment $\gamma([a,b])$ {\em joins x and y}; $x$ and $y$ are also called the {\em endpoints} of $\gamma$. A metric space $(X,d)$ is said to be a {\em (uniquely) geodesic space} if every two points are joined by a (unique) geodesic segment. 

In the following, $(X,d,W)$ is a $W$-hyperbolic space. For all $x,y\in X$, let us denote b
\[[x,y]:=\{(1-\lambda)x\oplus \lambda y\mid \lambda\in[0,1]\}.\]
 Thus, $[x,x]=\{x\}$ and for $x\ne y$, the map
\beq
\gamma_{xy}:[0,d(x,y)]\to\R, \quad \gamma(\alpha)=\left(1-\frac{\alpha}{d(x,y)}\right)x\oplus \frac{\alpha}{d(x,y)}y
\eeq\label{W-def-geodesic}
is a geodesic satisfying $\gamma_{xy}\big([0,d(x,y)]\big)=[x,y]$, so $[x,y]$ is a geodesic segment that joins $x$ and $y$. Hence, any $W$-hyperbolic space is a geodesic space.

A nonempty subset $C\subseteq X$ is {\em convex} if $[x,y]\se C$ for all $x,y\in C$. A nice feature of our setting is that any convex subset is itself a $W$-hyperbolic space with the restriction of $d$ and $W$ to $C$. It is easy to see that open and closed balls are convex and that the intersection of any family of convex sets is again convex. Moreover, using (W4), we get that the closure of a convex subset of a $W$-hyperbolic space is again convex.

If $C$ is a convex subset of $X$, then a function $f:C\to\R$ is said to be {\em convex} if 
\beq 
f\left(\lambdaxy\right)\le (1-\lambda)f(x)+\lambda f(y) \label{habil-convex-f}
\eeq
for all $x,y\in C,\lambda\in[0,1]$. $f$ is said to be {\em strictly convex} if strict inequality holds in (\ref{habil-convex-f}) for $x\ne y$ and $\lambda\in(0,1)$.

\subsection{$UCW$-hyperbolic spaces}\label{habil-UCW}

One of the most important classes of Banach spaces are the uniformly convex ones, introduced by Clarkson in the 30's \cite{Cla36}. Following Goebel and Reich \cite[p. 105]{GoeRei84}, we can define uniform convexity for $W$-hyperbolic spaces too.

A $W$-hyperbolic space $(X,d,W)$ is  {\em uniformly convex} \cite{Leu07} if for any $r>0$ and any $\varepsilon\in(0,2]$ there exists $\delta\in(0,1]$ such that 
for all $a,x,y\in X$,
\begin{eqnarray}
\left.\begin{array}{l}
d(x,a)\le r\\
d(y,a)\le r\\
d(x,y)\ge\varepsilon r
\end{array}
\right\}
& \quad \Rightarrow & \quad d\left(\frac12x\oplus\frac12y,a\right)\le (1-\delta)r. \label{habil-uc-def}
\end{eqnarray}
A mapping $\eta:(0,\infty)\times(0,2]\rightarrow (0,1]$ providing such a
$\delta:=\eta(r,\varepsilon)$ for given $r>0$ and $\varepsilon\in(0,2]$ is called a {\em modulus of uniform convexity}. We call $\eta$ {\em monotone} if it decreases with $r$ (for a fixed $\eps$).

\bprop\cite{Leu07}\\
Any uniformly convex $W$-hyperbolic space is strictly convex.
\eprop

\blem\label{UCW-eta-prop-1}\cite{Leu07,KohLeu08a}\\
Let $(X,d,W)$ be a uniformly convex $W$-hyperbolic space and $\eta$ be a modulus of uniform convexity. Assume that $r>0,\varepsilon\in(0,2], a,x,y\in X$ are such that 
\[d(x,a)\le r,\,\,d(y,a)\le r \text{~and~} d(x,y)\ge\eps r.\]
 Then for any $\lambda\in[0,1]$,
\be
\item\label{UCW-Groetsch-eta} $\ds d(\lambdaxy,a)\le  \big(1-2\lambda(1-\lambda)\eta(r,\varepsilon)\big)r$; 
\item\label{UCW-eta-monotone-eps} for any $\psi\in (0,2]$ such that $\psi\le\eps$, 
\[\ds d(\lambdaxy,a)\le  \big(1-2\lambda(1-\lambda)\eta(r,\psi)\big)r\,;\]
\item  \label{UCW-eta-s-geq-r} for any $s\geq r$, 
\[d(\lambdaxy,a) \le \left(1-2\lambda(1-\lambda)\eta\left(s,\frac{\eps r}{s}\right)\right)s\,;\]
\item\label{UCW-eta-monotone-s-geq-r}if $\eta$ is monotone, then for any $s\geq r$, 
\[d(\lambdaxy,a) \le \left(1-2\lambda(1-\lambda)\eta\left(s,\eps\right)\right)r\,.\]
\ee
\elem

We shall refer to  uniformly convex $W$-hyperbolic spaces with a monotone modulus of uniform convexity as {\em $UCW$-hyperbolic spaces}. We shall see in Subsubsection \ref{habil-CAT0} that $CAT(0)$ spaces  are  $UCW$-hyperbolic spaces with modulus of uniform convexity $\ds\eta(r,\varepsilon)=\varepsilon^2/8$ quadratic in $\eps$. Thus, $UCW$-hyperbolic spaces are a natural generalization of both uniformly convex normed spaces and $CAT(0)$ spaces.

Moreover, as we shall see in the sequel, complete $UCW$-hyperbolic spaces have very nice properties. For the rest of this section, $(X,d,W)$ is a complete $UCW$-hyperbolic space.

\bprop\label{UCW-CIP}\cite{KohLeu08a}\\
The intersection of any decreasing sequence of nonempty bounded closed convex subsets of $X$ is nonempty.
\eprop

The next result is  inspired by \cite[Proposition 2.2]{GoeRei84}.

\bprop\label{UCW-basic-prop-as-center}\cite{Leu08}\\
Let $C$ be a closed convex subset of $X$, $f:C\to[0,\infty)$ be convex and lower semicontinuous. Assume moreover that for all sequences $(x_n)$ in $C$,
\[\limn d(x_n,a)=\infty \text{~for some~} a\in X \text{~implies~} \limn f(x_n)=\infty.\]
Then $f$ attains its minimum on $C$.  If, in addition, for all $x\neq y$,
\[f\left(\midxy\right) < \max\{f(x),f(y)\}\]
then $f$ attains its minimum at exactly one point.
\eprop

Let us recall that a subset $C$ of a metric space $(X,d)$ is called a {\em Chebyshev set} if to each point $x\in X$ there corresponds a unique point $z\in C$ such that $d(x,z)= d(x,C)(=\inf\{d(x,y)\mid y\in C\})$. If $C$ is a Chebyshev set, {\em nearest point projection} $P:X\to C$ can be defined by assigning $z$ to $x$.

\bprop\label{UCW-closed-convex-Chebyshev}\cite{Leu08}\\
Every  closed convex subset  $C$ of $X$ is a Chebyshev set.
\eprop

\subsection{Some related structures}\label{W-related-classes}

\subsubsection{Spaces of hyperbolic type}

Spaces of hyperbolic type were introduced by Goebel and Kirk \cite{GoeKir83} (see also \cite{Kir82}). Let $(X,d)$ be a metric space and $S$ be a family of geodesic segments in $X$. We say that the structure $(X,d,S)$ is a {\em space of hyperbolic type} if the following conditions are satisfied:
\be
\item for each two points $x,y\in X$ there exists a unique geodesic segment from $S$ that joins them, denoted $[x,y]$;
\item if $p,x,y\in M$ and if $m\in[x,y]$ satisfies $d(x,m)=\lambda d(x,y)$ for some $\lambda\in[0,1]$, then $$d(p,m)\le (1-\lambda)d(p,x)+\lambda d(p,y).$$
\ee

The following result shows that spaces of hyperbolic type are exactly the metric spaces with a convexity mapping $W$ satisfying $(W1),(W2),(W3)$.
 
\bprop
Let $(X,d)$ be a metric space. The following are equivalent.
\be
\item There exists a family $S$ of metric segments such that $(X,d,S)$ is a space of hyperbolic type.
\item There exists a a convexity mapping $W$ such that $(X,d,W)$ satisfies $(W1),(W2),(W3)$.
\ee
\eprop
\begin{proof}
$(i)\Ra(ii)$ It is easy to see that for all $x,y\in X$ and any $\lambda\in[0,1]$ there exists a unique $m\in [x,y]$ satisfying $d(x,m)=\lambda d(x,y)$ and $d(y,m)=(1-\lambda)d(x,y)$. 
Define $W:X\times X\times[0,1]\to X$ by $W(x,y,\lambda)=$ this unique $m$. Then $(X,d,W)$ satisfies $(W1),(W2),(W3)$. \\
$(ii)\Ra(i)$ For all $x,y\in X, x\ne y$, consider the geodesic $\gamma_{xy}$ joining $x$ and $y$, defined  by (\ref{W-def-geodesic}). For $x=y$, let $\gamma_{xx}:\{0\}\to X, \gamma(0)=x$. Taking $S:=\{\gamma_{xy}\mid x,y\in X\}$, we obtain that $(X,d,S)$ is a space of hyperbolic type. 
\end{proof}
As a consequence, any $W$-hyperbolic space is a space of hyperbolic type. In fact, $W$-hyperbolic spaces are exactly the spaces of hyperbolic type satisfying $(W4)$.

\subsubsection{Hyperbolic spaces in the sense of Reich and Shafrir}

The class of hyperbolic spaces presented in this section was defined by Reich and Shafrir \cite{ReiSha90} as an appropriate context for the study of operator theory in general, and of iterative processes for 
nonexpansive mappings in particular.

Let $(X,d)$ be a metric space and $M$ be a nonempty family of straight lines in $X$ with the following property: for each two distinct points $x,y\in X$ there is a unique straight line from $M$ which passes through $x,y$. 

We shall denote by $(X,d,M)$ a metric space $(X,d)$ together with a family $M$ as above. Since $M\ne\emptyset$, there is  at least one geodesic line $\gamma:\R\ra X$ with $\gamma(\R)\in M$, so $card(X)\geq card(\R)=\aleph_1$, as $\gamma$ is injective. Furthermore, the metric space $(X,d)$ must be unbounded.

The following lemma collects some useful properties. We refer to \cite{KohLeu03} for the proofs.

\blem$\,$
\be
\item For any $x\in X$ there is at least one  straight line from $M$ that
passes through $x$.
\item For any distinct points $x$ and $y$ in $X$, the unique straight line that passes through $x$ and $y$ determines in a unique way a geodesic segment joining $x$ and $y$, denoted by  $[x,y]$.
\item For all $x,y\in X$ and all $\lambda\in [0,1]$ there is a unique point
$z\in[x,y]$ satisfying
\beq
d(x,z)=\lambda d(x,y)\quad \text{~and~}\quad  d(y,z)=(1-\lambda)d(x,y).\label{habil-hypRS-oplus}
\eeq
\ee
The unique point $z$ satisfying (\ref{habil-hypRS-oplus}) will be denoted by $\lambdaxy$.
\elem

We say that the structure $(X,d,M)$ is a {\em hyperbolic space} if the following inequality is satisfied
\beq
d\left(\frac{1}{2}x\oplus\frac{1}{2}y,\frac{1}{2}x\oplus\frac{1}{2}z\right)
\leq \frac{1}{2}d(y,z).
\eeq

\begin{proposition}\cite{KohLeu03}\\
Let $(X,d,M)$ be a hyperbolic space. Then
\beq
d((1-\lambda)x\oplus \lambda z,(1-\lambda)y\oplus \lambda w)\leq (1-\lambda)d(x,y)+\lambda d(z,w)\label{habil-hypRS-W4}
\eeq
for all $x,y,z,w\in X$ and any $\lambda\in[0,1]$.
\end{proposition}

If we define 
$$W:X\times X\times [0,1]\to X, \quad W(x,y,\lambda)=\lambdaxy,$$
it is easy to see that $(X,d,W)$ is a $W$-hyperbolic space. Thus, any hyperbolic space in the sense of Reich and Shafrir is a $W$-hyperbolic space. 

\subsubsection{$CAT(0)$ spaces}\label{habil-CAT0}

In this section we give a very brief exposition of $CAT(0)$ spaces. We refer to the monograph by Bridson and  Haefliger \cite{BriHae99} for an extensive study of this important class of spaces.

Let $(X,d)$ be a geodesic space. A {\em geodesic triangle} in $X$ consists of three points $p,q,r\in X$, its {\em vertices}, and  a choice of three geodesic segments $[p,q], [q,r], [r,s]$ joining them, its {\em sides}. Such a geodesic triangle will be denoted $\Delta([p,q],[q,r], [r,s])$. If a point lies in the union of $[p,q], [q,r], [r,s]$, then we write $x\in\Delta$.

A triangle $\ol{\Delta}=\Delta(\ol{p},\ol{q},\ol{r})$ in $\R^2$ is called a  {\em comparison triangle} for the geodesic triangle $\Delta([p,q],[q,r], [r,s])$ if $d_{\R^2}(\ol{p},\ol{q})=d(p,q), d_{\R^2}(\ol{q},\ol{r})=d(q,r)$ and $d_{\R^2}(\ol{r},\ol{p})=d(p,r)$. Such a triangle $\ol{\Delta}$ always exists and it is unique up to isometry \cite[Lemma I.2.14]{BriHae99}.  We write $\ol{\Delta}=\ol{\Delta}(p,q,r)$ or $\Delta(\ol{p},\ol{q},\ol{r})$ according to whether a specific choice of $\ol{p},\ol{q},\ol{r}$ is required. A point $\ol{x}\in[\ol{p},\ol{q}]$ is called a {\em comparison point} for $x\in[p,q]$ if $d(p,x)=d_{\R^2}(\ol{p},\ol{x})$. Comparison points on $[\ol{q},\ol{r}]$ and $[\ol{r},\ol{p}]$ are defined similarly.

Let $\Delta$ be a geodesic triangle in $X$ and $\ol{\Delta}$ be a comparison triangle for $\Delta$ in $\R^2$. Then $\Delta$ is said to satisfy the $CAT(0)$ {\em inequality} if
for all $x,y\in \Delta$ and for all comparison points $\ol{x},\ol{y}\in\ol{\Delta}$, 
\beq
d(x,y)\leq d_{\R^2}(\ol{x}, \ol{y}).\label{CAT(0)-ineq}
\eeq

A geodesic space $X$ is said to be a {\em $CAT(0)$ space} if all geodesic triangles satisfy the $CAT(0)$ {\em inequality}. Complete $CAT(0)$ spaces are often called {\em Hadamard spaces}. It can be shown that $CAT(0)$ spaces are uniquely geodesic and that a normed space is a $CAT(0)$-space if and only if it is a pre-Hilbert space.

In the sequel, we give an equivalent characterization of $CAT(0)$ spaces, using the so-called: {\em CN inequality of Bruhat-Tits} \cite{BruTit72}: for all $x,y,z\in X$  and all $m\in X$ with $\ds d(x,m)=d(y,m)=\frac12 d(x,y)$,
\beq
d(z,m)^2\leq \frac12d(z,x)^2+\frac12d(z,y)^2-\frac14d(x,y)^2. \label{habil-CN-ineq}
\eeq

In the setting of $W$-hyperbolic spaces, we consider the following reformulation of the $CN$ inequality, which is nicer from the point of view of the logical metatheorems to be presented in Section \ref{logical-meta}: for all $x,y,z\in X$,
\bea
CN^-: \quad\quad d\left(z,\frac12 x\oplus \frac12 y\right)^2\leq \frac12d(z,x)^2+\frac12d(z,y)^2-\frac14d(x,y)^2. \label{CN-}
\eea

We refer to \cite[p. 163]{BriHae99} and to \cite[p. 386-388]{Koh08-book} for the proof of the following result.

\bprop\label{habil-char-CAT0}
Let $(X,d)$ be a metric space. \uae.
\be
\item \label{eq-CAT} $X$ is a CAT(0)-space.
\item \label{eq-geodesic+CN}   $X$ is a geodesic space that satisfies the $CN$ inequality  (\ref{habil-CN-ineq});
\item \label{eq-W+CN} There exists a a convexity mapping $W$ such that $(X,d,W)$ is a $W$-hyperbolic space satisfying the $CN$ inequality (\ref{habil-CN-ineq}).
\item\label{eq-W+CN-} There exists a a convexity mapping $W$ such that $(X,d,W)$ is a $W$-hyperbolic space satisfying the $CN^-$ inequality (\ref{CN-}).
\ee
\eprop
Thus, $CAT(0)$ spaces are exactly the $W$-hyperbolic spaces satisfying the $CN$ inequality. Furthermore

\bprop\cite{Leu07}\\
$CAT(0)$ spaces are $UCW$-hyperbolic spaces with a monotone modulus of uniform convexity 
$$\eta(\eps,r)=\frac{\eps^2}8,$$
that does not depend on $r$.
\eprop

\subsubsection{The Hilbert ball}

Let $H$ be a complex Hilbert space, and let $\B$ be the open unit ball in $H$. We consider the Poincar\' e metric on $\B$, defined by 
\beq
\rho(x,y):= \text{argtanh}(1-\sigma(x,y))^{1/2}, \quad \text{where~~}\sigma(x,y)=\frac{(1-\|x\|^2)(1-\|y\|^2)}{
|1-\langle x,y\rangle |^2}. \label{def-hyperbolic-Hilbert-ball}
\eeq
The metric space $(\B,\rho)$ is called the {\em Hilbert ball}. 

The Hilbert ball is a uniquely geodesic space (see \cite[Theorem 4.1]{KucReiSho01} or \cite{GoeSekSta80}). Moreover, by the inequality (4.2) in \cite{ReiSha90}, the CN inequality is satisfied. Applying Proposition \ref{habil-char-CAT0}.(\ref{eq-geodesic+CN}), it follows that the Hilbert ball is a $CAT(0)$ space.

We refer to Goebel and Reich's book \cite{GoeRei84} for an extensive study of the Hilbert ball.  

\subsubsection{Gromov hyperbolic spaces}\label{habil-Gromov-hyp}

Gromov's theory of hyperbolic spaces is set out in \cite{Gro87}.  The study of Gromov hyperbolic spaces has been largely motivated and dominated by questions about (Gromov) hyperbolic groups, one of the main object of study in geometric group theory. In the sequel, we review some definitions and elementary facts concerning Gromov hyperbolic spaces. For a more detailed account of this material, the reader is referred to \cite{Gro87,GhyHar90,BriHae99}. 

Let $(X,d)$ be a metric space. Given three points $x,y,w$, the {\em Gromov product} of $x$ and $y$  with respect to the {\em base point} $w$ is defined to be:
\beq
(x\cdot y)_w=\frac 12(d(x,w)+d(y,w)-d(x,y)).
\eeq
It measures the failure of the triangle inequality to be an equality and it is always nonnegative. 

\begin{definition}\label{habil-delta-hyp}
Let $\delta\geq 0$. $X$ is called $\delta-hyperbolic$ if for all $x,y,z,w\in X$,
\begin{equation}
(x\cdot y)_w\geq \min\{(x\cdot z)_w, (y\cdot z)_w\}-\delta.\label{habil-delta-hyp-ineq}
\end{equation}
We say that $X$ is {\em hyperbolic} if  it is $(\delta)$-hyperbolic
for some $\delta\geq 0$.
\end{definition}
It turns out that the definition is independent of the choice of the base point $w$ in the sense that if there exists {\em some} $w\in X$ such that the above inequality holds for all $x,y,z\in X$, then $X$ is $2\delta$-hyperbolic.

By unraveling the definition of Gromov product, (\ref{habil-delta-hyp-ineq}) can be rewritten as a 4-point condition: for all $x,y,z,w\in X$,
\begin{equation}
d(x,y)+d(z,w)\leq \max\{d(x,z)+d(y,w), d(x,w)+d(y,z)\}+2\delta. \label{habil-Q(delta)-ineq}
\end{equation}

\subsubsection{$\R$-trees}

The notion of $\Real$-tree was introduced by Tits \cite{Tit77}, as a generalization of the notion of local Bruhat-Tits building for rank-one groups, which itself generalizes the notion of simplicial tree. A more general concept, that of a $\Lambda$-tree, where $\Lambda$ is a totally ordered abelian group,  made its appearance as an essential tool in the study of groups acting on hyperbolic manifolds in the work of Morgan and Shalen \cite{MorSha84}. For detailed informations about $\Real(\Lambda)$-trees, we refer to \cite{Bes02,Chi01}.

\begin{definition}\cite{Tit77}
An {\em $\Real$-tree}  is a  geodesic space containing no homeomorphic image of a circle.
\end{definition}

We remark that in the initial definition, Tits only considered $\Real$-trees that are complete as metric spaces, but the assumption of completeness is usually irrelevant. The following proposition gives some equivalent characterizations of $\Real$-trees, which can be found in the literature.

\begin{proposition}\label{habil-R-trees-equiv-char}(see, for example, \cite{AlpBas87,Bes02,GhyHar90})\\
Let $(X,d)$ be a metric space. The following are equivalent:
\begin{enumerate}
\item $X$ is an $\Real$-tree,
\item $X$ is uniquely geodesic and for all $x,y,z\in X$, 
\bua
[y,x]\cap[x,z]=\{x\}\Rightarrow [y,x]\cup[x,z]=[y,z].
\eua
(i.e., if two geodesic segments intersect in a single point, then their union is a geodesic segment.)
\item $X$ is a geodesic space that is  (Gromov) $0$-hyperbolic, i.e. satisfies the inequality (\ref{habil-Q(delta)-ineq}) with $\delta=0$.
\end{enumerate}
\end{proposition}

The fact that $\Real$-trees are exactly the geodesic $0$-hyperbolic spaces follows from a very important result of Alperin and Bass \cite[Theorem 3.17]{AlpBas87} (see also \cite[Chapter 2, Exercise 8]{GhyHar90} and is the basic ingredient for proving the following characterization of $\Real$-trees using our notion of $W$-hyperbolic space.
 
\begin{proposition}\label{habil-char-R-trees}
Let $(X,d)$ be a metric space. The following are equivalent:
\begin{enumerate}
\item $X$ is an $\Real$-tree;
\item  there exists a convexity mapping $W$ such that $(X,d,W)$ is a $W$-hyperbolic space satisfying for all $x,y,z,w\in X$,
\begin{eqnarray*}
d(x,y)+d(z,w)\leq \max\{d(x,z)+d(y,w),d(x,w)+d(y,z)\}.
\end {eqnarray*}
\end{enumerate}
\end{proposition}

\subsection{Asymptotic centers and fixed point theory of nonexpansive mappings}\label{UCW-as-cen-fpp} 

The asymptotic center technique, introduced by Edelstein \cite{Ede72,Ede74}, is one of the most useful tools in metric fixed point theory of nonexpansive mappings in uniformly convex Banach spaces, due to the fact that bounded sequences have unique asymptotic centers with respect to closed convex subsets.

Let us recall basic facts about asymptotic centers. We refer to \cite{Ede72,Ede74,GoeRei84,GoeKir90} for details.

Let $(X,d)$ be a metric space, $(x_n)$ be a bounded sequence in $X$ and $C\se X$ be a nonempty subset of $X$. We define the following functionals:
\bua
r_m(\cdot, (x_n)):X\to [0,\infty), \quad r_m(y,(x_n))&=&\sup\{d(y,x_n)\mid n\geq m\}\\
&& \text{for~}m\in\N,\\
r(\cdot, (x_n)):X\to[0,\infty), \quad r(y,(x_n)) &=&\limsup_{n}d(y,x_n)=\inf_{m}r_m(y,(x_n)) \label{def-r-xn}\\
&=&\lim_{m\to\infty}r_m(y,(x_n)).
\eua
The following lemma collects some basic properties of the above functionals. 

\blem\label{UCW-prop-rm-r}
Let $y\in X$.
\be
\item $r_m(\cdot, (x_n))$ is nonexpansive for all $m\in \N$;
\item $r(\cdot, (x_n))$ is continuous and $r(y,(x_n))\to\infty$ whenever $d(y,a)\to\infty$ for some $a\in X$;
\item $r(y, (x_n))=0$ if and only if $\limn x_n=y$;
\item if $(X,d,W)$ is a convex metric space and $C$ is convex, then $r(\cdot, (x_n))$ is a convex function.
\ee
\elem

The {\em asymptotic radius of } $(x_n)$ {\em with respect to} $C$ is defined by
\[r(C,(x_n))  = \inf\{r(y,(x_n))\mid y\in C\}.\]
The  {\em asymptotic radius} of $(x_n)$, denoted by $r((x_n))$, is the asymptotic radius of $(x_n)$ with respect to $X$, that is $r((x_n))=r(X,(x_n))$.

A point $c\in C$ is said to be an {\em asymptotic center} of $(x_n)$ {\em with respect to } $C$ if
\[r(c,(x_n))= r(C,(x_n))=\min\{r(y,(x_n))\mid y\in C\}.\]
We denote with $A(C,(x_n))$ the set of asymptotic centers of $(x_n)$ with respect to $C$. When $C=X$, we call $c$ an {\em asymptotic center} of $(x_n)$ and we use the notation $A((x_n))$ for $A(X,(x_n))$. 

The following lemma, inspired by \cite[Theorem 1]{Ede74}, turns out to be very useful in the following. 
 
\blem\label{UCW-useful-unique-as-center} \cite{Leu08}\\
Let $(x_n)$ be a bounded sequence in $X$ with $A(C,(x_n))=\{c\}$ and  $(\alpha_n),(\beta_n)$ be real sequences such that $\alpha_n\geq 0$ for all $n\in\N$, $\limsup_n \alpha_n\leq 1$ and $\limsup_n \beta_n\leq 0$.\\
 Assume that $y\in C$  is such that there exist $p,N\in\N$ satisfying 
\[\forall n\ge N\bigg(d(y,x_{n+p})\leq \alpha_nd(c,x_n)+\beta_n\bigg).\] 
Then $y=c$.
\elem

In general, the set $A(C,(x_n))$ of asymptotic centers of a bounded sequence $(x_n)$ with respect to $C\se X$ may be empty or, on the contrary, contain infinitely many points. 

The following result shows that in the case of complete $UCW$-hyperbolic spaces, the situation is as nice as for uniformly convex Banach spaces (see, for example, \cite[Theorem 4.1]{GoeRei84}).

\bprop\label{UCW-unique-ac}\cite{Leu08}\\
Let $(X,d,W)$ be a complete $UCW$-hyperbolic space. Every bounded sequence $(x_n)$ in $X$ has a unique asymptotic center with respect to any closed convex subset $C$ of $X$.
\eprop

As an application of Proposition \ref{UCW-unique-ac} and Lemma \ref{UCW-useful-unique-as-center}, we can prove the following characterization of the fact that a nonexpansive mapping $T:C\to C$ has fixed points. 

\bthm\cite{Leu08}\label{UCW-equiv-T-fpp}\\
Let $C$  be a convex closed subset of a complete $UCW$-hyperbolic space  $(X,d,W)$ and $T:C\to C$ be nonexpansive.
The following are equivalent.
\be
\item\label{UCW-T-f}  $T$ has fixed points;
\item\label{UCW-un-bounded-Tun-un} $T$ has a bounded approximate fixed point sequence;
\item for all $x\in C$ there exists $b>0$ such that $T$ has approximate fixed points in a $b$-neighborhood of $x$;
\item there exist $x\in C$ and $b>0$ such that $T$ has approximate fixed points in a $b$-neighborhood of $x$;
\item\label{UCW-P-bounded-exists-x} the sequence $(T^nx)$ of Picard iterates is bounded for some $x\in C$;
\item\label{UCW-P-bounded-forall-x} the sequence $(T^nx)$ of Picard iterates is bounded for all $x\in C$.
\ee
\ethm

As an immediate consequence we obtain the generalization to complete $UCW$-hyperbolic spaces of the  Browder-G\" ohde-Kirk Theorem.

\bcor\label{UCW-BGK}
Let $C$  be a bounded convex closed subset of a complete $UCW$-hyperbolic space  $(X,d,W)$ and $T:C\to C$ be nonexpansive. Then $T$ has fixed points.
\ecor

\section{Logical metatheorems}\label{logical-meta}

In this section we give an informal presentation of the general logical metatheorems proved by Kohlenbach \cite{Koh05}  and Gerhardy-Kohlenbach \cite{GerKoh08}. We refer to Kohlenbach's book \cite{Koh08-book} for a comprehensive treatment.

The system ${\cal A}^{\omega}$  of so-called {\em weakly extensional} classical analysis goes back to Spector 
\cite{Spe62}. It is formulated in the language of functionals of finite types and consists of a finite type extension $\mathbf{PA}^\omega$ of first order Peano arithmetic $\mathbf{PA}$ and the axiom schema of dependent choice in all types, which implies countable choice and hence comprehension over natural numbers. Full second order arithmetic in the sense of reverse mathematics \cite{Sim89} is contained in ${\cal A}^{\omega}$ if we identify subsets of $\N$ with their characteristic functions. 

Let us recall the so-called {\em Axiom of Countable Choice}: \, For each set $B$ and each binary relation $P\se \N\times B$ between natural numbers and members of $B$,
\bua
\forall n\in \N\,\exists y\in B \, P(n,y) \quad \Ra \quad \exists f:\N\to B\, \forall n\in \N \, P(n,f(n)).
\eua
In contrast to the full Axiom of Choice which demands the existence of choice functions $f:A\to B$ for arbitrary sets $A,B$, the Axiom of Countable Choice justifies only a sequence of independent choices from an arbitrary set $B$ which successively satisfy the conditions
\[P(0,f(0)),\,\, P(1,f(1)),\,\, P(2,f(2)),\ldots\]

A stronger axiom is the {\em Axiom of Dependent Choice (DC)}:\, 
For each set $A$ and each relation  $P\se A\times A$,
\[
a\in A \text{ and } \forall x\in A\,\exists y\in A \, P(x,y) \quad \Ra \quad \exists f:\N\to A\big[f(0)=a\text{ and }   \forall n\in \N \, P(f(n),f(n+1))\big].
\]
The Axiom of Dependent Choice also justifies only a sequence of choices, where, however, each of them may depend on the previous one, since they must now satisfy the conditions
\[P(f(0),f(1)),\,\, P(f(1),f(2)),\,\, P(f(2),f(3)), \ldots\]
It is easy to see that the Axiom of Choice implies the Axiom of Dependent Choice, which implies further the Axiom of Countable Choice. 

The axiom scheme of {\em Comprehension over natural numbers} says that
\[
\exists f:\N\to\N \, \forall n\in\N \big(f(n)=0 \Lra A(n)\big),
\]
where $A(n)$ is an arbitrary formula in our language, not containing $f$ free but otherwise with arbitrary parameters. We refer to the very nice monograph \cite{Mos06} for details on set theory. 

The set ${\bf T}$ of all {\em finite types} is defined inductively by the clauses:
\be
\item $0,\in {\bf T}$;
\item if $\rho,\tau\in {\bf T}$ then $(\rho\ra\tau)\in {\bf T}$.
\ee
We usually omit the outermost parentheses for types. The intended interpretation of the base type $0$ is the set of natural numbers $\N=\{0,1,2,\ldots\}$. Objects of type $\rho\ra\tau$ are functions which map objects of type $\rho$ to objects of type $\tau$.  For example, $0\to 0$ is the type of functions $f:\N\to \N$ and $(0\to 0)\to 0$ is the type of operations $F$ mapping such functions $f$ to natural numbers. 

Any type  $\rho\ne 0$ can be uniquely written in the normal form  $\rho=\rho_1\ra(\rho_2\ra\ldots\ra(\rho_n\ra 0)\ldots )$ (for suitable $n\ge 1$ and types $\rho_1,\ldots, \rho_n$), which is usually abbreviated by $\rho=\rho_1\ra\rho_2\ra\ldots\ra\rho_n\ra 0$ if it is clear to which types $\rho_1,\ldots,\rho_n$ we refer and there is no danger of confusion.

We use the notation $\ul{x}$ for tuples of variables $\ul{x}=x_1,\ldots, x_n$ and $\ul{\rho}$ for tuples of types $\ul{\rho}=\rho_1,\ldots,\rho_n$. When we write $\ul{x}^{\ul{\rho}}$ we mean that each $x_i$ has type $\rho_i$. The notations $\ul{x}^{\rho}$ or $\ul{x}\in\rho$ mean that each $x_i$ is of type $\rho$.

The set {\bf P} $\subset {\bf T}$ of {\em pure types} is defined inductively by: (i) $0\in {\bf P}$  and (ii) if $\rho\in{\bf P}$, then $\rho\to 0\in {\bf P}$. Pure types are often denoted by natural numbers: $0\ra 0=1, \,\, (0\to 0)\to 0=2$, in general $n\to 0=n+1$.

The {\em degree} (or {\em type level}) $deg(\rho)$ of a type $\rho$ is defined as
\[ deg(0):=0, \quad deg(\rho\to\tau ):=\max(deg(\tau),deg(\rho)+1). \]
Note that for pure types $\rho$, $deg(\rho)$ is just the number which denotes $\rho$. Objects of type $\rho$ with $deg(\rho)>1$ are usually called {\em functionals}.

We shall denote formulas with $A,B,C,\ldots$ and quantifier-free formulas with $A_0,B_0,C_0,\ldots$. A formula $A$ is said to be {\em universal} if it has the form $A\equiv \forall\ul{x}\,A_0(\ul{x},\ul{a})$, where $\ul{x},\ul{a}$ are tuples of variables. Similarly, $A$ is an {\em existential } formula if $A\equiv \exists\, \ul{x}\,A_0(\ul{x},\ul{a})$.

Furthermore, $A$ is called a $\ds\Pi_n^0$-formula if it has $n$-alternating blocks of equal quantifiers starting with a block of universal quantifiers, that is
\[\forall \ul{x_1}\,\exists \ul{x_2}\ldots \forall/\exists \ul{x_n} A_0(\ul{x_1},\ldots,\ul{x_n},\ul{a}).\]
 If the formula starts with a block of existential quantifiers, that is
\[\exists \ul{x_1}\,\forall \ul{x_2}\ldots \forall/\exists \ul{x_n} A_0(\ul{x_1},\ldots,\ul{x_n}\ul{a}),\]
it is called a $\ds\Sigma_n^0$-formula.

We only include equality $=_0$ between objects of type $0$ as a primitive predicate. Equality between objects of higher types is defined {\em extensionally}: \, if $\rho=\rho_1\ra\ldots\ra\rho_n\ra 0$ and $s,t$ are terms of type $\rho$, then
\bua
s=_\rho t \,:= \, \forall y_1^{\rho_1},\ldots, y_n^{\rho_n} \big(sy_1\ldots y_n=_0 ty_1\ldots  y_n\big),\label{def-equality-rho}
\eua
where $y_1,\ldots, y_n$ are variables not occurring in $s,t$.

Instead of the full axiom of extensionality in all types, the system ${\cal A}^{\omega}$ only has a {\em quantifier-free rule of extensionality}:
\bua
\displaystyle\frac{A_0 \rightarrow s=_{\rho} t}{A_0
\rightarrow r[s/x]=_{\tau} r[t/x]}\,,\\
\eua
where $A_0$ is a quantifier-free formula, $s,t$ are terms of type $\rho$, $r$ is a term of type $\tau$, and $r[s/x]$ (resp. $r[t/x]$) is the result of replacing every occurrence of $x$ in $r$ by $s$ (resp. $t$). 
We refer to \cite{Koh05} for an extensive discussion of extensionality issues. 

In the sequel, we briefly recall the representation of real numbers in  $\aomega$. We refer to \cite[Chapter 4]{Koh08-book} for details. 

We will most times use $\N$ instead of $0$ and $\N^\N$ instead of $1$, say "natural numbers" instead of "objects of type 0", and write $n\in\N$ instead of $n^0$, respectively $f:\N\to\N$ instead of $f^1$. 

Rational numbers are represented as codes $j(n,m)$ of pairs of natural numbers: $j(n,m)$ represents the rational number $\ds\frac{\frac{n}{2}}{m+1}$ if $n$ is even, and the negative rational number $\ds -\frac{\frac{n+1}{2}}{m+1}$ otherwise. Here we use the surjective Cantor pairing $j$, defined by 
$\ds j(n,m)=\frac12(n+m)(n+m+1)+m.$

As a consequence, each natural number codes a uniquely determined rational number. An equality $=_\Q$ on the representatives of the rational numbers (i.e. on $\N$) together with operations $+_\Q,-_\Q,\cdot_\Q$  and predicates $<_\Q,\le_Q$ are defined primitive recursively in a natural way.  

In order to express the statement that $n$ represents the rational $r$, we write $n=_\Q \langle r\rangle$ or simply $n=\langle r\rangle$. Since a rational number $r$ possesses infinitely many representatives, $\langle \cdot \rangle$ is not a function. In fact, rational numbers are equivalence classes on $\N$ with respect to $=_\Q$, but one can avoid formally introducing the set $\Q$ of all these equivalence classes. An alternative is to select a {\em canonical representative} by defining
\beq
c:\N\to\N, \quad c(n):=_0 \min\, m\le_0 n [n=_\Q m]. \label{rep-rational-canonical}
\eeq
Then $c(n)$ is the code of the irreducible fraction representing the rational number encoded by $n$. It is clear that $c(n)=_\Q n$ and $n=_\Q m\to c(n)=_\Q c(m)$. 

$\N$ can be naturally embedded into our representation of $\Q$ via $n\mapsto \langle n\rangle :=j(2n,0), \,\,0_{\Q} :=\langle 0\rangle, \,\, 1_{\Q} :=\langle 1\rangle$. 
Then $(\N,+_{\Q},\cdot_{\Q},0_{\Q},1_{\Q},<_{\Q})$ is an ordered field, which represents $(\Q,+,\cdot,0,1,<)$ in  $\aomega$.

Each function $f:\N\rightarrow\N$ can be conceived of as an infinite sequence of codes of rationals and therefore as a representative of a sequence of rationals. Real numbers are represented by functions $f:\N\rightarrow\N$ such that 
\beq
\ \aquant n\in\N\big( |f(n+1)-_{\Q} f(n)|_{\Q} <_{\Q} 2^{-n}\big)\label{Aomega-rep-reals}
\eeq
For better readability, we usually write $2^{-n}$  instead of its (canonical) code $\langle 
2^{-n} \rangle:=j\left(2,2^n-1\right)$.

(\ref{Aomega-rep-reals}) implies that for all $m,n,p\in\N$ with $m\ge n$,
\[ |f(m+p)-_{\Q} f(m)|_{\Q} \le_{\Q} 
\sum^{m+p-1}_{i=m} |f(i+1)-_{\Q} f(i)|_{\Q} \le_{\Q} \sum ^{\infty}_{i=n} |f(i+1)-_{\Q} f(i)|_{\Q} < 2^{-n},\]
hence each $f$ satisfying (\ref{Aomega-rep-reals}) in fact represents a Cauchy sequence of rationals with Cauchy modulus $2^{-n}$. In order to guarantee that each function $f:\N\to\N$ codes a real number, we use the following construction:
\beq
\widehat{f}(n):= \left\{ \ba{l} f(n) \ \mbox{if} \ \aquant k<n 
\big( |f(k+1)-_{\Q} f(k)|_{\Q} <_{\Q} 2^{-k-1} \big),
\\[0.2cm]
f(k) \ \mbox{for the least} \ k<n \ \mbox{with} \ |f(k+1)-_{\Q} 
f(k)|_{\Q} \ge_{\Q}  2^{-k-1} \quad\mbox{otherwise.} \ea 
\right. 
\eeq
Then $\widehat{f}$ always satisfies $(\ref{Aomega-rep-reals})$ and, moreover, if $(\ref{Aomega-rep-reals})$ is already valid for $f$, then $\aquant n(fn=_0 \widehat{f} n)$. Thus each function $f:\N\to\N$ codes a 
uniquely determined real number, namely the real number which is given 
by the Cauchy sequence coded by $\widehat{f}$. The construction $f\mapsto\widehat{f}$ allows us to reduce 
quantification over $\R$ to $\forall f:\N\to\N$ resp. $\exists f:\N\to\N$  without adding further quantifiers. 
This also holds for the operations on $\R$ defined below. 

On the representatives of real numbers, i.e. on the functions $f_1,f_2:\N\to\N$, one defines the relations $=_{\R}, <_{\R}$ and $\le_\R$:
\bua
f_1 =_{\R} f_2 &:\equiv & \aquant n\big( |\widehat{f}_1 (n+1)-_{\Q} \widehat{f}_2 (n+1)|_{\Q} <_{\Q} 2^{-n}\big), \\
f_1 <_{\R} f_2 &:\equiv & \exists n\big( \widehat{f}_2 (n+1)-_{\Q} \widehat{f}_1 (n+1) \ge_{\Q} 2^{-n}\big), \\
f_1 \le_{\R} f_2 &:\equiv& \neg (f_2 <_{\R} f_1).
\eua
Hence, the relations $=_\R, \le_\R$ are given by $\Pi_1^0$ predicates, while $<_\R$ is given by a $\Sigma_1^0$ predicate. 

The operations $+_\R,-_\R,\cdot_\R$, etc. on representatives of real numbers can be defined by primitive recursive functionals. If $n=\langle r\rangle$  codes the rational number $r$, then $\lambda k.n$ represents $r$ as a real number. Thus, $0_{\R} :=\lambda k.0_{\Q},\,\,1_{\R} :=\lambda k.1_{\Q}$ and $\left(2^{-n}\right)_\R:=\lambda k.j(2.2^n-1)$; we shall write simply $2^{-n}$ for $\left(2^{-n}\right)_\R$. $\R$ denotes the set of all equivalence classes on  $\N^\N$ with respect to $=_{\R}$. As in the case of $\Q$, we use $\R$ 
only informally and deal exclusively with the representatives and the operations defined on them. One  can verify that $\left(\N^\N,+_\R,\cdot_\R,0_\R,1_\R,<_\R\right)$ is an Archimedean ordered field which represents $(\R,+,\cdot,0,1,<)$ in $\aomega$.

In the sequel, we need a semantic operator  which for any real number $x\in[0,\infty)$ selects out of all the representatives $f:\N\to\N$ of $x$  a unique representative $(x)_\circ$  satisfying some "nice" properties. For  any $x\in[0,\infty)$, $(x)_\circ:\N\to\N$  is defined by
\beq
(x)_0(n):=j(2k_0,2^{n+1}-1), \quad \text{where~~} k_0:=\max \,k\left[\frac k{2^{n+1}}\leq x\right]. \label{def-circ-rep}
\eeq

\blem \cite[Lemma 17.8]{Koh08-book}
Let $x\in[0,\infty)$. Then
\be
\item $(x)_\circ$ is a representative of $x$, so $\widehat{(x)_\circ}=_{\N\to\N}(x)_\circ$;
\item if  $x,y\in[0,\infty)$ and $x\le y$, then  $(x)_\circ\le_\R (y)_\circ$ 
and $(x)_\circ\le_{\N\to\N} (y)_\circ$, i.e. $\forall n\in N\big((x)_\circ(n)\le_\N (y)_\circ(n)\big)$;
\item $(x)_\circ$ is monotone, that is $\forall n\in N\big((x)_\circ(n)\le_\N (x)_\circ(n+1)\big)$.
\ee
\elem

Since the interval $[0,1]$ will play a very important role in the theory of $W$-hyperbolic spaces, we use for it a special representation by number theoretic functions $\N\to\N$. For every $\lambda:\N\to\N$, let us define
\beq
\tilde{\lambda}:=\lambda n.j(2k_0, 2^{n+2}-1), \quad \text{where~~} k_0=\max k\le 2^{n+2}\left[\frac{k} {2^{n+2}}\le_\Q \widehat{\lambda}(n+2)\right] \label{def-tilde-rep}
\eeq
($k_0:=0$ if no such $k$ exists; recall that $j(2k_0, 2^{n+2}-1)$ encodes the rational number $k_0/2^{n+2}$).

It is easy to verify the following properties.

\blem\cite[Lemma 4.25]{Koh08-book}
Provably in ${\cal A}^\omega$, for all $\lambda,\theta:\N\to\N$:
\be
\item $0_\R\leq_\R \lambda\leq_\R 1_\R \ra \tilde{\lambda}=_\R \lambda$,  $\,\,\lambda>_\R 1_\R \ra  \tilde{\lambda}=_\R 1_\R$  and $\lambda<_\R 0_\R \ra  \tilde{\lambda}=_\R 0_\R$,
\item $0_\R\leq_\R \tilde{\lambda}\leq_\R 1_\R$,
\item $\lambda =_\R\theta\ra \widehat{\lambda}=_\R \widehat{\theta}$,
\item $\tilde{\lambda}\leq_1 M:=\lambda n.j(2^{n+3}, 2^{n+2}-1)$.
\ee
\elem

\subsection{Logical metatheorems for metric and $W$-hyperbolic spaces}

In order to be able to talk about arbitrary metric spaces, we axiomatically add general metric spaces $(X,d)$ to our system $\aomega$, resulting in a theory ${\cal A}^\omega[X,d]_{-b}$ which is based on two ground types $\N,X$ rather  than only $\N$. Hence, the theory ${\cal A}^\omega[X,d]_{-b}$ for abstract metric spaces is an extension of ${\cal A}^\omega$ defined as follows:
\be
\item extend ${\bf T}$ to the set ${\bf T}^X$ of all finite types over the ground types $\N$ and $X$, that is:
\[\N,X\in {\bf T}^X \quad \text{and}\quad  \rho,\tau\in {\bf T}^X\,\, \Ra \,\, \rho\ra\tau\in {\bf T}^X;\] 
\item extend all the axioms and rules of $\aomega$ to the new set of types ${\bf T}^X$;
\item add a constant $0_X$ of type $X$;
\item add a new constant $d_X$ of type $X\to X\to \N^\N$ together with the axioms
\be
\item[(M1)] $\forall x^X \, \big(d_X(x,x) =_\R 0_\R\big)$,
\item[(M2)] $\forall x^X, y^X \, \big(d_X(x,y)=_\R d_X(y,x)\big)$, 
\item[(M3)] $\forall x^X, y^X, z^X \, \big(d_X(x,z)\leq_\R d_X(x,y)+_\R d_X(y,z)\big).$
\ee
\ee

We use the subscript $_{-b}$ here and for the theories defined in the sequel in order to be consistent with the notations from \cite{Koh08-book}.

Equality $=_X$ between objects of type $X$ is defined by $x=_X y:\equiv d_X(x,y)=_\R 0_\R$ and equality for complex types is defined as before as extensional equality using $=_\N$ and $=_X$ for the base cases. The new axioms  (M1)-(M3) of $\aomegaxd$ express that $d_X$ represents  a pseudo-metric $d$ on the domain the variables of type $X$ are ranging over. Thus, $d_X$ represents a metric on the set of equivalence classes generated by $=_X$. We do not form these classes explicitly, but talk instead about representatives $x^X,y^X$. As a consequence, we have to keep in mind that a functional $f^{X\to X}$ represents a function $X\to X$ only if it respects this equivalence relation, i.e. $\forall x^X,y^X\big(x=_X y\ra f(x)=_X f(y)\big)$. However, the mathematical properties of the functions considered in applications of proof mining usually imply their full extensionality.

The theory ${\cal A}^\omega[X,d,W]_{-b}$ for $W$-hyperbolic spaces results from ${\cal A}^\omega[X,d]_{-b}$ by adding a new constant $W_X$ of type $X\to X\to \N^\N\to X$ together with the axioms:
\be
\item[] \hspace*{-0.5cm}$\forall x^X, y^X, z^X\, \forall \lambda:\N\to\N\bigg(d_X(z,W_X(x,y,\lambda)) \le_\R  (1_\R-_\R\tilde{\lambda})\cdot_\R d_X(z,x)+_\R\tilde{\lambda}\cdot_\R d_X(z,y)\bigg)$,
\item[] \hspace*{-0.5cm}$\forall x^X, y^X\,\forall \lambda_1:\N\to\N,\lambda_2:\N\to\N\bigg( d_X(W_X(x,y,\lambda),W_X(x,y,\tilde{\lambda}))  =_\R|\tilde{\lambda_1}-_\R\tilde{\lambda_2}|_\R \cdot_\R 
d_X(x,y)\bigg)$,
\item[] \hspace*{-0.5cm}$\forall x^X, y^X\,\forall  \lambda:\N\to\N\bigg(W_X(x,y,\lambda) =_X W_X(y,x,1_\R-_\R\lambda)\bigg)$,
\item[] \hspace*{-0.5cm}\mbox{$\forall x^X\!, \!y^X\!, \!z^X,\!w^W\forall \lambda
\,\bigg(d_X(W_X(x,z,\lambda),W_X(y,w,\lambda))\le_\R (1_\R-_\R\tilde{\lambda})\cdot_\R d_X(x,y)+_\R\tilde{\lambda}\cdot_\R d_X(z,w)\bigg).$}
\ee
In the above axioms, $\tilde{\lambda}$ is defined by (\ref{def-tilde-rep}). 

\bdfn
Let $X$ be a nonempty set. The {\em full-theoretic type structure} $S^{\omega,X}:=\left<S_\rho\right>_{\rho\in {\bf T}^X}$ over $\N$ and $X$ is defined as follows:
$$S_\N:=\N, \quad S_X:=X \quad \text{and  ~~} S_{\rho\rightarrow \tau}:=S_\tau^{S_\rho},$$
where $S_\tau^{S_\rho}$ is the set of all set-theoretic functions $S_\rho\rightarrow S_\tau$.
\edfn 

Let $(X,d)$ be a metric space. $S^{\omega,X}$ becomes a model of $\aomegaxd$ by letting the variables of type $\rho$ range over $S_\rho$, giving the natural interpretations to the constants of $\aomega$, interpreting  $0_X$ by an arbitrary element in $X$ and $d_X(x,y)$ (for $x,y\in X$) by $(d(x,y))_\circ$, where $(\cdot)_\circ$ refers to (\ref{def-circ-rep}).

If, moreover, $(X,d,W)$ is a $W$-hyperbolic space, then $S^{\omega,X}$ becomes a model of $\aomegaxdw$ if we interpret $W_X(x,y,\lambda)$  (for $x,y\in X,\lambda:\N\to\N$) as  $W(x,y,r_{\tilde{\lambda}})$ where $r_{\tilde{\lambda}}$ is the uniquely determined real number in $[0,1]$ which is represented by $\tilde{\lambda}$.

\begin{definition}\label{habil-model-metric+W-hyp}
We say that a sentence in the language ${\cal L}(\mathcal{A}^\omega[X,d]_{-b})$ of $\aomegaxd$ holds in a nonempty metric space $(X,d)$ if it holds in the models of $\mathcal{A}^\omega[X,d]_{-b}$ obtained from $S^{\omega,X}$ as specified above. 

The notion that a sentence of ${\cal L}(\mathcal{A}^\omega[X,d,W]_{-b})$ holds in a nonempty $W$-hyperbolic space is defined similarly.
\end{definition}

From now on, in order to improve readability, we shall usually omit the subscripts $_\N, _\R,_\Q,_X$ excepting the cases where such an omission could create confusions. We shall write, for example, $x\in X, T:X\to X$ instead of $x^X, T^{X\to X}$ and sometimes $x\in \rho$ instead of $x^\rho$. 

The notion of {\em majorizability} was originally introduced by Howard \cite{How73}, and subsequently modified by Bezem \cite{Bez85}. For any type $\rho\in {\bf T}^X$, we define the type $\widehat{\rho}\in {\bf T}$, which is the result of replacing all occurrences of the type $X$ in $\rho$ by $\N$. Based on Bezem's notion of {\em strong majorizability} {\em s-maj} \cite{Bez85}, Gerhardy and Kohlenbach \cite{GerKoh08} defined a parametrized {\em $a$-majorization} relation $\gtrsim^a_\rho$ between objects of type $\rho\in {\bf T}^X$ and their majorants of type $\widehat{\rho}\in {\bf T}$, where the parameter $a$ of type $X$ serves as a reference point for comparing and majorizing elements of $X$:
\be
\item ${x^*}\gtrsim^a_\N  x:\equiv  x^* \ge x$ for $x,x^*\in\N$
\item ${x^*}\gtrsim^a_X x  :\equiv (x^*)_\R\ge_\R d(x,a)$ for $x^*\in\N, x\in X$,
\item $x^* \gtrsim^a_{\rho\rightarrow\tau} x :\equiv  \forall y^*,y(y^*\gtrsim^a_{\rho} y \rightarrow x^* y^*  \gtrsim^a_{\tau} xy)\wedge \forall z^*,z(z^*\gtrsim^a_{\hat{\rho}} z \rightarrow x^* z^* \gtrsim^a_{\hat{\tau}} x^*z)$.
\ee
Restricted to the types $\bf T$ the relation $\gtrsim^a$ is identical with  Bezems's strong majorizability {\em s-maj} and, hence, for $\rho\in \bf T$ we write {\em s-maj}$_\rho$ instead of $\gtrsim^a_\rho$, since in this case the parameter $a$ is irrelevant. 

If $t^*\gtrsim^a t$ for terms $t^*,t$, we say that $t^*$ {\em $a$-majorizes} $t$ or that $t^*$ is an {\em $a$-majorant} of $t$.  A term $t$ is said to be {\em $a$-majorizable} if it has an $a$-majorant and $t$ is said to be {\em majorizable} if it is $a$-majorizable for some $a\in X$. Since it can be shown that if a term $t$ is {\em $a$-majorizable} for some $a\in X$, then this is true for all $a\in X$ \cite[Lemma 17.78]{Koh08-book}, it follows that $t$ is majorizable if and only if it is  $a$-majorizable for each $a\in X$. Although the question whether or not a certain term is $a$-majorizable is independent from the particular choice of $a\in X$, the complexity and possible uniformities of the majorants may depend crucially on that choice.
If $t^*$ $a$-majorizes $t$ and does not depend on $a$, then we say that $t^*$ {\em uniformly} $a$-majorizes $t$. We will in general look for uniform majorants so as to produce uniform bounds.

\blem\label{lemma-T-majorizable}
Let $T:X\to X$. The following are equivalent.
\be
\item $T$ is majorizable;
\item for all $x\in\N$ there exists $\,\Omega:\N\to\N$ such that 
\beq
\forall n\in\N, y\in X\bigg(d(x,y) < n \ra d(x,Ty)\leq \Omega(n)\bigg) \label{modulus-maj-x}
\eeq
\item for all $x\in\N$ there exists $\,\Omega:\N\to\N$ such that 
\beq
\forall n\in\N, y\in X\bigg(d(x,y)\le  n \ra d(x,Ty)\leq \Omega(n)\bigg) \label{modulus-maj-x-leq}
\eeq
\ee
\elem
\begin{proof}
$T$ is majorizable if and only if $T$ is $x$-majorizable for each $x\in X$ if and only if for each $x\in X$ there exists a function $T^*:\N\to\N$ such that $T^*$ is monotone and satisfies
$$\forall n\in\N\, \forall y\in X\bigg(d(x,y)\le n \ra d(x,Ty)\le T^*n\bigg).$$
$(i)\Ra(iii)$ is obvious: take $\Omega:=T^*$. For the implication $(iii)\Ra(i)$, given, for $x\in X$, $\Omega$ satisfying (\ref{modulus-maj-x-leq}), define $\ds T^*n:=\max_{k\le n}\Omega(k)$. \\
$(iii)\Ra(ii)$ is again obvious. For the converse implication, given $\Omega$ satisfying  (\ref{modulus-maj-x}) define $\tilde{\Omega}(n):=\Omega(n+1)$. Then  $\tilde{\Omega}$ satisfies (\ref{modulus-maj-x-leq})
\end{proof}

In the sequel, given a majorizable function $T:X\to X$, an $\Omega$ satisfying (\ref{modulus-maj-x-leq}) will be called  a {\em modulus of majorizability at $x$} of $T$; we say also that $T$ is {\em x-majorizable with modulus} $\Omega$. We gave in the lemma above the equivalent condition (\ref{modulus-maj-x}) for logical reasons: since $<_\R$ is a $\Sigma_1^0$ predicate and $\le_\R$ is a $\Pi_1^0$ predicate, the formula in (\ref{modulus-maj-x}) is (equivalent to) a universal sentence. 

The following lemma shows that natural classes of mappings in metric or $W$-hyperbolic spaces are majorizable with a very "nice" modulus; its proof is implicit in the proof of \cite[Corollary 17.55]{Koh08-book}.

\blem\label{habil-classes-majorizable}
Let $(X,d)$ be a metric space. 
\be
\item If $(X,d)$ is bounded with diameter $d_X$, then any function $T:\N\to\N$ is majorizable with modulus of majorizabiliy  $\Omega(n):=\lceil d_X\rceil$ for each $x\in X$.
\item If $T:X\to X$ is $L$-Lipschitz, then $T$ is majorizable with modulus at $x$ given by $\Omega(n):=n+ L^*b$, where $b,L^*\in\N$  are such that $d(x,Tx)\leq b$ and $L\le L^*$. In particular, any nonexpansive mapping is majorizable with modulus $\Omega(n):=n+b$.
\item If $(X,d,W)$ is a $W$-hyperbolic space, then and uniformly continuous mapping $T:X\to X$ is majorizable with modulus  $\Omega(n):=n\cdot 2^{\alpha_T(0)}+1+b$ at $x$,  where $d(x,Tx)\leq b\in\N$ and $\alpha_T$ is a modulus of uniform continuity of $T$, i.e. $\alpha_T:\N\to\N$ satisfies
\[\forall x,y\in X\,\forall k\in\N\left(d(x,y)\leq 2^{-\alpha_T(k)}\ra d(Tx,Ty)\leq 2^{-k}\right).\]
\ee
\elem

Before stating  the main logical metatheorem, let us give a couple of definitions. Let $\rho\in{\bf T}^X$ be a type. We say that
\be
\item $\rho$ has degree $(0,X)$ if $\rho=X$ or $\rho=\N\to\ldots \ra\N\to X$;
\item $\rho$ is of degree $(1,X)$  if $\rho=X$ or has the form $\rho=\rho_1\ra\ldots \ra\rho_n\ra X$, where $n\ge 1$ and each $\rho_i$ has degree $\le 1$ or $(0,X)$.
\item $\rho$ has degree $1^*$ if $deg(\widehat{\rho})\le 1$.
\ee
A formula $A$ is called a {\em $\forall$-formula} (resp. {\em $\exists$-formula}) if it has the form $$A\equiv\forall\underline{x}^{\underline{\sigma}}A_0(\underline{x},\underline{a}) \quad\text{(resp. } A\equiv \exists \underline{x}^{\underline{\sigma}}A_0(\underline{x},\underline{a})),$$
where $A_0$  is a quantifier free formula and the types in $\underline{\sigma}$ are of degree $1^*$ or $(1,X)$.

We assume in the following that the constant $0_X$ does not occur in the formulas we consider. This is no restriction, since $0_X$ is just an arbitrary constant which could have been replaced by any new variable of type $X$ that, by taking universal closure, would just add another input that had to be $a$-majorized.  Whenever we write $A(\ul{x})$, we mean that $A$ is a formula in our language which has only the variables $\ul{x}$ free.

Very general metatheorems were proved first by Kohlenbach \cite{Koh05} for bounded metric ($W$-hyperbolic) spaces, and then generalized to the unbounded case  by Gerhardy and Kohlenbach \cite{GerKoh08}. In the following we give a simplified version of these metatheorems, specially designed for concrete applications in mathematics.

\begin{theorem}\label{habil-logic-main-thm}(see \cite[Corollary 17.54]{Koh08-book})\\
Let $P$ be $\N$ or $\N^\N$, $K$ be an $\mathcal{A}^\omega$-definable compact metric space, $\rho$ be of degree $1^*$,  $B_\forall(u,y,z,n)$ be a $\forall$-formula and $C_\exists(u,y,z,N)$ be a $\exists$-formula. 

Assume that $\mathcal{A}^\omega[X,d]_{-b}$ proves that
\begin{equation}
\forall u\in P\forall y\in K \forall z^\rho\bigg(\forall n\in\N\,  B_\forall\rightarrow \exists N\in\N \,C_\exists\bigg).
\end{equation}
Then one can extract a computable functional $\Phi:P\times \N^{(\N\times\ldots \times\N)}\to\N$ such that the following holds in all nonempty metric spaces $(X,d)$: \\

for all $z\in S_\rho, z^*\in \N^{(\N\times\ldots \times\N)}$, if there exists $a\in X$ such that $ z^*\gtrsim^a_\rho  z$, then 
$$\forall u\in P\forall y\in K \bigg(\forall n\le \Phi(u,z^*)\, B_\forall\ra \exists N\le \Phi(u,z^*)  \,C_\exists\bigg).$$
\end{theorem}

\bfact
\be
\item The above theorem holds for $\mathcal{A}^\omega[X,d,W]_{-b}$ and nonempty $W$-hyperbolic spa-ces $(X,d,W)$ too.
\item Instead of single variables $u,y,n$ and single premises $\forall n B_\forall(u,y,z,n)$ we may have tuples $\ul{u}\in P, \ul{y}\in K, \ul{n}\in \N$ of variables and finite conjunctions of premises. Moreover, we can have also $\underline{z}^{\ul{\rho}}=z_1^{\rho_1}, \ldots z_k^{\rho_k}$ as long as all the types $\rho_1,\ldots, \rho_k$ are of degree $1^*$ and in the conclusion is assumed that $z_i^*\gtrsim^a_{\rho_i}z_i$ for a {\em common} $a\in X$ for all $i=1,\ldots, k$. Furthermore, the bound $\Phi$ depends now on all the $a$-majorants $z_1^*,\ldots, z_k^*$.
\ee
\efact

\bfact
The theory ${\cal A}^{\omega}[X,\|\cdot \|]$ of normed spaces and ${\cal A}^{\omega}[X,\|\cdot \|,\eta]$ corresponding to uniformly convex normed spaces were defined by Kohlenbach \cite{Koh05} and similar logical metatheorems were obtained for these theories too. We refer to \cite{Koh05} or to \cite[Section 17.3]{Koh08-book} for details.
\efact

The proof of the above logical metatheorem is based on an extension to $\mathcal{A}^\omega[X,d]_{-b}$, resp. $\mathcal{A}^\omega[X,d,W]_{-b}$, of Spector's \cite{Spe62} interpretation of classical analysis $\aomega$ by {\em bar-recursive} functionals followed by an interpretation of these functionals in an extension of Bezem's \cite{Bez85} type structure of hereditarily strongly majorizable functionals to all types ${\bf T}^X$, based on the $a$-majorization relation $\gtrsim^a$, parametrized by $a\in X$. Spector's work generalizes  G\" odel's well-known {\em functional interpretation}  \cite{God58} for intuitionistic and - via G\" odel's {\em double-negation} interpretation \cite{God33} as intermediate step - classical arithmetic to classical analysis. We refer to \cite{Koh08a} for a recent survey on applied aspects of functional interpretation  and to \cite{Luc73} for a book treatment of Spector's bar recursion.

Moreover, the proof of the metatheorem actually provides an extraction algorithm for the functional $\Phi$, which  can always be defined in the calculus of bar-recursive functionals. However, as we shall see in Section \ref{proof-mining-fpt},  for concrete applications usually small fragments of $\mathcal{A}^\omega[X,d,W]_{-b}$ or $\mathcal{A}^\omega[X,d]_{-b}$ (corresponding to fragments of $\aomega$) are needed to formalize the proof. 
In particular, it follows from results of  Kohlenbach \cite{Koh96,Koh96a} that a single use of sequential compactness (over a sufficiently weak base system) only gives rise to at most primitive recursive complexity in the sense of Kleene, often only simple exponential complexity. This corresponds to the complexity of the bounds  obtained in our applications from Section \ref{proof-mining-fpt}. 

In these applications, one actually is interested in the extraction of bounds which, in order to be useful, should be {\em uniform}, i.e. independent from  various parameters. This can be achieved by using Kohlenbach's {\em monotone} functional interpretation, introduced in \cite{Koh96a}  (see \cite[Chapter 9]{Koh08-book} for details), that systematically transforms any statement in a given proof into a new version for which explicit bounds are provided. In recent years, other "bounds-oriented" variants of functional interpretation were defined, as {\em bounded} functional interpretation introduced by  Ferreira and Oliva \cite{FerOli05,FerOli07} or the very recent  Shoenfield-like bounded functional interpretation of Ferreira \cite{Fer08}, that gives a direct interpretation of 
classical theories and so could be suitable for proof mining.\\

We give now a very useful corollary of Theorem \ref{habil-logic-main-thm}.

\bcor\label{meta-BRS}(see \cite[Corollary 17.54]{Koh08-book})\\
Let $P$ be $\N$ or $\N^\N$, $K$ be a $\aomega$-definable compact metric space, $B_\forall(\ul{u},\ul{y},x,x^*,T,n)$ be a $\forall$-formula and $C_\exists(\ul{u},\ul{y},x,x^*,T,N)$ a  $\exists$-formula. Assume that $\mathcal{A}^\omega[X,d,W]_{-b}$ proves that
\[\ba{l}
\forall\,  \ul{u}\in P\,\forall\,\Omega:\N\to\N\, \forall\, \ul{y}\in K \, \forall \, x,x^*\in X\,\forall\, T:X\to X \\
\quad\quad\quad\quad
\bigg( T \text{ is } x\text{-majorizable with modulus } \Omega\, \si\,   \forall n\in\N \, B_\forall\,\ra \,\exists N\in\N \, C_\exists\bigg).
\ea\]
Then one can extract a computable functional $\Phi$ such that  for all $b\in\N$, 
\[\ba{l}
\forall\,  \ul{u}\in P\,\forall\,\Omega:\N\to\N\,\forall\, \ul{y}\in K \, \forall \, x,x^*\in X\,\forall\, T:X\to X \\
\quad\quad\bigg( T \text{ is }  x\text{-majorizable with modulus } \Omega \, \si \, d(x,x^*)\leq b\, \si\,   \forall n\le \Phi(\ul{u},b,\Omega) \, B_\forall \\
\qquad\qquad\qquad\qquad\ra \,\exists N\le \Phi(\ul{u},b,\Omega) \, C_\exists\bigg).
\ea\]
holds in all nonempty $W$-hyperbolic spaces $(X,d,W)$.
\ecor
\begin{proof}
The premise "$T$ $x$-majorizable with modulus $\Omega$"  is a $\forall$-formula, by (\ref{modulus-maj-x}). Furthermore, $0$ $x$-majorizes $x$, $b$ is a $x$-majorant for $x^*$, since $d(x,x^*)\le b$, and $\ds T^*:=\lambda n. \max_{k\le n}\Omega(k)$ $x$-majorizes $T$, by the proof of Lemma \ref{lemma-T-majorizable}. Apply now Theorem \ref{habil-logic-main-thm}
\end{proof}

\bfact
As in the case of Theorem \ref{habil-logic-main-thm}, instead of single $n\in \N$ and a single premise $\forall n B_\forall$ we could have tuples $\ul{n}=n_1,\ldots, n_k$ and a conjunction of premises $\forall n_1 B_\forall^1\si\ldots \si \forall n_k B_\forall^k$. In this case, in the conclusion we shall have in the premise 
$\forall n_1\le \Phi\,B_\forall^1\si\ldots \si \forall n_k \le \Phi\, B_\forall^k$.
\efact

Corollary \ref{meta-BRS}  will be used for our first application in metric fixed point theory, a quantitative version of Borwein-Reich-Shafrir Theorem (see Subsection \ref{section-app-fpt-brs}). In fact, a simplified version of it suffices for this application, namely for $T$ nonexpansive. In this case, as we have seen in Lemma \ref{habil-classes-majorizable}, a modulus of majorizability at $x$ is given by $\Omega(n)=n+b$, where $b\ge d(x,Tx)$, so the bound $\Phi$ will depend only on the parameters $\ul{u}\in P$ and $b\in\N$ such that $d(x,Tx),d(x,x^*)\le b$.\\

A remarkable feature of the (proof of the) above logical metatheorem is the fact the same results hold true for extensions of the theories $\aomegaxd, \aomegaxdw$ obtained as follows: 
\begin{enumerate}
\item the theory may be extended by new axioms that have the form of $\forall$-sentences;
\item the language may be extended by new majorizable constants, in particular constants of type $\N$ or $\N^\N$ which are uniformly majorizable. In this case, the extracted bounds then additionally depend on $a$-majorants for the new constants.
\end{enumerate} 
Then the conclusion holds in all metric spaces $(X,d)$, resp. $W$-hyperbolic spaces $(X,d,W)$, satisfying these axioms (under a suitable interpretation of the new constants if any).

We shall exemplify this with three classes of spaces discussed in Subsection \ref{W-related-classes}: Gromov hyperbolic spaces, $CAT(0)$ spaces and $\R$-trees.

The theory of $\delta$-hyperbolic spaces, $\mathcal{A}^\omega[X,d,\delta\text{-hyperbolic}]_{-b}$ is an extension of $\mathcal{A}^\omega[X,d]_{-b}$ defined as follows: 
\begin{enumerate}
\item add a constant $\delta_\R$ of type $\N\to\N$ (representing the nonnegative real $\delta$);
\item add the axioms:
$$\begin{array}{l}
\hspace*{-0.5cm}\delta_\R \ge_\Real 0_\Real,\\[0.2cm]
\hspace*{-0.5cm}\forall x,y,z,w\in X\,\bigg(d_X(x,y)+_\Real d_X(z,w)\leq_{\Real} \max_\Real\{d_X(x,z)+_\Real d_X(y,w),d_X(x,w)+_\Real d_X(y,z)\}\\
\hfill+_\Real 2\cdot_\Real\delta_R\bigg). 
\end{array}$$
\end{enumerate}

The notion that a sentence of ${\cal L}(\mathcal{A}^\omega[X,d,\delta\text{-hyperbolic}]_{-b})$ holds in a nonempty $\delta$-hyperbolic space $(X,d)$ is defined as in Definition \ref{habil-model-metric+W-hyp}, by interpreting the new constant $\delta_\R$ as $(\delta)_0$.

Since $\leq_\Real$ is $\Pi_1^0$, the two axioms are $\forall$-sentences. Thus, in order to adapt Theorem \ref{habil-logic-main-thm}  to the theory of Gromov $\delta$-hyperbolic spaces, we need to show that the new constant $\delta_\R$ is strongly majorizable. It is easy to see that if $(X,d)$ is a $\delta$-hyperbolic space, and $k\in\N$ is such that $k\ge \delta$, then 
\[\delta_\R^*:=\lambda n.j(k\cdot 2^{n+2}, 2^{n+1}-1) \,\, \text{s-maj}_1 (\delta)_\circ.\] 

\begin{theorem}\label{habil-main-thm-delta-hyp}
Theorem \ref{habil-logic-main-thm} holds also for $\mathcal{A}^\omega[X,d, \delta\text{-hyperbolic}]_{-b}$ and nonempty  Gromov $\delta$-hyperbolic spaces $(X,d)$, with the bound $\Phi$ depending additionally on  $k\in\N$ such that $k\geq\delta$.
\end{theorem}

Let us consider the case of $CAT(0)$ spaces. As we have seen in Subsubsection \ref{habil-CAT0}, we can define the theory ${\cal A}^\omega[X,d,W,CAT(0)]_{-b}$  for $CAT(0)$ spaces  by adding to $\aomegaxdw$ the formalized form of the $CN^-$ inequality, which is a $\forall$-sentence.
\bua
\forall x,y,z\in X\left(
d_X\left(z,W_X\left(x,y,\frac 12\right)\right)^2\leq_\R\frac12 d_X(z,x)^2+_\R\frac12 d_X(z,y)^2-_\R \frac14 d_X(x,y)^2\right).
\eua

\begin{theorem}\label{habil-main-thm-CAT0}
Theorem \ref{habil-logic-main-thm} holds for ${\cal A}^\omega[X,d,W,CAT(0)]_{-b}$ and nonempty $CAT(0)$ spaces. 
\end{theorem}

Following Proposition \ref{habil-char-R-trees}, the theory $\mathcal{A}^\omega[X,d,W,\R\text{-tree}]_{-b}$ of $\R$-trees results from the theory $\mathcal{A}^\omega[X,d,W]_{-b}$ by adding a $\forall$-axiom:
\[
\forall x,y,z,w\in X\bigg(d_X(x,y)+_\Real d_X(z,w)\leq_\Real \max_\Real\{d_X(x,z)+_\Real d_X(y,w),d_X(x,w)+_\Real d_X(y,z)\}\bigg).
\]

As a consequence

\begin{theorem}\label{habil-main-thm-R-tree}
Theorem \ref{habil-logic-main-thm} holds also for $\mathcal{A}^\omega[X,d,W,R\text{-tree}]_{-b}$ and nonempty $\Real$-trees. 
\end{theorem}

\subsection{Logical metatheorems for $UCW$-hyperbolic spaces}

In the sequel, we shall see that the logical metatheorem from the previous subsection can be easily adapted to $UCW$-hyperbolic spaces (see \cite{Leu06}).

The theory ${\cal A}^{\omega}[X,d,UCW,\eta]_{-b}$, corresponding to the class of  $UCW$-hyperbolic spaces is obtained from ${\cal A}^{\omega}[X,d,W]_{-b}$ by adding a new constant $\eta_X:\N\to\N
\to\N$ together with axioms
\[\ba{l}
\forall r,k\in\N\forall x,y,a\in X \biggl(d_X(x,a) <_\R r \,\wedge\, d_X(y,a)<_\R r\,\\[0.1cm]
\quad \quad \quad\quad \quad \quad\quad \quad \quad \wedge\, \displaystyle d_X(W_X(x,y,1/2),a)>_\R \displaystyle \left(1-_\R 2^{-\eta_X(r,k)}\right)\cdot_\R r
\rightarrow d_X(x,y)\leq_\R 2^{-k}\cdot_\R r\biggr), \\[0.1cm]
\forall r_1,r_2,k\in\N\big(r_1\leq_\Q r_2\rightarrow \eta(r_1,k)\geq_0 \eta(r_2,k)\big),\\[0.2cm]
\forall r,k\in\N\big(\eta_X(r,k)=_0\eta_X(c(r),k)\big). 
\ea
\]
The first two  axioms express the fact that $\eta_X:\N\to\N\to\N$ represents a monotone modulus of uniform continuity. The meaning of the third axiom is that $\eta_X$ is a function having the first argument a rational number on the level of codes; $c$ is the canonical representation for rational numbers defined by (\ref{rep-rational-canonical}).
It is easy to see, using the representation of real numbers in $\aomega$, that all the three axioms are universal. Moreover, the constant $\eta_X$ of degree $1$ is majorizable. 

The notion that a sentence of ${\cal L}(\mathcal{A}^\omega[X,d,UCW,\eta]_{-b})$ holds in a nonempty $UCW$-hyperbolic space $(X,d,W)$ with monotone modulus of uniform convexity $\eta$ is defined as above, by interpreting the new constant $\eta_X$ as $\eta_X(r,k):=\eta(c(r),k)$. 

Since $\mathcal{A}^\omega[X,d,UCW,\eta]_{-b}$ results from $\aomegaxdw$ by adding a majorizable constant and three $\forall$-axioms, we get that the logical metatheorem and its corollaries hold in this setting too.

\begin{theorem}\label{intro-main-thm-uc-hyp}
Theorem \ref{habil-logic-main-thm} holds for $\mathcal{A}^\omega[X,d,UCW,\eta]_{-b}$ and nonempty $UCW$-hyperbolic spa-ces $(X,d,W)$ with monotone modulus of uniform convexity $\eta$, with the bound $\Phi$ depending additionally on $\eta$.
\end{theorem}

\bcor\label{meta-BRS-UCW}
Corollary \ref{meta-BRS} holds also for for $\mathcal{A}^\omega[X,d,UCW,\eta]_{-b}$ and nonempty $UCW$-hyperbolic spaces $(X,d,W)$ with monotone modulus of uniform convexity $\eta$, with the bound $\Phi$ depending additionally on $\eta$.
\ecor

\begin{corollary}\label{meta-Groetsch}
Let $P$ be $\N$ or $\N^\N$, $K$ be a $\aomega$-definable compact metric space, $B_\forall(\ul{u},\ul{y},x,T,n)$ be a $\forall$-formula and $C_\exists(\ul{u},\ul{y},x,T,N)$ a  $\exists$-formula. Assume that $\mathcal{A}^\omega[X,d,UCW,\eta]_{-b}$ proves that
\[\ba{l}
\hspace*{-0.5cm}\forall\, \ul{u}\in P\,\forall \,\Omega:\N\to\N\,\forall\, \ul{y}\in K \, \forall \, x\in X\,\forall\, T:X\to X\\
\quad\quad \bigg( T  \text{ is }   x\text{-majorizable with modulus } \Omega\, \si\, Fix(T)\ne\emptyset \si  \forall n\in\N \, B_\forall\,\ra \,\exists N\in\N \, C_\exists\bigg).
\ea\]
Then one can extract a computable functional $\Phi$ such that for all $b\in\N$, 
\[\ba{l}
\forall\,  \ul{u}\in P\,\forall \,\Omega:\N\to\N\,\forall\, \ul{y}\in K \, \forall \, x\in X\,\forall\, T:X\to X \\
\qquad\bigg( T  \text{ is }   x\text{-majorizable with modulus } \Omega\,\si\, \forall \delta>0\big(Fix_\delta(T,x,b)\ne\emptyset\big) \si  \forall n\le \Phi(\ul{u},b,\eta,\Omega) B_\forall\,\\
\qquad\qquad\qquad\qquad\qquad\qquad\ra \,\exists N\le \Phi(\ul{u},b,\eta,\Omega) \, C_\exists\bigg).
\ea\]
holds in any nonempty  $UCW$-hyperbolic space $(X,d,W)$ with monotone modulus of uniform convexity $\eta$. We recall that 
\[Fix_{\delta}(T,x,b):=\{y\in X\mid d(y,x)\leq b \text{~and~} d(y,Ty)<\delta\}.\]
\end{corollary}
\begin{proof}
The statement proved in $\mathcal{A}^\omega[X,d,UCW,\eta]_{-b}$ can be written as
\[\ba{l}
\forall\,  \ul{u}\in P\,\forall \,\Omega:\N\to\N\,\forall\, \ul{y}\in K \, \forall \, x,p\in X\,\forall\, T:X\to X \\
\bigg( T \text{ is }  x\text{-majorizable with modulus } \Omega  \, \si\, \forall k\in\N\left(d(p,Tp)\le 2^{-k}\right) \si  \forall n\in\N \, B_\forall\,\ra \,\exists N\in\N \, C_\exists\bigg).
\ea\]
We have used the fact that  $Fix(T)\ne\emptyset$ is equivalent with $\exists p\in X(Tp=_Xp)$ that is further equivalent with $\exists p\in X\,\forall \,k\in\N\big(d(p,Tp)\le 2^{-k}\big)$, by using the definition of $=_X$ and $=_\R$ in our system. 
As all the premises are $\forall$-formulas, we can apply Corollary \ref{meta-BRS-UCW} to extract a functional $\Phi$ such that for all $b\in\N$,
\[\ba{l}
\hspace*{-0.5cm}\forall\,  \ul{u}\in P\,\forall \,\Omega:\N\to\N\,\forall\, \ul{y}\in K \, \forall \, x,p\in X\,\forall\, T:X\to X \\
\quad \bigg(T \, x\text{-maj. w. modulus } \Omega \,\si\, d(x,p)\le b \, \si\,\forall k\le \Phi(\ul{u},b,\eta,\Omega)\big(d(p,Tp)\le 2^{-k}\big)\\ 
\quad\quad \quad\quad \si \, \forall n\le \Phi(\ul{u},b,\eta,\Omega) \, B_\forall\,\ra \,\exists N\le \Phi(\ul{u},b,\eta,\Omega)\, C_\exists\bigg),
\ea\]
that is
\[\ba{l}
\hspace*{-0.5cm}\forall\,  \ul{u}\in P\,\forall \,\Omega:\N\to\N\,\forall\, \ul{y}\in K \, \forall \, x\in X\,\forall\, T:X\to X \\
\quad \bigg(T \, x\text{-maj. w. modulus } \Omega \,\si\, \exists p\in X\bigg(d(x,p)\le b \, \si\,\forall k\le \Phi(\ul{u},b,\eta,\Omega)\big(d(p,Tp)\le 2^{-k}\big)\bigg)\\ 
\quad\quad \quad\quad \si \, \forall n\le \Phi(\ul{u},b,\eta,\Omega) \, B_\forall\,\ra \,\exists N\le \Phi(\ul{u},b,\eta,\Omega)\, C_\exists\bigg),
\ea\]
Use the fact that  the existence of $p\in X$ such that $d(x,p)\le b$ and $\forall k\le \Phi\big(d(p,Tp)\le 2^{-k}\big)$ is equivalent with the existence of $p\in X$ such that $d(x,p)\le b$ and $d(p,Tp)\le 2^{-\Phi}$ which is obviously implied by $\forall \delta>0\left(Fix_\delta(T,x,b)\ne\emptyset\right)$.
\end{proof}

We shall apply the above corollary twice. The first application will be in Subsection \ref{habil-app-fpt-CAT0} for nonexpansive mappings $T$. As we have already discussed, if $T$ is nonexpansive, then its modulus of majorizability at $x$ is simply $\Omega(n)=n+\tilde{b}$ with $\tilde{b}\ge d(x,Tx)$. 

For all $\delta>0$ there exists $y\in X$ such that $Fix_\delta(T,x,b)\ne\emptyset$, hence
$$d(x,Tx)\le d(x,y)+d(y,Ty)+d(Ty,Tx)\le 2d(x,y)+d(y,Ty)\le 2b+\delta \quad \text{for all } \delta >0.$$
It follows that $d(x,Tx)\le 2b$, so we can take $\tilde{b}:=2b$. As a consequence, the bound $\Phi$ will depend only on $\ul{u},b$ and $\eta$. 

The second application will be in Subsection \ref{habil-app-fpt-as-ne}, this time for asymptotically nonexpansive mappings. As we have discussed in Subsection \ref{intro-as-ne}, an asymptotically nonexpansive mapping $T:X\to X$ with sequence $(k_n)$ is a $(1+K)$-Lipschitz mapping, where $K\in \N$ is such that $k_1\le K$. By Lemma \ref{habil-classes-majorizable}, we get that $T$ is majorizable with modulus at $x$ given by $\Omega(n):=n+(1+K)\tilde{b}$, where again $\tilde{b}\ge d(x,Tx)$.  Reasoning as above, it is easy to see that if $b$ is such that $Fix_\delta(T,x,b)\ne\emptyset$ for all $\delta>0$ , then we can take $\tilde{b}:=(2+K)b$. Thus, the bound $\Phi$ depends on $\ul{u}, b,\eta$ and on $K\in\N$ with $K\ge k_1$.

\section{Proof mining in metric fixed point theory}\label{proof-mining-fpt}

In the sequel, $(X,d,W)$ is a $W$-hyperbolic space, $C\subseteq X$ a convex subset of $X$, and $T:C\to C$ a nonexpansive mapping. 

As in the case of normed spaces,  we can define the {\em Krasnoselski-Mann iteration} starting from $x\in C$ by
\begin{equation}
x_0:=x, \quad x_{n+1}:=(1-\lambda_n)x_n \oplus\lambda_n Tx_n, \label{app-KM-lambda-n-def-hyp}
\end{equation}
where $(\lambda_n)$ is a sequence in $[0,1]$. For constant $\lambda_n=\lambda \in(0,1)$, we get the {\em Krasnoselski iteration}, which can be also  defined as the Picard iteration $\big(T_\lambda^n(x)\big)$ of
$$T_\lambda:C\to C, \quad T_\lambda(x)=(1-\lambda)x\oplus\lambda Tx.$$ 
The {\em averaged mapping} $T_\lambda$ is also nonexpansive and  $Fix(T)=Fix(T_\lambda)$.

The following proposition collects some useful properties of Krasnoselski-Mann iterates in $W$-hyperbolic spaces. We refer to \cite{KohLeu03,Leu08} for the proofs.

\bprop\label{W-useful-KM-x-y}
Let $(x_n),(x^*_n)$ be the Krasnoselski-Mann iterations starting with $x,x^*\in C$. Then 
\be
\item\label{W-KM-xn-yn-nonincreasing} $(d(x_n,x^*_n))$ is nonincreasing;
\item\label{W-KM-xn-Txn-nonincreasing} $(d(x_n,Tx_n))$ is nonincreasing;
\item\label{W-KM-xn-fp-nonincreasing} $(d(x_n,p))$ is nonincreasing for any fixed point $p$ of $T$.
\ee
\eprop

The following very useful result was proved by Goebel and Kirk \cite{GoeKir83} for spaces of hyperbolic type, thus holds for $W$-hyperbolic spaces too.

\bthm\label{habil-GK-gen-thm}
Let $(X,d, W)$ be a $W$-hyperbolic space and $(\lambda_n)$ be a sequence in $[0,1]$ which is divergent in sum and bounded away from $1$. Assume that $(u_n),(v_n)$ are sequences in $X$ satisfying  for all $n\in\N$,
\beq
u_{n+1}=(1-\lambda_n)u_n\oplus\lambda_nv_n \quad \text{and} \quad d(v_n, v_{n+1})\leq d(u_n, u_{n+1}).
\eeq
Then $(d(u_n,v_n))$ is nonincreasing and $\ds \limn d(u_n,v_n)=0$ whenever $(u_n)$ is bounded.
\ethm

As an immediate consequence of the above theorem, we get the generalization of Theorem \ref{asreg-Ishikawa-bounded-xn} to $W$-hyperbolic spaces.

\begin{theorem}\label{asreg-Ishikawa-bounded-xn-W}
Let $C$ be a convex subset of a $W$-hyperbolic space $(X,d,W)$ and $T:C\rightarrow C$ a nonexpansive mapping. Assume that $(\lambda_n)$ is divergent in sum and bounded away from $1$. 

If there exists $x^*\in C$ such that $(x^*_n)$ is bounded, then $T$ is $\lambda_n$-asymptotically regular, that is $\limn d(x_n,Tx_n)=0$ for all $x\in C$. 
\end{theorem}
\begin{proof}
Since $(d(x_n,x^*_n))$ is nonincreasing, we get that $(x_n)$  is bounded for all $x\in C$. Apply now Theorem \ref{habil-GK-gen-thm} with $u_n:=x_n$ and $v_n:=Tx_n$. 
\end{proof}

As a consequence, we get that for bounded convex $C$, any nonexpansive self-mapping of $C$ is approximately fixed.
  
\bcor\label{habil-bounded-hyperbolic-afpp-ne}
Bounded convex subsets of $W$-hyperbolic spaces have the AFPP for nonexpansive mappings.
\ecor
 
As we have already remarked in Section \ref{intro-fpt-as-reg}, Theorem \ref{asreg-uniform-Goebel-Kirk}, unifying Ishikawa's and Edelstein/O'Brien's results is valid in spaces of hyperbolic type, hence in $W$-hyperbolic spaces too.

\bthm\label{asreg-uniform-Goebel-Kirk-W}\cite{GoeKir83}
Let $C$ be a bounded convex subset of a  $W$-hyperbolic space $(X,d,W)$ and $\lambda_n$ divergent in sum and bounded away from $1$. Then for every $\eps >0$ there exists a positive integer $N$ such that for all $x\in C$ and all $T:C\to C$ nonexpansive,
\beq
\forall n\geq N \big(d(x_n,Tx_n) <\eps \big).
\eeq 
\ethm

\subsection{A quantitative version of Borwein-Reich-Shafrir Theorem}\label{section-app-fpt-brs}

Our first application of proof mining is an effective quantitative version of the following theorem due to Borwein, Reich and Shafrir \cite{BorReiSha92}.
 
\bthm \label{BRS-thm-hyperbolic}\cite{BorReiSha92}
Let $C$ be a convex subset of a $W$-hyperbolic space $(X,d,W)$  and $T:C\rightarrow C$ a nonexpansive mapping. Assume that $(\lambda_n)$ is divergent in sum and bounded away from $1$.

Then for all $x\in C$, 
\beq
\lim d(x_n,Tx_n)=r_C(T).
\eeq
\ethm
We recall that $r_C(T)=\inf\{d(x,Tx)\mid x\in C\}$ is the minimal displacement of $T$. As we have already remarked in Section \ref{intro-fpt-as-reg}, the above theorem was initially proved for normed spaces. Anyway, it is easy to see that its proof holds also in the more general context of $W$-hyperbolic spaces.

In the following, we give an explicit quantitative version of the above theorem, generalizing to $W$-hyperbolic spaces and directionally nonexpansive mappings the logical analysis made by Kohlenbach \cite{Koh01,Koh03} for normed spaces and nonexpansive functions. Our Theorem \ref{habil-quant-BRS-bounded} extends Kohlenbach's results (even with the same numerical bounds)  to  $W$-hyperbolic spaces and directionally nonexpansive mappings and contains all previously known results of this kind as special cases. In this way, we obtain significantly stronger and much more general forms of Kirk's  Theorem \ref{asreg-uniform-Kirk-dir-ne} with  explicit bounds. As a special feature of our approach, which is based on logical analysis instead of functional analysis, no functional analytic embeddings are needed to obtain our uniformity results. 

The main application of the quantitative version of the Borwein-Reich-Shafrir Theorem is a  {\em uniform} effective rate of $\lambda_n$-asymptotic regularity  in the case of bounded $C$ for  general $(\lambda_n)$  divergent in sum and bounded away from $1$ (see Theorem \ref{habil-quant-BRS-bounded}). Thus, the rate of asymptotic regularity is uniform in the nonexpansive mapping $T:C\to C$ and in the starting point $x\in C$ of the Krasnoselski-Mann iteration $(x_n)$ and  in the bounded convex subset $C$ (by this we mean that it depends on $C$ only via its diameter $d_C$). 

As we have already discussed in Section \ref{intro-fpt-as-reg}, uniformity in $x\in C$ for Banach spaces and constant $\lambda_n=\lambda$ was first established by Edelstein and O'Brien in \cite{EdeOBr78}. Subsequently, in \cite{GoeKir83}, Goebel and Kirk  obtained uniformity in $x$ and $T$ for general $(\lambda_n)$, but no uniformity in $C$;  their result holds even for spaces of hyperbolic type. In 2000 \cite{Kir00}, Kirk  established uniformity in $x,T$ for Banach  spaces and directionally nonexpansive mappings only in the case of constant $\lambda_n=\lambda$. In 2001 \cite{Koh01}, by using methods of proof mining,  Kohlenbach  obtained for the first time uniformity in $x,T,C$ for nonexpansive mappings and general $(\lambda_n)$ in the case of Banach spaces with explicit rates of asymptotic regularity. 

None of the papers \cite{Ish76,EdeOBr78,GoeKir83,BorReiSha92,Kir00} contain any bounds and in fact \cite{EdeOBr78,GoeKir83,Kir00} use non-trivial functional theoretic embeddings to get the uniformities. Kirk and Martinez-Yanez \cite[p.191]{KirMar90} explicitly mention the non-effectivity of all these results and state that {\em "it seems unlikely that such estimates would be easy to obtain in a general setting''} and, therefore, only study the {\em tractable} case of uniformly convex Banach spaces. 

Not even the ineffective existence of bounds uniform in $C$ was known for general $(\lambda_n)$ and still in 1990, Goebel and Kirk conjecture \cite[p. 101]{GoeKir90} as {\em ``unlikely'' to be true}. Only for Banach spaces and constant $\lambda_n=\lambda$, uniformity with respect to $C$ has been established by Baillon and Bruck in \cite{BaiBru96}, where for this special case an optimal quadratic bound was obtained.

\subsubsection{Logical discussion}\label{BRS-logical}

The proof of Theorem \ref{BRS-thm-hyperbolic} is prima facie ineffective and does not provide any rate of convergence of $(d(x_n,Tx_n))$. Moreover, its statement does not have the required logical form for the logical metatheorems from Section \ref{logical-meta} to apply, due to the two implicative assumptions on $(\lambda_n)$ and, more seriously, to the existence of $r_C(T)$, which can not be formed in the theory $\mathcal{A}^\omega[X,d,W]_{-b}$ of $W$-hyperbolic spaces. 

However, we show in the sequel that it can be reformulated in such  a way that the logical metatheorems apply (more precisely Corollary \ref{meta-BRS}). Firstly, let us remark that any convex subset $C$ of a $W$-hyperbolic space is also a $W$-hyperbolic space, so it suffices to consider only the case $C=X$, and hence only nonexpansive functions $T:X\to X$. For simplicity, we shall denote $r_X(T)$ with $r(T)$.
 
Let us consider the conclusion of Theorem \ref{BRS-thm-hyperbolic}. 

\bprop
The following are equivalent for all $x\in X$.
\be
\item $\limn d(x_n,Tx_n)=r(T)$;
\item $\forall\,\varepsilon >0\,\exists N\in\N\,\forall\, m\ge N\big( d(x_m,Tx_m) < r(T)+\varepsilon\big)$;
\item $\forall\,\varepsilon >0\,\exists N\in\N\,\forall\, m\ge N\,\forall\, x^*\in X\big( d(x_m,Tx_m) <  d(x^*,Tx^*)+\varepsilon\big)$;
\item $\forall\,\varepsilon >0\,\exists N\in\N\,\forall\, x^*\in X\big( d(x_N,Tx_N) <  d(x^*,Tx^*)+\varepsilon\big)$;
\item $\forall\,\varepsilon >0\,\forall\, x^*\in X\,\exists N\in\N\big( d(x_N,Tx_N) <  d(x^*,Tx^*)+\varepsilon\big)$.
\ee
\eprop
\begin{proof}$\,$
$(i)\Lra(ii)\Lra(iii)$ are obvious, by the definition of $r(T)$. \\
$(iii)\Lra(iv)$ follows immediately from the fact that $( d(x_n,Tx_n))$ is nonincreasing, hence the quantifier $\forall\,  m\ge N$ in (iii) is superfluous.\\
$(iv)\Ra(v)$ is obvious, so it remains to prove $(v)\Ra(iv)$ Since $r(T)=\inf\{ d(x^*,Tx^*):x^*\in X\}$, there exists $y^*\in X$ such that $ d(y^*, Ty^*)< r(T)+\eps/2$. Applying (v) with $\eps/2$ and $y^*$, we get $N\in\N$ such that $ d(x_N,Tx_N) < d(y^*,Ty^*)+\varepsilon/2<r(T)+\eps\leq  d(x^*,Tx^*)+\eps$ for all $x^*\in X$. Thus, (iv) is satisfied with this $N$.
\end{proof}
Thus, the conclusion $\forall\, x\in X\left(\limn d(x_n,Tx_n)=r(T)\right)$ of Borwein-Reich-Shafrir Theorem can be reformulated as 
$$\forall\, x\in X\,\forall\,\varepsilon >0\,\forall\, x^*\in X\,\exists N\in\N\bigg( d(x_N,Tx_N) <  d(x^*,Tx^*)+\varepsilon\bigg),$$ that has the $\forall\,\exists$-form required by the logical metatheorems.

Let us now examine the hypotheses on $(\lambda_n)$: $\ds \limsup \lambda_n<1$ and  $\ds\sum_{n=0}^\infty \lambda_n=\infty$. 

The first one, $\limsup \lambda_n<1$, states the existence of a $K\in\N^*$ such that $\ds\lambda_n\leq 1-\frac1K$ for all $n$ from some index $N_0$ on. Since $N_0$ only contributes an additive constant to our bound, we may assume for simplicity that $N_0=0$, which is anyway the case if $(\lambda_n)$ is  a sequence in $[0,1)$. Hence, we may replace the hypothesis $\limsup \lambda_n<1$ with 
\beq
\exists K\in\N\,\forall\, n\in\N\left(\lambda_n\leq 1-\frac1K\right).\label{ref-intro-lambda-n-limsup}
\eeq

The second one, $\ds\sum_{n=0}^\infty \lambda_n=\infty$, is (ineffectively, using countable axiom of choice) equivalent with 
\beq
\exists \,\theta:\N\to\N\,\forall\, n\in\N\left(\ds\sum_{s=0}^{\theta(n)}\lambda_s\geq n\right),\label{ref-intro-lambda-n-div}
\eeq
that is with the existence of a rate of divergence $\theta:\N\to\N$.

It is easy to see that $\adw$ proves the following formalized version of Theorem \ref{BRS-thm-hyperbolic}:
\[\ba{l}
\forall\, (\lambda_n)\in[0,1]^\N\,\forall\, x\in X\, \forall\,\, T:X\to X\\
\biggl(T \text{~nonexpansive~} \si \,\exists K\in\N\,\forall\, n\in\N\ds\left(\lambda_n\leq 1- \frac1K\right) \, \si\,\exists \,\theta:\N\to\N\,\forall\, n\in\N\left(\ds\sum_{s=0}^{\theta(n)}\lambda_s\geq n\right)\\
\hfill\ra \,\forall\, \eps>0\,\forall\, x^*\in X\,\exists N\in \N\bigg(d(x_N,Tx_N) < d(x^*,Tx^*)+ \eps\bigg)\biggr),
\ea\]
hence,
\[\ba{l}
\!\!\forall\, K\in\N\, \forall\, \eps>0\, \forall\, \theta:\N\to\N \,\forall\, (\lambda_n)\in[0,1]^\N\,\forall\, x,x^*\in X\, \forall\, T:X\to X\quad\quad\\
\quad \quad\quad \quad\quad \quad \hfill\biggl(T \text{~nonexpansive~} \si \,\forall\, n\in\N\ds \left(\lambda_n\leq 1- \frac1K\right)\, \si \,\forall\, n\in\N\left(\ds\sum_{s=0}^{\theta(n)}\lambda_s\geq n\right)\\
\hfill\ra \, \exists N\in \N\bigg(d(x_N,Tx_N) < d(x^*,Tx^*)+\eps\bigg)\biggr),
\ea\]
The Hilbert cube $[0,1]^\N$ is a compact metric space which is ${\cal A}^\omega$-definable and we can let $\eps=2^{-p}$ with $p\in\N$, hence the above formalization of the statement of Borwein-Reich-Shafrir Theorem has the required logical form. 

Corollary \ref{meta-BRS} yields the existence of a computable functional $\Phi$ such that for all $b\in\N$, 
\[\ba{l}
\!\!\forall\,K\in\N\, \forall\, \eps>0\, \forall\, \theta:\N\to\N\, \forall (\lambda_n)\in[0,1]^\N\,\forall\, x,x^*\in X\, \forall\,T:X\to X\\
\biggl(T \text{ n.e. } \, \si \, d(x,Tx)\leq b \, \si \, d(x,x^*)\leq b \, \si \,\forall\, n\in\N\ds\left(\lambda_n\leq 1- \frac1K\right)\, \si \,\forall\, n\in\N\left(\ds\sum_{s=0}^{\theta(n)}\lambda_s\geq n\right) \\
\hfill\ra \,\exists N\leq \Phi(\eps,b, K,\theta)\bigg(d(x_N,Tx_N) < d(x^*,Tx^*)+ \eps\bigg)\biggr).
\ea\]
holds in any $W$-hyperbolic space $(X,d,W)$; "n.e." abbreviates "nonexpansive". Using again the fact that $(d(x_n,Tx_n))$ is nonincreasing, we get in fact that 
\[\forall n\geq \Phi(\eps,b, K,\theta)\bigg(d(x_n,Tx_n) < d(x^*,Tx^*)+ \eps\bigg).\]

In fact, a slight reformulation of  (\ref{ref-intro-lambda-n-div}) is better suited for the proof of our theorem.

\blem\label{BRS-lambda-rate-divergence}
The following are equivalent:
\be
\item  there exists $\theta:\N\to\N$ such that $\ds\sum_{s=0}^{\theta(n)}\lambda_s\geq n$ for all $n\in\N$;
\item there exists  $\gamma:\N\times \N\to\N$ such that $\ds\sum_{s=i}^{i+\gamma(i,n)-1}\lambda_s \geq n$ for all $n,i\in\N$;
\item there exists $\alpha:\N\times\N\to\N$ such that for all $n,i\in\N$,
\beq
\alpha(i,n)\le\alpha(i+1,n) \text{~~and~~} \ds\sum_{s=i}^{i+\alpha(i,n)-1}\lambda_s \geq n. \label{habil-BRS-alpha}
\eeq
\ee
\elem
\begin{proof} 
$(i)\Ra(ii)$ Define $\gamma(i,n)=\theta(n+i)-i+1\geq 0$, since $n+i\leq  \ds\sum_{s=0}^{\theta(n+i)}\lambda_s \le \theta(n+i)+1$. Furthermore,
\bua
\ds\sum_{s=i}^{i+\gamma(i,n)-1}\lambda_s=\ds\sum_{s=i}^{\theta(n+i)}\lambda_s=\ds\sum_{s=0}^{\theta(n+i)}\lambda_s-\ds\sum_{s=0}^{i-1}\lambda_s\geq n+i-i=n, \quad\text{as~} \ds\sum_{s=0}^{i-1}\lambda_s\leq i.
\eua
$(ii)\Ra(iii)$ Define $\ds \alpha(i,n)=\max_{j\leq i}\{\gamma(j,n)\}.$
 Then $\alpha$ is increasing in $i$, $\alpha(i,n)\ge \gamma(i,n)$, so 
$\ds\sum_{s=i}^{i+\alpha(i,n)-1}\lambda_s\geq  \ds\sum_{s=i}^{i+\gamma(i,n)-1}\lambda_s\geq n$.\\
$(iii)\Ra(i)$ Applying (iii) with $i=0$, we get that $n\leq \ds\sum_{s=0}^{\alpha(0,n)-1}\lambda_s\le \alpha(0,n)$, so $\alpha(0,n)-1\geq  n-1\ge 0$ for all $n\geq 1$. We can define then $\theta(n)=\alpha(0,n)-1$ for $n\geq 1$ and $\theta(0)$ arbitrary. 
\end{proof}

Hence, Corollary \ref{meta-BRS} guarantees the extractability of a computable functional $\Phi$ such that
for all $b\in\N$, 
\[\ba{l}
\!\!\forall\, K\in\N\, \forall\, \eps>0\, \forall\, \alpha:\N\times\N\to\N\, \forall (\lambda_n)\in[0,1]^\N\,\forall\, x,x^*\in X\, \forall\,T:X\to X\\
\biggl(T \text{ n.e.} \, \si \, d(x,Tx)\leq b \, \si \, d(x,x^*)\leq b \, \si \,\forall\, n\in\N\ds\left(\lambda_n\leq 1- \frac1K\right)\, \si\, \alpha \text{ satisfies (\ref{habil-BRS-alpha})}\\
\hfill \ra \,\forall n\geq \Phi(\eps,b, K,\alpha)\bigg(d(x_n,Tx_n) < d(x^*,Tx^*)+ \eps\bigg)\biggr).
\ea\]

An explicit such bound $\Phi$ has been extracted by Kohlenbach and the author in \cite{KohLeu03} and will be given in the following. \\

\subsubsection{Main results}

We present now the quantitative version of the Borwein-Reich-Shafrir Theorem. 

\bthm\label{habil-quant-BRS} Let $K\in \N, K\geq 1$, $\alpha:\N\times\N\to N$ and $b>0$.\\
 Then for all $W$-hyperbolic spaces $(X,d,W)$, for all convex subsets $C\se X$, \\
 for all sequences $(\lambda_n)$ in $\left[0,1-1/K\right]$ satisfying
\beq
\forall\, i,n\in\N \left((\alpha(i,n)\le\alpha(i+1,n))\quad\text{~and~}\quad 
n\le\sum\limits^{i+\alpha(i,n)-1}_{s=i}\lambda_s \right),\label{habil-quant-BRS-hyp-alpha}
\eeq
for all $x,x^*\in C$ and for all nonexpansive mappings $T:C\to C$ such that
\beq
d(x,Tx)\leq b \quad\text{~and~}\quad d(x,x^*)\leq b,\label{habil-quant-BRS-hyp-x-b} 
\eeq
the following holds
\beq
\aquant \varepsilon >0 \aquant n\ge
\Phi(\varepsilon,b,K,\alpha)\bigg(d(x_n,Tx_n)< d(x^*,Tx^*)
+\varepsilon\bigg),
\eeq
where $\Phi(\varepsilon,b,K,\alpha)=\widehat{\alpha}(\lceil
2b\cdot \exp(K(M+1)) \rceil\remin\, 1,M),$
with 
\[ \ba{l} n\, \remin 1=\max\{0,n-1\}, \quad \ds M=\left\lceil \frac{1+2b}{\varepsilon}\right\rceil, \quad \widehat{\alpha},\tilde{\alpha}:\N\times\N\to \N\\
\widehat{\alpha}(0,n)=\tilde{\alpha}(0,n), \quad 
\widehat{\alpha}(i+1,n)=\tilde{\alpha}(\widehat{\alpha}(i,n),n), \quad 
\tilde{\alpha}(i,n)=i+\alpha(i,n). \ea\]
\ethm

\bfact
As we have seen in Lemma \ref{BRS-lambda-rate-divergence}, we could have started with a rate of divergence $\theta:\N\to\N$ for $\ds\sum_{n=0}^\infty \lambda_n$  and then define $\ds \alpha(i,n)=\max_{j\le i}\big\{\theta(n+j)-j+1\big\}$. Starting with $\theta$ would in general give less good bounds than when working with $\alpha$ directly, as it can be seen  in \cite[Remark 3.19]{KohLeu03}.
\efact

The above theorem was proved for normed spaces and nonexpansive mappings by Kohlenbach \cite{Koh01}. For $W$-hyperbolic spaces,  it was obtained by Kohlenbach and the author in \cite{KohLeu03} as a consequence of an extension to the more general class of directionally nonexpansive mappings 

As we have seen in Section \ref{intro-fpt-as-reg}, the notion of directionally nonexpansive mapping was introduced by Kirk \cite{Kir00} in the context of normed spaces, but $W$-hyperbolic spaces in our sense suffice: 
\bce
 $T:C\to C$ is {\em directionally nonexpansive} if $d(Tx,Ty)\leq d(x,y)$ for all $x\in C$ and  all $y\in[x,Tx]$.
 \ece
\noindent Obviously, any nonexpansive mapping is directionally nonexpansive, but the converse fails as directionally nonexpansive  mappings not even need to be continuous on the whole space, as it can be seen from the following example.

\bex (simplified by Paulo Oliva): Consider the normed space $(\R^2,\|\cdot\|_{\max})$ and the mapping 
\[ T:[0,1]^2\rightarrow [0,1]^2, \ T(x,y)=\left\{\ba{l} 
(1,y), \ \mbox{if $y>0$} \\ (0,y), \ \mbox{if $y=0$}. \ea \right. \] 
Clearly, $T$ is directionally nonexpansive, but discontinuous at $(0,0)$, hence $T$ is not nonexpansive.
\eex

Since $x_{n+1}\in[x_n, Tx_n]$, we have  that $d(Tx_n,Tx_{n+1})\leq d(x_n,x_{n+1})$ for directionally nonexpansive mappings too, so we can apply Goebel-Kirk Theorem \ref{habil-GK-gen-thm} to get that $(d(x_n,Tx_n))$ is nonincreasing  and to obtain the following generalization of Ishikawa Theorem \ref{asreg-Ishikawa}.

\begin{theorem}\label{asreg-Ishikawa-bounded-xn-W-dir-ne}
Let $C$ be a convex subset of a $W$-hyperbolic space $(X,d,W)$ and $T:C\rightarrow C$ a directionally nonexpansive mapping. Assume that $(\lambda_n)$ is divergent in sum and bounded away from $1$. 

If there exists $x\in C$ such that $(x_n)$ is bounded, then $\limn d(x_n,Tx_n)=0$ for all $x\in C$.
\end{theorem}

As a consequence

\bcor\label{habil-bounded-hyperbolic-afpp-dir-ne}
Bounded convex subsets of $W$-hyperbolic spaces have the AFPP for directionally nonexpansive mappings.
\ecor

We remark that in the case of directionally nonexpansive mappings, the sequence $(d(x_n,x^*_n))$ is not necessarily nonincreasing, so we do not have an analogue of Theorem \ref{asreg-Ishikawa-bounded-xn-W}.

The following is the main result of \cite{KohLeu03}.

\bthm\label{habil-quant-BRS-dir-ne}
Theorem \ref{habil-quant-BRS} holds for directionally nonexpansive mappings too, if the hypothesis $\,d(x,x^*)\leq b$ is strengthened to $d(x_n,x_n^*)\leq b$ for all $n\in\N$.
\ethm
\noindent As we have already remarked,  $(d(x_n,x^*_n))$ is not necessarily nonincreasing for directionally nonexpansive mappings and that's why we need the stronger assumption that $d(x_n,x_n^*)\leq b$ for all $n\in\N$, which is equivalent to $d(x,x^*)\leq b$ in the nonexpansive case, since $(d(x_n,x^*_n))$ is nonincreasing. Thus, Theorem \ref{habil-quant-BRS} is an immediate consequence of Theorem \ref{habil-quant-BRS-dir-ne}. 

Let us note also that as a corollary to Theorem \ref{habil-quant-BRS-dir-ne} we get the following (non-quantitative) generalization of Borwein-Reich-Shafrir Theorem to directionally nonexpansive mappings.

\begin{corollary}\label{habil-BRS-dir-ne-non-quant}
Let $C$  be a convex subset of a $W$-hyperbolic space $(X,d, W)$, $T:C\rightarrow C$ be a directionally nonexpansive mapping, and $(\lambda_n)$ be divergent in sum and bounded away from $1$. \\
Assume $x\in C$ is such that for all $\eps>0$ there exists $x^*\in C$ satisfying 
\beq
d(x_n,x_n^*)\text{ is bounded } \quad\text{ and } \quad d( x^*,Tx^*) \le r_C(T)+
\varepsilon.
\eeq
Then $\ds \limn d(x_n,Tx_n)=r_C(T)$. 
\end{corollary}

Combining Corollaries \ref{habil-bounded-hyperbolic-afpp-dir-ne} and \ref{habil-BRS-dir-ne-non-quant} we get asymptotic regularity for bounded $C$.

\begin{theorem}\label{asreg-KM-bounded-C-dir-ne}
Let $C$ be a bounded convex subset of a $W$-hyperbolic space $(X,d,W)$ and $T:C\rightarrow C$ a directionally nonexpansive mapping. Assume that $(\lambda_n)$ is divergent in sum and bounded away from $1$. 

Then $T$ is $\lambda_n$-asymptotically regular.
\end{theorem}

From Theorem \ref{habil-quant-BRS-dir-ne}, various strong effective uniformity results for the case of bounded $C$ can be derived, as well as for the more general case of bounded $(x_n)$ for some $x\in C$. 

In the case of bounded $C$ with finite diameter $d_C$, the assumptions $d(x,Tx)\leq d_C$ and $d(x_n,x^*_n)\leq d_C$ hold trivially for all $x,x^*\in C$ and all $n\in\N$. The following result is a consequence of Theorema \ref{habil-quant-BRS-dir-ne} and \ref{asreg-KM-bounded-C-dir-ne}.  

\bthm\label{habil-quant-BRS-bounded}$\,$
Let $(X,d, W)$ be a $W$-hyperbolic space, $C\se X$ be a {\em bounded} convex subset with diameter $d_C$, and $T:C\to C$ be directionally nonexpansive.
Assume that $K,\alpha,(\lambda_n)$ are as in the hypothesis of Theorem \ref{habil-quant-BRS}. 

Then $\limn d(x_n,Tx_n)=0$ for all $x\in C$ and, moreover, 
\beq
\aquant \varepsilon>0 \aquant n\ge \Phi(\varepsilon,d_C,K,\alpha)\bigg(d(x_n,Tx_n)< \varepsilon\bigg),
\eeq
where $\Phi(\eps,d_C,K,\alpha)$ is defined as in Theorem \ref{habil-quant-BRS} by replacing $b$ with $d_C$.
\ethm
For bounded $C$, we derive an explicit rate of asymptotic regularity $\Phi(\varepsilon,d_C,K,\alpha)$  depending only on the error $\eps$, on the diameter $d_C$ of $C$, and on $(\lambda_n)$  via $K$ and $\alpha$, but which does not depend on the nonexpansive mapping $T$, the starting point $x\in C$ of the Krasnoselski-Mann iteration or other data related with $C$ and $X$. 

We can simplify the rate of asymptotic regularity further, if we assume that $(\lambda_n)$ is a sequence in $\ds \left[1/K,1-1/K\right]$. In this case, it is easy to see that
\[\alpha:\N\times\N\to \N, \quad \alpha(i,n)=Kn\]
satisfies (\ref{habil-quant-BRS-hyp-alpha}).

\begin{corollary} \label{habil-quant-BRS-bounded-exp}
Let $(X,d, W)$ be a $W$-hyperbolic space, $C\se X$ be a {\em bounded} convex subset with diameter $d_C$, and $T:C\to C$ be directionally nonexpansive.
Let $K\in \N, K\geq 2$ and assume that $\ds \lambda_n\in \left[1/K,1-1/K\right]$ for all $n\in\N$.
Then 
\beq
\aquant \varepsilon >0 \aquant n\ge
\Phi(\varepsilon,d_C,K)\bigg(d(x_n,Tx_n)< \varepsilon\bigg),
\eeq
where $\Phi(\varepsilon,d_C,K)=K\cdot M\cdot\lceil 2d_C\cdot\exp(K(M+1))\rceil$, with $\ds M=\left\lceil\frac{1+2d_C}{\varepsilon}\right\rceil.$
\end{corollary}

Thus, we obtain an exponential (in $1/\eps$) rate of asymptotic regularity. The above corollary is significantly stronger and more general than Kirk  Theorem \ref{asreg-uniform-Kirk-dir-ne}.(ii).

As another consequence of our quantitative version of Borwein-Reich-Shafrir Theorem, we extend, for the case of nonexpansive mappings, Theorem \ref{habil-quant-BRS-bounded} to the situation where $C$ no longer is required to be bounded but only the existence of a point $x^*\in C$ whose iteration sequence $(x^*_n)$ is bounded. In this way, we obtain a quantitative version of Theorem \ref{asreg-Ishikawa-bounded-xn-W}. This is of interest, since the functional analytic embedding techniques from \cite{GoeKir83,Kir00} seem to require that $C$ is bounded, while our proof is a straightforward generalization of Kohlenbach's proof of the corresponding result for normed spaces \cite{Koh03}. 

\bthm \label{habil-quant-BRS-bounded-xn}
Assume that $(X,d, W), C, (\lambda_n),K,\alpha$ are as in the hypothesis of Theorem \ref{habil-quant-BRS} and let $T:C\to C$ be nonexpansive.
Suppose $x,x^*\in C$  and $b>0$ satisfy
\beq
d(x,x^*)\le b\quad \text{and}\quad\aquant n,m\in\N (d(x^*_n,x^*_m)\le b).
\eeq 
Then the following holds
\beq
\aquant \varepsilon >0 \aquant n\ge \Phi(\varepsilon,b,K,\alpha)\bigg(d(x_n,Tx_n)< \varepsilon\bigg),
\eeq
where $\Phi(\varepsilon,b,K,\alpha)=\widehat{\alpha}\left(\lceil 12b\cdot \exp(K(M+1)) \rceil\remin\, 1,M\right),
$
with $\ds M=\left\lceil\frac{1+6b}{\varepsilon}\right\rceil$ and $\widehat{\alpha}$ as in Theorem \ref{habil-quant-BRS}.
\ethm

For the case of directionally nonexpansive mappings, however, the additional assumption in  Theorem \ref{habil-quant-BRS-dir-ne} causes various problems and significant changes in the proofs. In the following, we will only consider the case where $(x_n)$ itself is bounded (i.e. $x=x^*$). 

For any $k\in\N$, we define the sequence $((x_k)_m)_{m\in\N}$ by:
\[(x_k)_0:=x_k, \qquad (x_k)_{m+1}:=(1-\lambda_m)(x_k)_m\oplus
\lambda_m T((x_k)_m).\]
Hence, for any $k\in\N$, $((x_k)_m)_{m\in\N}$ is the Krasnoselski-Mann
iteration starting with $x_k$. Let us remark that $((x_k)_m)_{m\in\N}$ is not in general a subsequence of $(x_n)$. 

The following result is the quantitative version of Theorem \ref{asreg-Ishikawa-bounded-xn-W-dir-ne}.

\begin{theorem}\label{quant-BRS-bounded-xn-dir-ne}
Let $(X,d, W), C, (\lambda_n),K,\alpha$ be  as in the hypothesis of Theorem \ref{habil-quant-BRS} and $T:C\to C$ be directionally nonexpansive.\\
Assume that $x\in C, b>0$ are such that  
\beq
\forall\, n,k,m\in\N \left(d(x_n,(x_k)_m\right)\le b\big). \label{quant-BRS-bounded-xn-dir-ne-hyp-xn}
\eeq
Then 
\beq
\aquant \varepsilon >0\aquant n\ge \Phi(\varepsilon,b,K,\alpha)
\bigg(d(x_n,Tx_n)<\varepsilon\bigg), 
\eeq
where $\Phi(\varepsilon,b,K,\alpha)=\alpha(0,1)+\widehat{\beta}(\lceil 
2b\cdot\alpha(0,1)\cdot\exp(K(M+1)) \rceil -1,M)$, with 
\[ \ba{l}
\ds M=\left\lceil\frac{1+2b}{\varepsilon}\right\rceil,\quad \beta, \widehat{\beta},\tilde{\beta}:\N\times\N\to \N,  \quad \beta(i,n)=\alpha(i+\alpha(0,1),n)\\
\tilde{\beta}(i,n)=i+\beta(i,n), \quad \widehat{\beta}(0,n)=\tilde{\beta}(0,n), \quad \widehat{\beta}(i+1,n)=
\tilde{\beta}(\widehat{\beta}(i,n),n).
\ea\]
\end{theorem}

\noindent Thus, in the case of directionally nonexpansive mappings, we need the stronger requirement (\ref{quant-BRS-bounded-xn-dir-ne-hyp-xn}). Note that for constant $\lambda_n=\lambda$,  $(x_k)_m=x_{k+m}$ for all $m,k\in\N$,  so $((x_k)_m)_{m\in\N}$ is a subsequence of $(x_n)$. In this case, the assumption $d(x_n,x_m)\leq b$ for all $m,n\in\N$ suffices.

\begin{corollary}\label{habil-BRS-bounded-xn-constant-lambda-dir-ne} 
Let $(X,d, W), C, T, K$ be  as before. Assume that $\lambda_n=\lambda$ for all $n\in\N$, where  $\lambda\in[1/K,1-1/K]$.
Let $x\in C, b>0$ be such that  $d(x_n,x_m)\le b$ for all $m,n\in\N$. 

Then the following holds
\beq
\aquant \varepsilon >0\aquant n\ge \Phi(\varepsilon,b,K)
\bigg(d(x_n,Tx_n)<\varepsilon\bigg), 
\eeq
where $\Phi(\varepsilon,b,K)=K+K\cdot M\cdot \lceil 2b\cdot K
\cdot\exp(K(M+1))\rceil$, with $\ds M=\left\lceil\frac{1+2b}{\varepsilon}\right\rceil.$
\end{corollary}
Hence, we obtain a  strong uniform version of Kirk Theorem \ref{asreg-uniform-Kirk-dir-ne}.(i), which does not   state any uniformity of the convergence at all.

\subsection{A quadratic rate of asymptotic regularity for $CAT(0)$ spaces}\label{habil-app-fpt-CAT0}

If $T:C\to C$ is a nonexpansive self-mapping of a bounded convex subset $C$ of a $W$-hyperbolic space and $(\lambda_n)$ is a sequence in  $\left[1/K,1-1/K\right]$ for some $K\in\N,K\ge 2$ (in particular, $\lambda_n=\lambda\in (0,1)$), then, as we have seen in the previous subsection, Corollary \ref{habil-quant-BRS-bounded-exp} gives an exponential (in $1/\eps$) rate of asymptotic regularity for the Krasnoselski-Mann iteration.

For normed spaces and the special case of constant $\lambda_n=\lambda\in (0,1)$, this exponential bound  is not optimal. In this case, a uniform and optimal quadratic bound was obtained by Baillon and Bruck \cite{BaiBru96} using an extremely complicated computer aided proof, and only for $\lambda_n=1/2$ a classical proof of a  result of this type was given \cite{Bru96}. However, the questions whether the methods of proof used by them hold for  non-constant sequences $(\lambda_n)$ or for $W$-hyperbolic spaces are left as open problems in \cite{BaiBru96}, and as far as we know they received no positive answer until now. Hence, the bound from Corollary \ref{habil-quant-BRS-bounded-exp} is the only effective bound known at all for non-constant sequences $(\lambda_n)$ (even for normed spaces).

Our result guarantees only an exponential rate of  asymptotic regularity in the case of $CAT(0)$ spaces, and as we have already remarked, it seems that Baillon and Bruck's approach does not extend to this more general setting.

In this subsection we show that we can still get a quadratic rate of asymptotic regularity for $CAT(0)$ spaces, but following a completely different approach, inspired by the results on asymptotic regularity obtained before Ishikawa and Edelstein-O'Brien theorems, in the setting of uniformly convex Banach spaces.  The method we use is to find explicit uniform bounds on  the rate of asymptotic regularity in the general setting of $UCW$-hyperbolic spaces and then to specialize them to $CAT(0)$ spaces. As we have seen in Section \ref{hyperbolic-spaces}, $CAT(0)$ spaces are $UCW$-hyperbolic spaces with a nice modulus of uniform convexity.

More specifically, our point of departure is the following theorem due to Groetsch.

\bthm\cite{Gro72}\label{habil-Groetsch-UCW}
Let $C$ be a  convex subset of a $UCW$-hyperbolic space and $T:C\to C$ be a nonexpansive mapping  such that $T$ has at least one fixed point. \\
Assume that $(\lambda_n)$ is a sequence in $[0,1]$ satisfying
\beq
\sum_{n=0}^\infty\lambda_n(1-\lambda_n)=\infty.\label{habil-hyp-lambda-n-Groetsch}
\eeq
Then $\displaystyle\lim_{n\to\infty} d(x_n,Tx_n)=0$ for all $x\in C$.
\ethm

\noindent The above theorem was proved by Groetsch for uniformly convex Banach spaces (see Theorem \ref{intro-Groetsch-thm}), but it is easy to see that its proof extends to $UCW$-hyperbolic spaces. By proof mining, Kohlenbach \cite{Koh03} obtained a quantitative version of Groetsch Theorem \ref{intro-Groetsch-thm} for uniformly convex Banach spaces, generalizing previous results obtained by Kirk and Martinez-Yanez \cite{KirMar90} for constant $\lambda_n=\lambda\in(0,1)$. 

In \cite{Leu07} we  extended Kohlenbach's results to the more general setting of $UCW$-hyperbolic spaces. The most important consequence of our results is that for $CAT(0)$ spaces we obtain a quadratic rate of asymptotic regularity (see Corollary \ref{habil-CAT0-constant-lambda}).

The following table presents a general picture of the cases where effective bounds for asymptotic regularity were obtained.\\[0.4cm]
\begin{tabular}{|c|c|c|}
\hline 
&&\\
& $\lambda_n=\lambda$ & non-constant $\lambda_n$ \\[0.2cm]
\hline
&&\\
Hilbert spaces & {\em quadratic} in $1/\eps$: Browder and Petryshyn \cite{BroPet67} & $\theta\left(1/\varepsilon^2\right)$: Kohlenbach \cite{Koh03}\\[0.2cm]
\hline
&&\\
$\ell_p$, $2\leq p<\infty$ &  {\em quadratic} in $1/\eps$:  Kirk and Martinez-Yanez\cite{KirMar90}, & $\theta\left(1/\varepsilon^p\right)$: Kohlenbach \cite{Koh03}\\[0.1cm]
&  \quad  \quad \quad Kohlenbach \cite{Koh03} & \\[0.2cm]
\hline
&&\\
 uniformly convex &Kirk and Martinez-Yanez\cite{KirMar90}, Kohlenbach \cite{Koh03}  & Kohlenbach \cite{Koh03}\\[0.1cm]
 Banach  spaces & & \\[0.2cm]
\hline
&&\\
Banach   & {\em quadratic} in $1/\eps$: Baillon and Bruck \cite{BaiBru96} & Kohlenbach \cite{Koh01}\\[0.2cm]
\hline
&&\\
$CAT(0)$\, spaces & {\em quadratic} in $1/\eps$:  Corollary \ref{habil-CAT0-constant-lambda} & $\theta\left(1/\varepsilon^2\right)$: Corollary \ref{habil-CAT0-general-lambda} \\[0.2cm]
\hline
&&\\
UCW-hyperbolic   & Corollary \ref{habil-Groetsch-bounded-C-constant-lambda} & Corollary  \ref{habil-bounded-C-general-lambda} \\[0.1cm]
spaces &&\\[0.2cm]
\hline
&&\\
W-hyperbolic  & {\em exponential} in $1/\eps$: Corollary \ref{habil-quant-BRS-bounded-exp} & Theorem \ref{habil-quant-BRS-bounded} \\[0.1cm]
spaces &&\\[0.2cm]
\hline
\end{tabular}

\subsubsection{Logical discussion}

As in the case of the logical analysis of Borwein-Reich-Shafrir Theorem, it suffices to consider nonexpansive mappings $T:X\to X$. Moreover, it is easy to see that the proof of Groetsch Theorem can be formalized in the theory $\mathcal{A}^\omega[X,d,UCW,\eta]_{-b}$ of $UCW$-hyperbolic spaces with a monotone modulus of uniform convexity $\eta$.

The assumption  on $(\lambda_n)$ in Theorem \ref{habil-Groetsch-UCW} is equivalent with the existence of a rate of divergence  $\theta:\N\to\N$ such that for all $n\in \N$,
$$\displaystyle\sum_{i=0}^{\theta(n)} \lambda_i(1-\lambda_i)\geq n.$$

Using the fact that $(d(x_n,Tx_n)$ is nonincreasing, it  follows that $\mathcal{A}^\omega[X,d,UCW,\eta]_{-b}$ proves the following formalized version of Theorem \ref{habil-Groetsch-UCW}:
\[\ba{l}
\forall\, \eps>0\, \forall\, \theta:\N\to\N \,\forall\, (\lambda_n)\in[0,1]^\N\,\forall\, x\in X\, \forall\,\, T:X\to X \\
\quad\quad \quad\quad\biggl(T \text{~nonexpansive~} \si \, Fix(T)\ne\emptyset\, \si \, \ds \forall n\in\N\left( \sum_{i=0}^{\theta(n)}\lambda_i(1-\lambda_i)\ge n \right)\\
\hfill \,\ra\, \exists N\in\N \bigg(d(x_N, Tx_N)<\eps\bigg)\biggr)
\ea\]
Since we can let $\eps=2^{-p}$ with $p\in\N$, the above formalization of the statement of Theorem \ref{habil-Groetsch-UCW} has the required logical form for applying Corollary \ref{meta-Groetsch}.
It follows that we can extract a computable functional $\Phi$ such that for all $\eps>0, b\in\N, \theta:\N\to\N$,

\[\ba{l}
\forall\, (\lambda_n)\in[0,1]^\N\,\forall\, x\in X\, \forall\,\, T:X\to X \\
\quad\quad \quad\quad\biggl(T \text{~nonexpansive~} \si \, \forall \delta>0(Fix_\delta(T,x,b)\ne\emptyset)\, \si \, \ds \forall n\in\N\left( \sum_{i=0}^{\theta(n)}\lambda_i(1-\lambda_i)\ge n \right)\\
\hfill \,\ra\, \exists N\le \Phi(\eps,\eta,b,\theta)\bigg(d(x_N, Tx_N)<\eps\bigg)\biggr)
\ea\]
holds  in any  $UCW$-hyperbolic space with monotone modulus $\eta$.  We recall that 
\[Fix_\delta(T,x,b)= \{y\in X\mid d(x,y)\le b \,\si \,d(y,Ty)<\delta\}.\]
Using again that $(d(x_n,Tx_n))$ is nonincreasing, it follows that $\Phi(\eps,\eta, b, \theta)$ is in fact a rate convergence of $(d(x_n,Tx_n))$ towards $0$.

\subsubsection{Main results}

The following quantitative version of Groetsch Theorem is the main result of \cite{Leu07}.

\begin{theorem}  \label{habil-main-Groetsch-thm} 
Let $C$ be a  convex subset of a $UCW$-hyperbolic space $(X,d,W)$ and $T:C\rightarrow C$ be a nonexpansive mapping. 

Assume that $(\lambda_n)$ is a sequence in $[0,1]$ and $\theta :\N\to\N$ satisfies for all $n\in\N$, 
\begin{equation}
\sum\limits_{k=0}^{\theta(n)} \lambda_k(1-\lambda_k) \ge n. \label{habil-main-hyp-lambda-theta}
\end{equation} 
Let $x\in C,b>0$ be such that $T$ has approximate fixed points in a $b$-neighborhood of $x$. 

Then  $\displaystyle\limn  d(x_n,Tx_n)=0$ and, moreover,
\begin{equation} 
\forall \varepsilon >0\, \forall n\ge \Phi(\varepsilon,\eta, b,\theta)\, \bigg(d(x_n,Tx_n) < \varepsilon\bigg), \label{habil-main-thm-conclusion}
\end{equation}
where  $\eta$ is a monotone  modulus of uniform convexity and
\[\Phi(\eps,\eta,b,\theta)=\left\{ \begin{array}{ll}
        \displaystyle \theta\left(\left\lceil\frac{b+1}{\varepsilon\cdot\eta\left(b+1,\displaystyle\frac{\varepsilon}{b+1}\right)}
\right\rceil\right) & \text{for~ } \varepsilon <2b\\
        0 & \text{otherwise. }
        \end{array}\right.
\]
If we assume moreover that $\eta$ can be written as $\eta(r,\varepsilon)=\varepsilon\cdot\tilde{\eta}(r,\varepsilon)$ such that $\tilde{\eta}$ increases with $\varepsilon$ (for a fixed $r$), then the bound $\Phi(\varepsilon,\eta,b,\theta)$ can be replaced for $\varepsilon <2b$ by
\[ \tilde{\Phi}(\varepsilon,\eta,b,\theta)= \theta\left(\left\lceil
\frac{b+1}{2\varepsilon\cdot \tilde{\eta}\left(b+1,\displaystyle\frac{\varepsilon}{b+1}\right)}\right\rceil\right). \] 
\end{theorem}

\noindent As an immediate consequence of our main theorem, we obtain a slight strengthening of Groetsch Theorem.

\bcor
Let $C$ be a  convex subset of a $UCW$-hyperbolic space $(X,d,W)$ and $T:C\rightarrow C$ be nonexpansive. Assume that $(\lambda_n)$ is a sequence in $[0,1]$ satisfying $\ds\sum_{n=0}^\infty\lambda_n(1-\lambda_n)=\infty$.\\
Let $x\in C,b>0$ be such that $T$ has approximate fixed points in a $b$-neighborhood of $x$. 

Then $\limn d(x_n,Tx_n)=0$.
\ecor
Thus, we assume that  $T$ has approximate fixed points in a $b$-neighborhood of some $x\in C$ instead of having fixed points. However, by Proposition \ref{UCW-equiv-T-fpp}, for {\em closed} convex  subsets $C$ of {\em complete} 
$UCW$-hyperbolic spaces, $T$ has fixed points is equivalent with $T$ having approximate fixed points in a $b$-neighborhood of $x$.

If $C$ is bounded with diameter $d_C$, then $C$ has the AFPP for nonexpansive mappings by Proposition \ref{habil-bounded-hyperbolic-afpp-ne}, so we can apply Theorem \ref{habil-main-Groetsch-thm} for all $x\in C$ with $d_C$ instead of $b$.

\begin{corollary}\label{habil-bounded-C-general-lambda}
 Let $(X,d,W),\eta,C,T,(\lambda_n),\theta$ be as in the hypothesis of Theorem \ref{habil-main-Groetsch-thm}. Assume moreover that $C$ is bounded with diameter $d_C$.
 
Then $T$ is $\lambda_n$-asymptotically regular and  the following holds for all $x\in C$:
\[\forall \varepsilon >0\,\forall n\ge \Phi(\varepsilon,\eta,d_C,\theta)\,
\bigg(d(x_n,Tx_n) <\varepsilon\bigg), \]
where $\Phi(\varepsilon,\eta,d_C,\theta)$ is defined as in Theorem \ref{habil-main-Groetsch-thm} by replacing $b$ with $d_C$.
\end{corollary}

For bounded $C$, we get $\lambda_n$-asymptotic regularity for general $(\lambda_n)$ satisfying (\ref{habil-hyp-lambda-n-Groetsch}) and we also obtain an effective rate $\Phi(\varepsilon,\eta,d_C,\theta)$ of asymptotic regularity that depends only on the error $\varepsilon$, on the modulus of uniform convexity $\eta$,  on the diameter $d_C$ of $C$, and on $(\lambda_n)$ via $\theta$, but not on the nonexpansive mapping $T$, the starting point $x\in C$ of the iteration or other data related with $C$ and $X$.

The rate of asymptotic regularity can be further simplified for constant $\lambda_n=\lambda\in(0,1)$. In this case, it is easy to see that $\ds \theta:\N\to\N  \quad\theta(n)=n\cdot\left\lceil\frac{1}{\lambda(1-\lambda)}\right\rceil$ satisfies (\ref{habil-main-hyp-lambda-theta}).

\begin{corollary}\label{habil-Groetsch-bounded-C-constant-lambda}
Let $(X,d,W),\eta, C, d_C,T$ be as in the hypothesis of Corollary \ref{habil-bounded-C-general-lambda}.
Assume moreover that $\lambda_n=\lambda\in(0,1)$ for all $n\in\N$.

Then $T$ is $\lambda$-asymptotically regular and  for all $x\in C$,
\begin{equation} 
\forall \varepsilon >0\,\forall n\ge \Phi(\varepsilon,\eta,d_C,\lambda)\,\bigg(d(x_n,Tx_n) <\varepsilon\bigg),
\end{equation}
where
\[\Phi(\varepsilon,\eta,d_C,\lambda)=\left\{ \begin{array}{ll}\displaystyle\left\lceil\frac{1}{\lambda(1-\lambda)}\right\rceil\cdot\left\lceil\frac{d_C+1}{\varepsilon\cdot\eta\left(d_C+1,\displaystyle\frac{\varepsilon}{d_C+1}\right)}\right\rceil & \text{for~ } \varepsilon <2d_C\\
        0 & \text{otherwise. }
        \end{array}\right.
\]
Moreover, if $\eta(r,\varepsilon)$ can be written as $\eta(r,\varepsilon)=\varepsilon\cdot\tilde{\eta}(r,\varepsilon)$ such that $\tilde{\eta}$ increases with $\varepsilon$ (for fixed $r$), then the bound $\Phi(\varepsilon,\eta,d_C,\lambda)$ can be replaced for $\varepsilon<2d_C$ with
\[\tilde{\Phi}(\varepsilon,\eta,d_C,\lambda)=\displaystyle\left\lceil\frac{1}{\lambda(1-\lambda)}\right\rceil\cdot\left\lceil\frac{d_C+1}{2\varepsilon\cdot\tilde{\eta}\left(d_C+1,\displaystyle\frac{\varepsilon}{d_C+1}\right)}\right\rceil.\]
\end{corollary}

As we have seen in Subsubsection \ref{habil-CAT0}, $CAT(0)$ spaces are $UCW$-hyperbolic spaces with a modulus of uniform convexity $\displaystyle \eta(r,\varepsilon)=\frac{\varepsilon^2}{8}=\eps\cdot\tilde{\eta}(r,\varepsilon)$, where $\ds\tilde{\eta}(r,\eps)=\frac{\eps}8$ increases with $\eps$.  It follows that the above results can be applied to $CAT(0)$ spaces. 

\begin{corollary}\label{habil-CAT0-general-lambda}
Let $X$ be a $CAT(0)$ space,  and $C,d_C,T,(\lambda_n), \theta$ be as in the hypothesis of Corollary \ref{habil-bounded-C-general-lambda}. 

Then $T$ is $\lambda_n$-asymptotically regular and for all $x\in C$,
\begin{equation} 
\forall \varepsilon >0\, \forall n\ge \Psi(\varepsilon,d_C,\theta)\, \bigg(d(x_n,Tx_n)< \varepsilon\bigg),
\end{equation}
where
\[\Psi(\varepsilon,d_C,\theta)=\left\{ \begin{array}{ll}
        \displaystyle \theta\left(\left\lceil\frac{4(d_C+1)^2}{\varepsilon^2}\right\rceil\right) & \text{for~ } \varepsilon <2d_C\\
        0 & \text{otherwise. }
        \end{array}\right.
\]
\end{corollary}

For general $(\lambda_n)$, the rate of asymptotic regularity is of order $\ds\theta\left(\frac1{\eps^2}\right)$, where $\theta$ is a rate of divergence for $\ds\sum_{n=1}^\infty \lambda_n(1-\lambda_n)$.

\begin{corollary}\label{habil-CAT0-constant-lambda}
Let $X$ be a $CAT(0)$ space,  $C\se X$ be a bounded convex subset with diameter $d_C$, and  $T:C\to C$ be nonexpansive. Assume that $\lambda_n=\lambda\in(0,1)$.

Then $T$ is $\lambda$-asymptotically regular, and for all $x\in C$,
\begin{equation} 
\forall \varepsilon >0\, \forall n\ge \Psi(\varepsilon,d_C,\lambda)\, \bigg( d(x_n,Tx_n)< \varepsilon\bigg),
\end{equation}
where
\[\Psi(\varepsilon,d_C,\lambda)=\left\{ \begin{array}{ll}
        \displaystyle\left\lceil\frac{1}{\lambda(1-\lambda)}\right\rceil\cdot\left\lceil\frac{4(d_C+1)^2}{\varepsilon^2}\right\rceil & \text{for~ } \varepsilon <2d_C\\
        0 & \text{otherwise. }
        \end{array}\right.
\]
\end{corollary}

Hence, for bounded convex subsets of $CAT(0)$ spaces and  constant $\lambda_n=\lambda$, we get a quadratic (in $1/\varepsilon$) rate of asymptotic regularity.

\subsection{Uniform approximate fixed point property}\label{habil-app-fpt-uafpp}

Inspired by Theorem \ref{habil-quant-BRS}, our quantitative version of Borwein-Reich-Shafrir Theorem, we introduced in \cite{KohLeu07} the notions of uniform approximate fixed point property and uniform asymptotic regularity property. The idea is to forget about the quantitative features of  Theorem \ref{habil-quant-BRS} and to look only at the uniformities.

Let $(X,d)$ be a metric space, $C\subseteq X$  and $\mathcal{F}$ be a class of mappings $T:C\to C$. We say that $C$ has the {\em uniform approximate fixed point property (UAFPP)} for $\mathcal{F}$ if for all $\varepsilon>0$ and $b>0$ there exists  $D>0$ such that for 
each point $x\in C$ and for each mapping $T\in{\mathcal F}$,
\begin{equation}
d(x,Tx)\le b \text{ implies } T \text{ has } \eps\text{-fixed points in a } D\text{-neighborhood of }x. 
\label{habil-uafpp-def}
\end{equation}
Formally, $d(x,Tx)\le b\,\, \Rightarrow \,\, \exists x^*\in C\big(d(x,x^*)\le D\,\wedge \,
d(x^*, Tx^*)< \varepsilon\big)$.

Using the same ideas,  we can define the notion of $C$ having the uniform fixed point property. Thus, $C$ has the {\em uniform fixed point property (UFPP)} for $\mathcal{F}$ if for all $b>0$ there exists $D>0$ such that for each point $x\in C$ and 
for each mapping $T\in{\mathcal F}$,
\begin{equation}
d(x,Tx)\le b \text{ implies } T \text{ has fixed points in a } D\text{-neighborhood of }x. 
\label{habil-ufpp-def}
\end{equation}
That is, $d(x,Tx)\le b\,\, \Rightarrow \,\, \exists x^*\in C\big(d(x,x^*)\le D\,\wedge\, 
Tx^*=x^*$\big). As an immediate application of Banach's Contraction Mapping Principle, we get the following.

\bprop
Assume that $(X,d)$ is a complete metric space and let $\mathcal{F}$ be the class of contractions with a common contraction constant $k\in (0,1)$. Then each closed subset $C$ of $X$ has the UFPP for $\mathcal{F}$.
\eprop
\begin{proof}
By Banach's Contraction Mapping Principle we know that each mapping $T\in\mathcal{F}$ has a unique fixed point $x_0$ and, moreover, for each $x\in C$, 
\begin{equation}
d(T^nx, x_0)\le \frac{k^n}{1-k} d(x,Tx) \quad \text{ for all } n\in\N.\label{contraction-Banach-estimation}
\end{equation}
For $n=0$, this yields $\displaystyle d(x,x_0)\le \frac{d(x,Tx)}{1-k} $, so $\displaystyle  d(x,Tx)\le b$ implies $\displaystyle d(x,x_0)\le \frac{b}{1-k}$. Hence, (\ref{habil-ufpp-def}) holds with $\displaystyle D=\frac{b}{1-k}b$.
\end{proof}

Let $(X,d,W)$ be a $W$-hyperbolic space, and $C\subseteq X$ be a convex subset and assume that $(\lambda_n)$ is a sequence in $[0,1]$. We say that $C$ has the {\em $\lambda_n$-uniform asymptotic regularity property } for $\mathcal{F}$ if for all $\varepsilon>0$ 
and $b>0$ there exists  $N\in\N$ such that for each point $x\in C$ and for 
each mapping $T\in{\mathcal F}$, 
\begin{equation}
d(x,Tx)\le b\quad \Rightarrow \quad\forall n\ge N\big(d(x_n,Tx_n)<\varepsilon\big), \label{habil-lambda-uar-def}
\end{equation}
where $(x_n)$ is the Krasnoselski-Mann iteration.

As an immediate consequence of Theorem \ref{habil-quant-BRS-bounded}, bounded convex subsets of $W$-hyperbolic spaces have the $\lambda_n$-uniform asymptotic regularity property for directionally nonexpansive mappings for all 
$(\lambda_n)$ divergent in sum and bounded away from 1.

Theorem \ref{habil-quant-BRS} is used to prove the following equivalent characterizations. 

\begin{proposition}\cite{KohLeu07}
Let $C$ be a convex subset of a $W$-hyperbolic space $(X,d,W)$. The following are equivalent.
\begin{enumerate}
\item[(i)] $C$ has the UAFPP for nonexpansive mappings;
\item[(ii)] there exists  $(\lambda_n)$  in $[0,1]$ such that $C$ has the $\lambda_n$-uniform asymptotic regularity property for nonexpansive mappings;

\item[(iii)] for all $(\lambda_n)$  in $[0,1]$ which are divergent in sum and bounded away from 1, $C$ has the $\lambda_n$-uniform asymptotic regularity property for nonexpansive mappings.
\end{enumerate}
\end{proposition}
\begin{proof}
We give only the proof of $(i)\Rightarrow (iii)$, for which the main ingredient is our quantitative Borwein-Reich-Shafrir Theorem \ref{habil-quant-BRS}. We refer to \cite[Proposition 16]{KohLeu07} for the complete proof. 

Let $\varepsilon>0, b>0$, and $D>0$ be such that (\ref{habil-uafpp-def}) holds with ${\mathcal F}$ being the class of nonexpansive mappings. If  $(\lambda_n)$  in $[0,1]$ is divergent in sum and bounded away from 1, then, as we have already discussed in Subsubsection \ref{BRS-logical}, there exist  $K\in\N$ and $\alpha:\N\to\N$ satisfying the corresponding hypothesis of  Theorem \ref{habil-quant-BRS}.  Let $x\in C$ and $T:C\to C$ nonexpansive 
be such that $d(x,Tx)\le b$. By (\ref{habil-uafpp-def}), there exists $x^*\in C$ satisfying $d(x,x^*)\le D,$ and $d(x^*, Tx^*)<\varepsilon$. By taking $b^*=\max\{b,D\}$, it follows that 
\[d(x,Tx)\le b^* \quad \text{and}\quad d(x,x^*)\le b^*,\]
so the hypothesis (\ref{habil-quant-BRS-hyp-x-b}) is also satisfied. It follows that we can apply Theorem \ref{habil-quant-BRS} to get 
$N=\Phi(\varepsilon,b^*,K,\alpha)$ such that $d(x_n,Tx_n) < d(x^*,Tx^*)+\varepsilon < 2\varepsilon$ for all $n\ge N$.
\end{proof}

Let us remark the following fact. A first attempt to define  the property 
that $C$ has the uniform approximate fixed point property for nonexpansive 
mappings is in the line of Goebel-Kirk Theorem \ref{asreg-uniform-Goebel-Kirk-W}, that is: for all $\varepsilon>0$ there exists $D>0$ such that for all $x\in C$ and for all  $T\in{\mathcal F}$
 
\beq
\exists x^*\in C\big(d(x,x^*)\le D\wedge d(x^*, Tx^*)<\varepsilon\big). 
\label{habil-uafpp-goebel-kirk}
\eeq
In this case, it follows  that, even if we consider only constant mappings $T$, the only  subsets $C$  satisfying (\ref{habil-uafpp-goebel-kirk}) are the bounded ones.  If $C$ is bounded, then $C$ satisfies (\ref{habil-uafpp-goebel-kirk}) by Goebel-Kirk Theorem \ref{asreg-uniform-Goebel-Kirk-W}. Conversely, assume that $C$ satisfies (\ref{habil-uafpp-goebel-kirk}) for all constant mappings $T$. Then for $\varepsilon=1$ we get $D_1\in\N$ such that for all $x\in C$, and for 
all constant mappings $T:C\to C$, there is $x^*\in C$ with $d(x,x^*)\le D_1$ and 
$d(x^*, Tx^*)< 1$. It follows that 
\beq
d(x,Tx)\le d(x,x^*)+d(x^*, Tx^*)+d(Tx^*,Tx)
\le 2D_1+1\label{habil-ineq-unif-GK}
\eeq
Now, if we assume that $C$ is unbounded, there are $x,y\in C$ such that 
$d(x,y)>2D_1+1$. Define $T:C\to C, \, T(z)=y$ for all $z\in C$. Then 
$d(x,Tx)=d(x,y)>2D_1+1$ which contradicts (\ref{habil-ineq-unif-GK}).\\

We conclude this subsection with an  {\bf open problem}:\\

\noindent{\em  Are there  unbounded convex subsets $C$ of some $W$-hyperbolic space which have the UAFPP for all nonexpansive mappings $T:C\to C$ ? }

\subsection{Approximate fixed points in product spaces}

If $(X,\rho)$ and $(Y,d)$ are metric spaces, then the metric $d_\infty$ on $X\times Y$ is defined in the usual way:
\[ d_\infty ((x,u),(y,v))=\max\{\rho(x,y), d(u,v)\}\]
for $(x,u), (y,v)\in X\times Y$. We denote by $(X\times Y)_\infty$ the metric space thus obtained.

The following theorem was proved first by Esp\' inola and Kirk \cite{EspKir01} for Banach spaces and then by Kirk  \cite{Kir04} for $CAT(0)$ spaces. 

\begin{theorem}\label{habil-th-kirk-CAT}
Assume that $X$ is a  Banach space or a $CAT(0)$ space and $C\subseteq X$ is  a  bounded closed convex subset of $X$. 
If $(M,d)$ is a metric space with the AFPP for nonexpansive mappings, then
\[H:=(C\times M)_\infty\]
has the AFPP for nonexpansive mappings.
\end{theorem}
The proof of this result uses essentially Goebel-Kirk Theorem \ref{asreg-uniform-Goebel-Kirk-W}.

In the following, we generalize  Theorem \ref{habil-th-kirk-CAT} to {\em unbounded} convex subsets $C$ of $W$-hyperbolic spaces. We extend the results further, to families $(C_u)_{u\in M}$ of unbounded convex subsets of a $W$-hyperbolic space. The key ingredient in obtaining these generalizations is Theorem \ref{habil-quant-BRS}, our uniform quantitative version of Borwein-Reich-Shafrir Theorem.  The results presented in this subsection were obtained by Kohlenbach and the author in \cite{KohLeu07}.

\subsubsection{The case of one convex subset $C$}

In the sequel, $C\subseteq X$ is a convex subset of a $W$-hyperbolic space $(X,\rho,W)$, $(M,d)$  is a metric space which has the AFPP for nonexpansive mappings and $H:=(C\times M)_\infty$ and $(\lambda_n)$ is a sequence in $[0,1]$.

Let us denote with $P_1:H\to C,\, P_2:H\to M$  the coordinate projections and define for each nonexpansive mapping $T:H\to H$ and for  each $u\in M$,
\[T_u:C\to C, \quad T_u(x)=(P_1\circ T)(x,u).\]
It is easy to see that $T_u$ is nonexpansive, so we can associate with $T_u$ the Krasnoselski-Mann iteration $(x_n^u)$  starting with an arbitrary $x\in C$.

In the sequel, $\delta:M\to C$ is a nonexpansive mapping that {\em selects} for each $u\in M$ an element $\delta(u)\in C$. Trivial examples of such nonexpansive selection mappings are the constant ones. For simplicity, we shall denote the Krasnoselski-Mann iteration starting from $\delta(u)$ and associated with $T_u$ by $(\delta_n(u))$:
\[\delta_0(u):=\delta(u), \quad \delta_{n+1}(u):=(1-\lambda_n)\delta_n(u)\oplus\lambda_nT_u(\delta_n(u)).\]
 For each $n\in\N$, let us define 
\begin{eqnarray*}
\varphi_n:M\to M, & \varphi_n(u)=(P_2\circ T)(\delta_n(u),u).
\end{eqnarray*}

\begin{theorem}\label{habil-main-rH-rC}
Assume that 
\[\displaystyle\sup_{u\in M}r_C(T_u)<\infty,\]
and $\varphi:\R^*_+\to \R^*_+$ is such that for each $\varepsilon>0$ and $v\in M$ there exists $x^*\in C$ satisfying 
\begin{eqnarray}
\rho(\delta(v),x^*)\le \varphi(\varepsilon)\quad\text{and}\quad \rho(x^*,T_v(x^*))\le 
\sup_{u\in M}r_C(T_u)+\varepsilon.\label{habil-hyp-rH-rC}
\end{eqnarray}
Then $\ds r_H(T)\le \sup_{u\in M}r_C(T_u)$.
\end{theorem}

As an immediate consequence, we get the following result.

\begin{corollary}\label{habil-cor-main-rH-rC-used}
Assume that $\varphi:\R^*_+\to \R^*_+$ is such that
\begin{equation}
\forall\varepsilon>0\forall u\in M\exists x^*\in C\bigg(\rho(\delta(u),x^*)\le 
\varphi(\varepsilon)\quad\text{and}\quad\rho(x^*,T_u(x^*))\le\varepsilon\bigg).\label{habil-hyp-C-AFPP}
\end{equation}
Then $r_H(T)=0$.
\end{corollary}
\begin{proof}
From the hypothesis, it follows immediately that $r_C(T_u)=0$ for all $u\in M$. 
\end{proof}

The next theorem is obtained by applying Theorem \ref{habil-quant-BRS-bounded-xn} to the family $(T_u)_{u\in M}$. 

\begin{theorem}\label{habil-main-x-yn-bounded}
Assume that $(\lambda_n)$ is divergent in sum and bounded away from $1$
and that there exists $b>0$ such that 
\begin{eqnarray}
\forall u\in M\exists y\in C\bigg(\rho(\delta(u),y)\le b\quad\text{and}\quad \forall m,p\in\N \big(\rho(y^u_m,y^u_p)\le b\big)\bigg), \label{habil-hyp-main-x-yn-bounded}
\end{eqnarray}
where $(y^u_n)$ is the Krasnoselski-Mann iteration associated with $T_u$, starting with $y$: 
\[y^u_0:=y, \quad y^u_{n+1}=(1-\lambda_n)y^u_n\oplus\lambda_nT_u(y^u_n).\]
Then $r_H(T)=0$. 
\end{theorem}

Applying the above theorem with  $y:=\delta(u)$, we get the following generalization of Theorem \ref{habil-th-kirk-CAT}.

\begin{corollary}\label{habil-cor-useful}
Assume that for all $u\in M$, the Krasnoselski-Mann iteration $\delta_n(u)$ is bounded. Then $r_H(T)=0$.   
\end{corollary}

Theorem \ref{habil-th-kirk-CAT} is an immediate consequence of Corollary \ref{habil-cor-useful}, since if $C$ is bounded, $\delta_n(u)$ is bounded for each $u\in M$.

\subsubsection{Families of unbounded convex sets}

\noindent In the following we indicate that  all the above results can be generalized to families $(C_u)_{u\in M}$ of   {\em unbounded} convex subsets of the $W$-hyperbolic space $(X,\rho,W)$.

Let $(C_u)_{u\in M}$ be a family of convex subsets of $X$ with the property that there exists a nonexpansive {\em selection } mapping $\delta:M\to \bigcup_{u\in M}C_u$, that is a nonexpansive mapping satisfying
\begin{equation} \forall u\in M\big(\delta(u)\in C_u\big). \label{habil-selection}
\end{equation}

We consider the following subspace of $(X\times M)_\infty$:
\[ H:=\{(x,u): u\in M, x\in C_u\}\]
and let  $P_1:H\to \bigcup\limits_{u\in M} C_u, \, P_2:H\to M$ be the projections. 

In the following, we consider nonexpansive mappings  $T:H\to H$ satisfying
\begin{equation}
 \forall (x,u)\in H\,\, \bigg((P_1\circ T)(x,u)\in C_u\bigg).\label{habil-hyp-T-Cu}
\end{equation}
It is easy to see that we can define a nonexpansive mapping
\[T_u:C_u\to C_u, \quad T_u(x)=(P_1\circ T)(x,u)\]
for each $u\in M$. We denote the Krasnoselski-Mann iteration starting from $x\in C_u$ and associated with $T_u$ by $(x^u_n)$.

For each $n\in\N$, we define 
\begin{eqnarray*}
\varphi_n:M\to M, & \varphi_n(u)=(P_2\circ T)(\delta_n(u),u).
\end{eqnarray*}

The following results can be proved in a similar manner with Theorems \ref{habil-main-rH-rC}, \ref{habil-main-x-yn-bounded}.

\begin{theorem}\label{habil-main-family-rH-rC}
Assume that 
\[\sup_{u\in M}r_{C_u}(T_u)<\infty\]
and that $\varphi:\R^*_+\to \R^*_+$ is such that for each $\varepsilon>0$ and $v\in M$ there exists $x^*\in C_v$ satisfying 
\begin{eqnarray*}
\rho(\delta(v),x^*)\le \varphi(\varepsilon)\,\quad\text{and}\quad\,\rho(x^*,T_v(x^*))\le 
\sup_{u\in M}r_{C_u}(T_u)+\varepsilon.\label{habil-hyp-main-family-rH-rC}
\end{eqnarray*}
Then $\ds r_H(T)\le \sup_{u\in M}r_{C_u}(T_u)$.
\end{theorem}

\begin{theorem}\label{habil-main-family-x-yn-bounded}
Let $(\lambda_n)$ divergent in sum and bounded away from $1$. Assume that there is $b>0$ such that 
\begin{eqnarray*}
\forall u\in M\exists y\in C_u\big(\rho(\delta(u),y)\le b\quad\text{and}\quad \forall m,p\in\N (\rho(y^u_m,y^u_p)\le b)\big).\label{habil-Cu-hyp-th-3.21}
\end{eqnarray*}
Then $r_H(T)=0$.
\end{theorem}

We get also the following corollary.

\begin{corollary}
Assume that $(C_u)_{u\in M}$ is a family of   {\em bounded} convex subsets of $X$ such that  $\ds\sup_{u\in M}diam(C_u)<\infty$.

Then $H$ has the AFPP for nonexpansive mappings $T\!:\!H\!\to\!H$ satisfying (\ref{habil-hyp-T-Cu}).
\end{corollary}
\begin{proof} 
The hypothesis of Theorem \ref{habil-main-family-x-yn-bounded} is satisfied with $y:=\delta(u)$.
 \end{proof}

\subsubsection{Partial answer to an open problem of Kirk}

In the following, we use our notion of uniform approximate fixed point property, introduced in Subsection \ref{habil-app-fpt-uafpp}, to give some partial answers to the following problem of Kirk  \cite[Problem 27]{Kir04}:

Let $C$ be a closed convex subset  of a complete $CAT(0)$ space $X$  (having the geodesic line extension property) and $M$ be a metric space. If both $C$ and $M$ have the AFPP for nonexpansive mappings, is it true that the product $H:=(C\times M)_\infty$ again has the AFPP?

We show that this is true if $C$ has the UAFPP (even in the case where $X$ is just a $W$-hyperbolic space) and a technical condition is satisfied which, in particular, holds if $M$ is bounded.

\begin{theorem}\label{habil-thm-UAFPP+AFPP}
Let $C$ be a convex subset  of a $W$-hyperbolic space $(X,\rho,W)$  and $(M,d)$ be a metric space with the AFPP for nonexpansive mappings. Assume that $C$ has the UAFPP for nonexpansive mappings.\\
Let $\delta:M\to C$ be a nonexpansive selection mapping and $T:H\to H$ be a nonexpansive mapping such that 
$\ds \sup_{u\in M}\rho(T_u(\delta(u)),\delta(u))<\infty$.

Then $r_H(T)=0$.
\end{theorem}
\begin{proof} 
Let $\varepsilon>0$ and $b>0$ be such that $\rho(T_u(\delta(u)),\delta(u))\le b$ for all $u\in M$. Since $C$ has the UAFPP for nonexpansive mappings, there exists $D>0$ (depending on $\varepsilon$ and $b$) such that (\ref{habil-uafpp-def}) holds for each nonexpansive self-mapping of $C$ and each $x\in C$. For each $u\in M$,  we can apply (\ref{habil-uafpp-def}) for $x:=\delta(u)$ and $T_u$ to get $x^*\in C$ such that $\rho(\delta(u),x^*)\le D$ and $\rho(x^*, T_u(x^*))\leq\varepsilon $. 
Hence, the hypothesis of Corollary \ref{habil-cor-main-rH-rC-used} is satisfied with $\varphi(\varepsilon)=D$, so $r_H(T)=0$ follows.
 \end{proof}

\begin{corollary}
Let $C$ be a convex subset  of a $W$-hyperbolic space $(X,\rho,W)$  and $(M,d)$ be a {\em bounded} metric space. Assume that $C$ has the UAFPP  and that $(M,d)$ has the AFPP for nonexpansive mappings. 

Then  $H:=(C\times M)_\infty$ has the AFPP for nonexpansive mappings. 
\end{corollary}
\begin{proof}
Let $x\in C$ be arbitrary, and define $\delta:M\to C$ by $\delta(u)=x$. Let $T:H\to H$ be a nonexpansive mapping. Fix some $u_0\in M$, and define $b:=\rho(x,T_{u_0}(x))+diam(M)$. Then $\rho(x,T_u(x))\le \rho(x,T_{u_0}(x))+ d(u_0,u)\le b$  for each $u\in M$, so we can apply Theorem \ref{habil-thm-UAFPP+AFPP} to conclude that  $r_H(T)=0$. 
\end{proof}
\subsection{Rates of asymptotic regularity for Halpern iterations}\label{habil-app-fpt-Halpern}

\noindent Let $C$ be a  convex subset of a W-hyperbolic space $(X,d,W)$ and $T:C\to C$ be nonexpansive. 

As in the case of normed spaces, we can define the {\em Halpern iteration} starting with $x\in C$ by 
\beq
x_0:=x, \quad x_{n+1}:=\lambda_{n+1}x\oplus(1-\lambda_{n+1})Tx_n, 
 \label{habil-Halpern-iterate-x}
\eeq
where $(\lambda_n)$ is a sequence in $[0,1]$.

The following lemma collects some useful properties of Halpern iterations.

\begin{lemma}\label{lemma-Halpern-hyp}
Assume that $(x_n)_{n\geq 1}$ is the Halpern iteration starting with $x\in C$. Then
\be
\item For all $n\geq 1$,
\bea
d(Tx_n,x)&\leq & d(x_n,x)+d(Tx,x) \label{H-Txn-x}\\
d(Tx_n,x_n) &\leq & d(x_{n+1},x_n)+ \lambda_{n+1}d(Tx_n,x) \label{H-Txn-xn}\\
d(x_{n+1},x) &\leq & (1-\lambda_{n+1})d(x_n,x)+(1-\lambda_{n+1})d(Tx,x) \label{H-xn+1-x}\\
d(x_{n+1},x_n) &\leq &  (1-\lambda_{n+1})d(x_n,x_{n-1})+|\lambda_{n+1}-\lambda_n|\,d(x,Tx_{n-1}) \label{H-xn+1-xn-ineq1}\\ 
d(x_{n+1},x_n) &\leq & \lambda_{n+1}d(x_n,x)+(1-\lambda_{n+1})d(Tx_n,x_n). \label{H-xn+1-xn-ineq2}
\eea
\item  If $(x_n)$ is bounded, then  $(Tx_n)$ is also bounded. Moreover, if $M\geq  d(x,Tx)$ and $M\geq  d(x_n,x)$ for all $n\geq 1$, then
\bea
d(Tx_n,x) \leq  2M  \quad \text{and}  \quad d(Tx_n,x_n) \leq  d(x_{n+1},x_n) +2M\lambda_{n+1}\label{H-ineq-Txn-xn}\\
d(x_{n+1},x_n) \leq   (1-\lambda_{n+1})d(x_n,x_{n-1})+2M|\lambda_{n+1}-\lambda_n|. \label{H-ineq-xn+1-xn}
\eea
for all $n\geq 1$.
\ee 
\end{lemma}
\begin{proof}
\be
\item
\bua
d(Tx_n,x)&\leq & d(Tx_n,Tx)+d(Tx,x) \leq d(x_n,x)+d(Tx,x)\\
d(Tx_n,x_n) &\leq & d(Tx_n,x_{n+1})+d(x_{n+1},x_n)\\
&=& d(x_{n+1},x_n) + d(Tx_n,\lambda_{n+1}x\oplus (1-\lambda_{n+1})Tx_n) \\
& = & d(x_{n+1},x_n)+ \lambda_{n+1}d(x,Tx_n)\,\quad \text{~by (\ref{habil-prop-xylambda}})\\
d(x_{n+1},x) &= & d(\lambda_{n+1}x\oplus (1-\lambda_{n+1})Tx_n,x)\\
 & = & (1-\lambda_{n+1}) d(Tx_n,x) \,\quad \text{~by (\ref{habil-prop-xylambda}})\\
&\leq & (1-\lambda_{n+1}) d(Tx_n,Tx)+(1-\lambda_{n+1}) d(Tx,x)\\
&\leq & (1-\lambda_{n+1})d(x_n,x)+(1-\lambda_{n+1})d(Tx,x),\\
d(x_{n+1},x_n) &=& d(\lambda_{n+1}x\oplus (1-\lambda_{n+1})Tx_n,\lambda_nx\oplus(1-\lambda_n)Tx_{n-1})\\
&\leq & d(\lambda_{n\!+\!1}x\oplus (1\!-\!\lambda_{n\!+\!1})Tx_n,\lambda_{n\!+\!1}x\oplus(1\!-\!\lambda_{n\!+\!1})Tx_{n\!-\!1})\!\\
&& + d(\lambda_{n+1}x\oplus(1-\lambda_{n+1})Tx_{n-1},\lambda_nx\oplus(1-\lambda_n)Tx_{n-1})\\
&\leq &  (1-\lambda_{n+1})d(Tx_n,Tx_{n-1})+ |\lambda_{n+1}-\lambda_n|d(x,Tx_{n-1})\\
&&\text{by (W4) and (W2)}\\
&\leq &  (1-\lambda_{n+1})d(x_n,x_{n-1})+|\lambda_{n+1}-\lambda_n|d(x,Tx_{n-1})
\eua
\bua
d(x_{n+1},x_n) &=& d(\lambda_{n+1}x\oplus (1-\lambda_{n+1})Tx_n,x_n) \\ 
&\leq & \lambda_{n}d(x_n,x)+(1-\lambda_{n+1})d(Tx_n,x_n)\,\text{~by (W1)}.
\eua
\item is an immediate consequence of (i).
\ee
\end{proof}

In the sequel  we give effective rates of asymptotic regularity for Halpern iterations, that is rates of convergence of the sequence $(d(x_n,Tx_n))$ towards $0$, where $(x_n)$ is the Halpern iteration starting with $x\in C$.

By inspecting the proof of Wittmann Theorem \ref{habil-wittmann-thm} (and its generalizations),  it is easy to see that the first step is to obtain asymptotic regularity, and that this can be done in  a much more general setting.   Thus, the following theorem, essentially contained in \cite{Wit92,Xu02,Xu04}, can be proved.

\begin{theorem}\label{habil-Halpern-ass-reg}
Let $C$ be a convex subset of a normed space $X$ and  $T:C\to C$ be nonexpansive. Assume that $(\lambda_n)_{n\geq 1}$ is a sequence in $[0,1]$  satisfies the following conditions
\beq
\limn \lambda_n =0, \quad \ds\sum_{n=1}^\infty \lambda_n \text{ is divergent}\quad \text{and } \ds \sum_{n=1}^\infty|\lambda_{n+1}-\lambda_n| \text{ is convergent}. \label{app-hyp-Halpern}
\eeq
Then $\ds\limn \|x_n-Tx_n\|=0$ for every $x\in C$  with the property that $(x_n)$ is bounded.
\end{theorem}

Applying proof mining techniques, we obtained in \cite{Leu07a} a quantitative version of the above theorem, which provides for the first time effective rates of asymptotic regularity for the Halpern iterations. Moreover, for $\lambda_n=1/n$, we get an exponential (in $1/\eps$)  rate of asymptotic regularity.

In the sequel, we present generalizations of these quantitative results to $W$-hyperbolic spaces. Their proofs follow closely the proofs of the corresponding results from \cite{Leu07a}, thus we omit them.

\subsubsection{Main results}\label{habil-jucs-main-results}

Before stating our main theorem, let us recall some terminology. If $(a_n)_{n\geq 1}$ is a convergent sequence of real numbers, then a function $\gamma:(0,\infty)\to\N^*$ is called a {\em Cauchy modulus} of $(a_n)$ if 
\beq
\forall \eps>0\,\forall n\in\N^*\left(|a_{\gamma(\eps)+n}-a_{\gamma(\eps)}| < \eps\right)\label{def-mod-Cauchy}.
\eeq

\begin{theorem}\label{habil-Halpern-main-thm}(see \cite[Theorem 3]{Leu07a})\\
Let $C$ be a  convex subset of a W-hyperbolic space $(X,d,W)$ and  $T:C\to C$ be nonexpansive. 
Assume that 
\beq
\lim_{n\to\infty} \lambda_n =0, \quad \sum_{n=1}^\infty \lambda_n=\infty \quad \text{and }\sum_{n=1}^\infty|\lambda_{n+1}-\lambda_n| \text{ converges}. \label{habil-Halpern-lambda-n-hyp}
\eeq

Then $\ds\limn d(x_n,Tx_n)=0$ for every $x\in C$  with the property that $(x_n)$ is bounded.\\[0.2cm]
\noindent Furthermore, let $\alpha:(0,\infty)\to\N^*$ be a rate of convergence of $(\lambda_n)$, $\beta:(0,\infty)\to\N^*$ be  a Cauchy modulus of $s_n:=\ds\sum_{i=1}^n|\lambda_{i+1}-\lambda_i|$ and $\theta:\N^*\to\N^*$ be a rate of divergence of $\ds\sum_{n=1}^\infty \lambda_n$.

Then 
\bua
 \forall \eps\in(0,2) \forall n\ge \Phi(\alpha,\beta,\theta,M,\eps)\ \bigg(d(x_n,Tx_n)< \eps\bigg), \label{habil-Halpern-main-thm-hyperbolic-conclusion}
\eua
where $\ds\Phi(\alpha,\beta,\theta,M,\eps)=\max\left\{\theta\biggl(\beta\left(\frac\eps {8M}\right)+1+\left\lceil\ln\left(\frac{8M}\eps\right)\right\rceil\biggr),\,\,\alpha\left(\frac\eps {4M}\right)\right\},$ with $M\in\N^*$ such that $M\geq d(x,Tx), d(x_n,x)$ for all $n\ge 1$.\\
\end{theorem}

If $C$ is bounded with diameter $d_C$, we can take $M:=d_C$ in the above theorem. 

\bcor\label{habil-Halpern-hyperbolic-C-bounded}
Let $(X,d,W), (\lambda_n), C, T,\alpha, \beta,\theta$  be as in the hypothesis of Theorem \ref{habil-Halpern-main-thm}. Assume moreover that $C$ is bounded with diameter $d_C$.

Then  $T$ is $\lambda_n$-asymptotically regular and for all $x\in C$,
\bua
 \forall \eps\in(0,2) \forall n\ge \Phi(\alpha,\beta,\theta,d_C,\eps)\ \bigg(d(x_n,Tx_n)< \eps\bigg), 
\eua
where $\Phi(\alpha,\beta,\theta,d_C,\eps)$ is defined as in Theorem \ref{habil-Halpern-main-thm} by replacing $M$ with $d_C$.
\ecor

The rate of asymptotic regularity can be simplified for $(\lambda_n)$ nonincreasing.

\begin{corollary}\label{habil-Halpern-hyperbolic-lambda-nonincreasing}
Let $(X,d,W), C, T$ be as in the hypothesis of Theorem \ref{habil-Halpern-main-thm}. 
Assume that $(\lambda_n)_{n\geq 1}$ is a {\em nonincreasing} sequence in $[0,1]$  such that $\ds\lim_{n\to\infty} \lambda_n =0$ and $\ds\sum_{n=1}^\infty \lambda_n$ is divergent.

Then $\ds\limn d(x_n,Tx_n)=0$ for every $x\in C$  with the property that $(x_n)$ is bounded.\\[0.2cm]
Furthermore, if $\alpha:(0,\infty)\to\N^*$ is a rate of convergence of $(\lambda_n)$ and $\theta:\N^*\to\N^*$ is a rate of divergence of $\ds\sum_{n=1}^\infty \lambda_n$, then
\bua
 \forall \eps\in(0,2)\forall n\ge \Psi(\alpha,\theta,M,\eps)\ \bigg(d(x_n,Tx_n)< \eps\bigg), 
\eua
where $\ds\Psi(\alpha,\theta,M,\eps)=\max\left\{\theta\left(\alpha\left(\frac\eps {8M}\right)+1+\left\lceil\ln\left(\frac{8M}\eps\right)\right\rceil\right),\,\,\alpha\left(\frac\eps {4M}\right)\right\},$ with $M\in\N^*$ such that $M\geq d(x,Tx), d(x_n,x)$ for all $n\ge 1$.
\end{corollary}
\begin{proof}
Remark that $(\lambda_n)$ nonincreasing implies that $\ds \sum_{n=1}^\infty|\lambda_{n+1}-\lambda_n|$ converges with Cauchy modulus $\alpha$. Apply Theorem \ref{habil-Halpern-main-thm} with $\beta:=\alpha$.
\end{proof}

Finally, by taking $\lambda_n=1/n$, we get an exponential (in $1/\eps$) rate of asymptotic regularity.

\bcor\label{habil-Halpern-exponential}
Let $C$ be a  convex subset of a W-hyperbolic space $(X,d,W)$ and  $T:C\to C$ be nonexpansive.  Assume that $\ds\lambda_n=\frac 1n$ for all $n\geq 1$.

Then  $\ds\limn d(x_n,Tx_n)=0$ for all $x\in C$ and, moreover, 
\bua
 \forall \eps\in(0,2)\forall n\ge \Phi(d_C,\eps)\ \bigg( d(x_n,Tx_n)< \eps\bigg), 
\eua
where $\ds\Phi(d_C,\eps)=\exp\left(\ln 4\cdot\left(\frac{16d_C}\eps+3\right)\right)$.
\ecor

\subsection{Rates of asymptotic regularity for Ishikawa iterations}\label{habil-app-fpt-Ishikawa}

\noindent Let $C$  be a convex subset of a  W-hyperbolic space $(X,d,W)$ and $T:C\to C$ be nonexpansive. 

As in the case of normed spaces, we can define the {\em Ishikawa iteration} starting with $x\in C$ by 
\beq
x_0:=x, \quad x_{n+1}=(1-\lambda_n)x_n\oplus \lambda_nT((1-s_n)x_n\oplus s_nTx_n),
\eeq
where $(\lambda_n),(s_n)$ are sequences in $[0,1]$. By letting $s_n=0$ for all $n\in\N$, we get the Krasnoselski-Mann iteration as  a special case.

We shall use the following notations 
\[y_n := (1-s_n)x_n\oplus s_nTx_n\]
and
\[T_n:C\to C, \quad T_n(x)= (1-\lambda_n)x\oplus \lambda_nT((1-s_n)x\oplus s_nTx).\]
Then 
\[x_{n+1}=(1-\lambda_n)x_n\oplus \lambda_nTy_n=T_nx_n\]
and it is easy to see that $Fix(T)\se Fix(T_n)$ for all $n\in\N$.

The following lemma collects some basic properties of Ishikawa iterations; we refer to \cite{Leu08} for the proofs.

\blem\label{habil-Ishikawa-useful}
\be
\item $d(x_{n+1},Tx_{n+1})\leq (1+2s_n(1-\lambda_n))d(x_n,Tx_n)$ for all $n\in\N$;
\item $T_n$ is nonexpansive for all $n\in\N$;
\item\label{habil-Ishikawa-useful-p} For all  $p\in Fix(T)$, the sequence $(d(x_n,p))$ is nonincreasing and for all $n\in\N$, 
\[d(y_n,p)\leq d(x_n,p)\quad \text{and}\quad d(x_n,Ty_n),d(x_n,Tx_n)\leq 2d(x_n,p).\]
\ee
\elem

We consider the important problem of asymptotic regularity, this time associated with the Ishikawa iterations:
\[\lim_{n\to\infty}d(x_n,Tx_n)=0.\]

Our point of departure is Theorem \ref{intro-Ishikawa-as-reg}. We recall it here.

\bthm
Let $X$ be a uniformly convex Banach space or a $CAT(0)$ space, $C\se X$ a  bounded closed convex subset and $T:C\to C$ be nonexpansive.  Assume that $\ds\sum_{n=0}^\infty\lambda_n(1-\lambda_n)$ diverges, $\limsup_n s_n<1$ and $\ds\sum_{n=0}^\infty s_n(1-\lambda_n)$ converges.

Then $\limn d(x_n,Tx_n)=0$ for all $x\in C$.
\ethm

Using proof mining methods, we obtained \cite{Leu08} a quantitative version (Theorem \ref{habil-Ishikawa-quant-as-reg}) of a two-fold generalization of the above result:
\begin{itemize}
\item[-] firstly, we  consider $UCW$-hyperbolic spaces;
\item[-] secondly, we assume that $Fix(T)\ne\emptyset$ instead of assuming the boundedness of $C$.
\end{itemize}

The idea is to combine methods used in \cite{Leu07} (see Subsection \ref{habil-app-fpt-CAT0}) to obtain effective rates of asymptotic regularity for Krasnoselski-Mann iterates with the ones used in \cite{Leu07a} (see Subsection \ref{habil-app-fpt-Halpern})  to get rates of asymptotic regularity for Halpern iterates. 

In this way, we provided for the first time (even for the normed case) effective rates of asymptotic regularity for the Ishikawa iterates, i.e. rates of convergence of $(d(x_n,Tx_n))$ towards $0$. 

For bounded $C$ (Corollary \ref{habil-Ishikawa-bounded-C}), the rate of asymptotic regularity is uniform in  the nonexpansive mapping $T$ and the starting point $x\in C$ of the iteration, and it depends  on $C$ only via its diameter and on the space $X$ only via a monotone modulus of uniform convexity.

\subsubsection{Main resulta}

\bprop\label{habil-Ishikawa-liminf-xn-Tyn=0}\cite{Leu08}
Let $C$ be a convex subset of a $UCW$-hyperbolic space $(X,d,W)$ and $T:C\rightarrow C$ nonexpansive with $Fix(T)\ne\emptyset$. Assume that $\ds\sum_{n=0}^\infty\lambda_n(1-\lambda_n)$ is divergent.

Then $\liminf_n d(x_n,Ty_n)=0$ for all $x\in C$.\\[0.1cm]
\noindent Furthermore, if $\eta$ is a monotone  modulus of uniform convexity and $\theta :\N\to\N$ is a rate of divergence for  $\ds\sum_{n=0}^\infty\lambda_n(1-\lambda_n)$, then 

for all  $x\in C,\eps>0,k\in\N$ there exists $N\in\N$ satisfying 
\beq 
k\leq N\leq h(\varepsilon,k,\eta,b,\theta) \text{~~and~~} d(x_N,Ty_N)<\varepsilon,\label{habil-Ishikawa-quant-liminf-d-xn-Tyn}
\eeq
 where
\[h(\varepsilon,k,\eta,b,\theta)=\left\{\begin{array}{ll}\displaystyle \theta\left(\left\lceil\frac{b+1}{\varepsilon\cdot\eta\left(b,\displaystyle\frac{\varepsilon}{b}\right)}\right\rceil+k\right)  & \text{for~ } \varepsilon \le 2b,\\
k & \text{otherwise,}
\end{array}\right.
\]
with $b>0$ such that $b\ge d(x,p)$ for some $p\in Fix(T)$.
\eprop

As an immediate consequence of the above proposition, we get a rate of asymptotic regularity for the Krasnoselski-Mann iterates that is basically the same with the one obtained in Theorem \ref{habil-main-Groetsch-thm} .

\bcor
Let $(X,d,W),\eta,C,T,b,(\lambda_n),\theta$ be as in the hypotheses of Proposition 
\ref{habil-Ishikawa-liminf-xn-Tyn=0} and assume that $(x_n)$ is the Krasnoselski-Mann iteration starting with $x$, defined by (\ref{app-KM-lambda-n-def-hyp}). 

Then $\limn d(x_n,Tx_n)=0$ for all $x\in C$ and  
\beq
\forall \eps>0\, \forall n\ge \Phi(\eps,\eta,b,\theta)\bigg(d(x_n,Tx_n)<\eps\bigg), \label{habil-Ishikawa-KM-rate-as-reg}
\eeq
where $\Phi(\varepsilon,\eta,b,\theta)=h(\varepsilon,0,\eta,b,\theta)$,
with $h$ defined as above.
\ecor

\bprop\label{habil-Ishikawa-liminf-xn-Txn=0}\cite{Leu08}
In the hypotheses of Proposition \ref{habil-Ishikawa-liminf-xn-Tyn=0}, assume  that $\limsup_n s_n<1$.

 Then $\liminf_n d(x_n,Tx_n)=0$ for all $x\in C$.\\[0.1cm]
\noindent Furthermore, if $L,N_0\in\N$ are such that $\ds s_n\leq 1-\frac1L$ for all $n\geq N_0$, then 

for all  $x\in C,\eps>0,k\in\N$ there exists $N\in\N$ such that 
\beq
k\leq N\leq \Psi(\varepsilon,k,\eta,b,\theta,L,N_0) \text{~~and~~} d(x_N,Tx_N)<\varepsilon, \label{habil-quant-Ishikawa-liminf-d-xn-Txn}
\eeq
where $\ds \Psi(\varepsilon,k,\eta,b,\theta,L,N_0)=h\left(\frac{\eps}L, k+N_0,\eta,b,\theta\right)$, with $h$ defined as in Proposition \ref{habil-Ishikawa-liminf-xn-Tyn=0}.
\eprop

As a corollary, we obtain an approximate fixed point bound for the nonexpansive mapping $T$.

\bcor\label{habil-Ishikawa-AFP-bound}
In the hypotheses of Proposition \ref{habil-Ishikawa-liminf-xn-Txn=0},
\beq
\forall \eps>0\, \exists N\le \Phi(\varepsilon,\eta,b,\theta,L,N_0)\bigg( d(x_N,Tx_N)<\varepsilon\bigg),
\eeq
where $\Phi(\eps,\eta,b,\theta,L,N_0)=\Psi(\eps,0,\eta,b,\theta,L,N_0)$,  with $\Psi$ defined as above.
\ecor

The following theorem is the main result of \cite{Leu08}.

\bthm\label{habil-Ishikawa-quant-as-reg}
Let $C$ be a convex subset of a $UCW$-hyperbolic space $(X,d,W)$ and $T:C\rightarrow C$ nonexpansive with $Fix(T)\ne\emptyset$. Assume that $\ds\sum_{n=0}^\infty\lambda_n(1-\lambda_n)$ diverges, $\limsup_n s_n<1$ and $\ds\sum_{n=0}^\infty s_n(1-\lambda_n)$ converges. 

Then $\limn d(x_n,Tx_n)=0$ for all $x\in C$.\\[0.2cm]
Furthermore, if $\eta$ is a monotone  modulus of uniform convexity, $\theta$ is a rate of divergence for $\ds\sum_{n=0}^\infty\lambda_n(1-\lambda_n)$, $L,N_0$ are such that $\ds s_n\leq 1-\frac1L$ for all $n\geq N_0$  and $\gamma$ is a Cauchy modulus for $\ds\sum_{n=0}^\infty s_n(1-\lambda_n)$, then for all  $x\in C$,
\beq
\forall \eps>0\forall n\ge \Phi(\eps,\eta,b,\theta,L,N_0,\gamma)\bigg(d(x_n,Tx_n)<\eps\bigg),
\eeq
where 
\[\Phi(\varepsilon,\eta,b,\theta,L,N_0,\gamma)= \left\{\begin{array}{ll}\!\!\displaystyle \theta\left(\left\lceil\frac{2L(b+1)}{\varepsilon\cdot\eta\left(b,\displaystyle\frac{\varepsilon}{2Lb}\right)}\right\rceil+\gamma\left(\frac{\eps}{8b}\right)+N_0+1\right)  & \!\!\!\text{for~} \eps\le 4Lb,\\
\!\!\ds \gamma\left(\frac{\eps}{8b}\right)+N_0+1 & \!\!\!\text{otherwise,}
\end{array}\right.
\]
with $b>0$ such that $b\ge d(x,p)$ for some $p\in Fix(T)$.
\ethm

\begin{fact}\label{habil-Ishikawa-quant-as-reg-tilde-eta}
In the hypotheses of Theorem \ref{habil-Ishikawa-quant-as-reg}, assume, moreover, that  $\eta(r,\varepsilon)=\varepsilon\cdot\tilde{\eta}(r,\varepsilon)$ such that $\tilde{\eta}$ increases with $\varepsilon$ (for a fixed $r$). Then the bound $\Phi(\varepsilon,\eta,b,\theta,L,N_0,\gamma)$ can be replaced  for $\eps\le 4Lb$ with
\[ \tilde{\Phi}(\varepsilon,\eta,b,\theta,L,N_0,\gamma)=\theta\left(\left\lceil\frac{L(b+1)}{\varepsilon\cdot\tilde{\eta}\left(b,\displaystyle\frac{\varepsilon}{2Lb}\right)}\right\rceil+\gamma\left(\frac{\eps}{8b}\right)+N_0+1\right).\] 
\end{fact}

For bounded $C$, we get an  effective rate of asymptotic regularity which depends on the error $\varepsilon$, on the modulus of uniform convexity $\eta$,  on the diameter $d_C$ of $C$, on $(\lambda_n), (s_n)$ via $\theta,L,N_0,\gamma$, but does not depend  on the nonexpansive mapping $T$, the starting point $x\in C$ of the iteration or other data related with $C$ and $X$.

\begin{corollary}\label{habil-Ishikawa-bounded-C}
 Let $(X,d,W)$ be a complete $UCW$-hyperbolic space, $C\se X$  a {\em bounded} convex closed subset with diameter $d_C$ and $T:C\rightarrow C$ nonexpansive.\\
Assume that $\eta,(\lambda_n),(s_n),\theta, L,N_0,\gamma$ are as in the hypotheses of Theorem \ref{habil-Ishikawa-quant-as-reg}. 

Then $\limn d(x_n,Tx_n)=0$ for all $x\in C$ and, moreover,
\[\forall \varepsilon >0\,\forall n\ge \Phi(\eps,\eta,d_C,\theta,L,N_0,\gamma)\,
\bigg(d(x_n,Tx_n) <\varepsilon\bigg), \]
where $\Phi(\eps,\eta,d_C,\theta,L,N_0,\gamma)$ is defined as in Theorem \ref{habil-Ishikawa-quant-as-reg} by replacing $b$ with $d_C$.
\ecor
\begin{proof}
We can apply  Corollary \ref{UCW-BGK}, the generalization of Browder-Goehde-Kirk Theorem to complete $UCW$-hyperbolic spaces, to get that $Fix(T)\ne\emptyset$. Moreover, $d(x,p)\le d_C$ for any $x\in C$ and any $p\in Fix(T)$, hence we can take $b:=d_C$ in Theorem \ref{habil-Ishikawa-quant-as-reg}. 
\end{proof}

The rate of asymptotic regularity can be further simplified for constant $\lambda_n=\lambda\in(0,1)$.

\begin{corollary}\label{habil-Ishikawa-bounded-C-constant-lambda}
 Let $(X,d,W),\eta,C,d_C,T$ be as in the hypotheses of Corollary \ref{habil-Ishikawa-bounded-C}.
Assume that $\lambda_n=\lambda\in(0,1)$ for all $n\in\N$. \\
Furthermore, let $L,N_0$ be such that $\ds s_n\leq 1-\frac1L$ for all $n\geq N_0$  and assume that the series  $\ds\sum_{n=0}^\infty s_n$ converges  with Cauchy modulus $\delta$. 

Then for all $x\in C$, 
\beq
\forall \eps>0\forall n\ge \Phi(\eps,\eta,d_C,\lambda,L,N_0,\delta)\bigg(d(x_n,Tx_n)<\eps\bigg),
\eeq
where 
\[\Phi(\varepsilon,\eta,d_C,\lambda,L,N_0,\delta)= \left\{\begin{array}{ll}\!\displaystyle \left\lceil\frac{1}{\lambda(1-\lambda)}\cdot \frac{2L(d_C+1)}{\varepsilon\cdot\eta\left(d_C,\displaystyle\frac{\varepsilon}{2Ld_C}\right)}\right\rceil+M & \!\!\!\text{for~ } \varepsilon \le 4Ld_C,\\
\! M & \!\!\!\text{otherwise,}
\end{array}\right.
\]
with $\ds M=\delta\left(\frac{\eps}{8d_C(1-\lambda)}\right)+N_0+1$.

Moreover, if $\eta(r,\varepsilon)$ can be written as $\eta(r,\varepsilon)=\varepsilon\cdot\tilde{\eta}(r,\varepsilon)$ such that $\tilde{\eta}$ increases with $\varepsilon$ (for a fixed $r$),  then the bound $\Phi(\varepsilon,\eta,d_C,\lambda,L,N_0,\delta)$ can be replaced  for $\eps\le 4Ld_C$ with
\[
\Phi(\varepsilon,\eta,d_C,\lambda,L,N_0,\delta)=  \left\lceil\frac{1}{\lambda(1-\lambda)}\cdot \frac{L(d_C+1)}{\varepsilon\cdot\tilde{\eta}\left(d_C,\displaystyle\frac{\varepsilon}{2Ld_C}\right)}\right\rceil+M.
\]
\ecor

As we have already seen, $CAT(0)$ spaces are $UCW$-hyperbolic spaces with a modulus of uniform convexity $\ds \eta(r,\varepsilon)=\frac{\varepsilon^2}{8}$, which has the form required in Remark \ref{habil-Ishikawa-quant-as-reg-tilde-eta}. Thus, the above result can be applied to $CAT(0)$ spaces. 

\begin{corollary}\label{habil-Ishikawa-CAT0-constant-lambda}
Let $X$ be a $CAT(0)$ space, $C\se X$  a {\em bounded} convex closed subset with diameter $d_C$ and $T:C\rightarrow C$ nonexpansive. Assume that $\lambda_n=\lambda\in(0,1)$ for all $n\in\N$ and $L,N_0,(s_n),\delta$ are as in the hypotheses of Corollary \ref{habil-Ishikawa-bounded-C-constant-lambda}

Then $\limn d(x_n,Tx_n)=0$ for all $x\in C$ and, moreover
\beq
\forall \eps>0\forall n\ge \Phi(\eps,d_C,\lambda,L,N_0,\delta)\bigg(d(x_n,Tx_n)<\eps\bigg),
\eeq
where 
\[\Phi(\varepsilon,d_C,\lambda,L,N_0,\delta)= \left\{\begin{array}{ll}
\left\lceil\ds \frac{D}{\eps^2}\right\rceil+M,
& \text{for~ } \varepsilon \le 4Ld_C,\\
M & \text{otherwise,}
\end{array}\right.
\]
with $\ds M=\delta\left(\frac{\eps}{8d_C(1-\lambda)}\right)+N_0+1,\,\, D=\ds\frac{16L^2d_C(d_C+1)}{\lambda(1-\lambda)}$.
\ecor

\subsection{Asymptotically nonexpansive mappings in $UCW$-spaces}\label{habil-app-fpt-as-ne}

In this subsection, we present results on fixed point theory of asymptotically nonexpansive mappings in the very general setting of $UCW$-spaces. These results were obtained by Kohlenbach and the author in  \cite{KohLeu08a} .

In the following, $(X,d,W)$ is a $UCW$-hyperbolic space and $C\se X$ a convex subset of $X$. Let us recall that a mapping $T:C\to C$ is said to be asymptotically nonexpansive with sequence $(k_n)$ in $[0,\infty)$ if $\ds\limn k_n =0$ and 
\[
d(T^n x,T^ny) \leq (1+k_n)d(x,y) \quad  \text{for all }n\in\N, x,y\in C.
 \] 

The first main result is a generalization  to $UCW$-spaces of Goebel-Kirk Theorem \ref{habil-as-ne-FPP-ucBanach} and Kirk Theorem \ref{habil-as-ne-FPP-CAT0}.

\bthm\label{habil-FPP-ass-ne-UCW}\cite{KohLeu08a}
Closed convex and bounded subsets of complete $UCW$-hyperbolic spaces have the FPP for asymptotically nonexpansive mappings.
\ethm

\noindent Our proof generalizes Goebel and Kirk's proof of Theorem \ref{habil-as-ne-FPP-ucBanach} and, as a consequence, we obtain also an elementary proof of Theorem \ref{habil-as-ne-FPP-CAT0}. 

In fact, as it was already pointed out for uniformly convex normed spaces in \cite{KohLam04}, the proof of the FPP can be transformed into an elementary proof of the AFPP,  which does not need the completeness of $X$ or the closedness of $C$.

\bprop\label{habil-AFPP-ass-ne-uchyp}
Bounded  convex subsets of $UCW$-hyperbolic spaces have the AFPP for asymptotically nonexpansive mappings.
\eprop

The main part of this subsection is devoted  to getting a quantitative version  of an asymptotic regularity theorem for the Krasnoselski-Mann iterations of asymptotically nonexpansive mappings.

As in the case of normed spaces, the {\em Krasnoselski-Mann iteration} starting from $x\in C$ is defined by:   
\begin{equation}
x_0:=x, \quad x_{n+1}:=(1-\lambda_n)x_n \oplus\lambda_n T^nx_n, \label{habil-ass-ne-KM-lambda-n-def-hyp}
\end{equation}
where $(\lambda_n)$ is a sequence in $[0,1]$.

We apply proof mining techniques to the following generalization to $UCW$-hyperbolic spaces of Theorem \ref{intro-as-ne-as-reg}. 

\begin{theorem} \label{habil-as-ne-original-theorem}
Let $C$ be a convex subset  of a $UCW$-hyperbolic space $(X,d,W)$ and $T:C\rightarrow C$  be asymptotically nonexpansive with sequence $(k_n)\in [0,\infty)$ satisfying $\sum\limits^{\infty}_{i=0} k_i<\infty$. Assume that $(\lambda_n)$ be a sequence in $[a,b]$ for $0<a<b<1.$

If  $Fix(T)\ne\emptyset$, then $T$ is $\lambda_n$-asymptotically regular.
\end{theorem} 

There does not seem to exist a computable rate of asymptotic regularity in this case; in \cite{KohLam04} it is shown that the proof even holds for asymptotically weakly-quasi nonexpansive functions
for which one can prove that no uniform effective rate does exist. Anyway, the general 
logical metatheorems from Section \ref{logical-meta} guarantee (see also the logical discussion below) 
effective uniform bounds on the so-called  {\em Herbrand normal form} or {\em no-counterexample 
interpretation} of the convergence i.e. on 
\beq
\forall \varepsilon >0\,\forall g:\N\to\N \,\exists N\in\N 
\,\forall m\in [N,N+g(N)] \big( d(x_m,Tx_m)<\varepsilon\big),
\eeq
which (ineffectively) is equivalent to the fact that $\limn d(x_n,Tx_n)=0$. Here $[n,n+m]:=\{ n,n+1,n+2,\ldots,n+m\}$. 

This coincides with what recently has been advocated under the name {\em metastability} or {\em finite convergence} in an essay posted by Terence Tao \cite{Tao07} (see also \cite{Tao06,Tao08}).  Thus, in Tao's terminology, the logical metatheorems guarantee an effective uniform bound on the {\em metastability of} $(d(x_n,Tx_n))$.

In the sequel, we give a quantitative version of the above theorem, generalizing to $UCW$-hyperbolic spaces the logical analysis and the results of Kohlenbach and Lambov \cite{KohLam04}. As a consequence, for $CAT(0)$ spaces we get a quadratic bound on the approximate fixed point property of $(x_n)$ (see Corollary \ref{habil-as-ne-bounded-C-CAT0}). We recall that  for nonexpansive mappings, a quadratic rate of asymptotic regularity for the Krasnoselski-Mann iterations was obtained in Corollary \ref{habil-CAT0-constant-lambda}. 

These  results can also be seen as an instance of {\em `hard analysis'} as proposed by Tao in his essay \cite{Tao07}. 

\subsubsection{Logical discussion}\label{habil-as-ne-section-logic}

It is easy to see that the proof of the above theorem can be formalized in ${\cal A}^{\omega}[X,d,UCW,\eta]_{-b}$, the theory  of $UCW$-hyperbolic spaces. Unfortunately, the conclusion of the above theorem, that for all $x\in C$, $\limn d(x_n,Tx_n)=0$, i.e.  
\beq
\forall \eps>0\exists N\in\N\forall p\in\N\big(d(x_{N+p},Tx_{N+p})<\eps \big),\label{habil-as-reg-fef}
\eeq
is a $\forall\exists\forall$-formula, so it has a too complicated logical form for the logical metatheorems to apply. 
In the case of nonexpansive mappings, due to the fact that $(d(x_n,Tx_n))$ is nonincreasing,  (\ref{habil-as-reg-fef}) could be rewritten as
\beq
\forall \eps>0\exists N\in\N\big(d(x_{N},Tx_{N})<\eps \big)\label{habil-as-reg-ef},
\eeq
which has the required $\forall\exists$-form. This is no longer possible for asymptotically nonexpansive mappings, since for this class of mappings the sequence $(d(x_n,Tx_n))$ is not necessarily nonincreasing. 

\blem\label{habil-lemma-Herbrand-converges-0}
The following are equivalent
\be
\item[] $(1) \quad \forall \eps>0\,\exists N\in\N\,\forall p\in\N\big(d(x_{N+p},Tx_{N+p})<\eps\big)$;
\item[]  $(2) \quad \forall \eps>0\, \exists N\in\N\, \forall m\in\N \,\forall i\in [N,N+m]\big(d(x_{i},Tx_{i})<\eps\big)$;
\item[]  $(2^H) \quad \forall \eps>0\, \forall g:\N\to\N\, \exists N\in\N\,\forall i\in [N,N+g(N)]\big(d(x_{i},Tx_{i})<\eps\big)$.
\ee
\elem
\begin{proof}
$(1)\Lra (2)$ and $(2)\Ra (2^H)$ are obvious. Assume that  $(2^H)$ is true. If $(2)$ would be false, 
then for some $\eps>0$
\[ \forall n\in\N\,\exists m_n\in\N \,\exists i \in [n,n+m_n]\, (d(x_i,Tx_i)\ge \eps). \] 
Define $g(n):=m_n.$ Then $(2^H)$ applied to $g$ leads to a contradiction. 
\end{proof}
The transformed version $(2^H)$ is the {\em Herbrand normal form} of $(2)$ or  the {\em no-counterexample 
interpretation} \cite{Kre51,Kre52} of $(2)$, well-known in mathematical logic. The good news is that $(2^H)$ has the $\forall\exists$-form, as the universal quantifier over $i$ is bounded. Obviously, since the above argument is ineffective, a bound on $\exists N\in\N$ in $(2^H)$ cannot be converted effectively into a bound on $\exists N\in\N$ in $(2)$.  

As it suffices to consider only mappings $T:X\to X$, it is easy to see that  $\mathcal{A}^\omega[X,d,UCW,\eta]_{-b}$ proves the following formalized version of Theorem \ref{habil-as-ne-original-theorem}:

\[ \ba{l} \forall\, g:\N\to\N\,\forall\, \eps>0\,\forall\, K,L\in\N\,\forall\, g:\N\to\N\, \forall\, (\lambda_n)\in [0,1]^{\N}\,\forall \,
(k_n)\in [0,K]^{\N}\,\forall \,x\in X\,\forall\, T:X\to X \\ 
\biggl(Fix(T)\not=\emptyset  \, \si \forall n\in\N\,\forall y,z\in X\bigg(d(T^ny,T^nz)\le  (1+k_n)d(y,z)\bigg) \, \si \,\forall n\in\N\ \left(\ds\sum^n_{i=0} k_i\le K\right)\, \\
\quad\quad \si\, L\ge 2\,\wedge\, \forall n\in\N\ \left(\ds\frac{1}{L} \le 
\lambda_n\le 1-\frac{1}{L}\right)\, \\
\hfill \rightarrow \exists N\in\N\,\forall i\in [N,N+g(N)]\bigg(d(x_{i},Tx_{i})<\eps\bigg)\biggr). \ea \]
Moreover, the asymptotic nonexpansivity of $T$  and the fact that $k_1\le K$ imply that $T$ is $(1+K)$-Lipschitz. Thus, we can apply Corollary \ref{meta-Groetsch} which guarantees the extractability of a computable 
bound $\Phi$ on $\exists N\in\N$ in the conclusion 
\[ \ba{l} 
\hspace*{-1cm}\,\forall b\in\N\,\forall\, g:\N\to\N\,\forall\, \eps>0\,\forall\, K,L\in\N\, \forall\, (\lambda_n)\in [0,1]^{\N}\,\forall \,
(k_n)\in [0,K]^{\N}\,\forall \,x\in X\,\forall\, T:X\to X \\ 
\biggl(\forall \delta>0\bigg(Fix_\delta(T,x,b)\ne\emptyset\bigg)  \, \si \forall n\in\N\,\forall y,z\in X\bigg(d(T^ny,T^nz)\le  (1+k_n)d(y,z)\bigg) \, \\
\qquad\si \,\forall n\in\N\ \left(\ds\sum^n_{i=0} k_i\le K\right)\,  \si\, L\ge 2\,\wedge\, \forall n\in\N\ \left(\ds\frac{1}{L} \le 
\lambda_n\le 1-\frac{1}{L}\right)\, \\
\hfill \rightarrow \exists N\le \Phi(\eps,K,L,b,\eta,g) \,\forall i\in [N,N+g(N)]\bigg(d(x_{i},Tx_{i})<\eps\bigg)\biggr). \ea \]

Thus, the premise that $T$ has fixed points is weakened to $T$ having approximate fixed points in a $b$-neighborhood of $x$ and the bound $\Phi$ depends, in addition to $\eps,K,L,\eta$, on $b\in\N$ and $g:\N\to \N$.
By taking $g(n)\equiv 0$, we get an approximate fixed point bound for $T$. 

We refer to \cite[Section 5]{KohLeu08a} for details on the above logical discussion.

\subsubsection{Main results on asymptotic regularity}\label{habil-as-ne-main-results}

The following quantitative version of Theorem \ref{habil-as-ne-original-theorem} is the second main result of the paper \cite{KohLeu08a}.

\begin{theorem} \label{habil-as-ne-Herbrand}
Let $C$ be a convex subset of a $UCW$-hyperbolic space $(X,d,W)$  and $T:C\to C$ be asymptotically nonexpansive with sequence $(k_n)$. \\
Assume that $\eta$ is a monotone modulus of uniform convexity $\eta$,  $K\in\N$ is such that $\ds\sum_{n=0}^\infty k_n\le K$ and  $L\in\N, L\geq 2$  satisfies $\ds\frac{1}{L}\leq \lambda_n \leq 1-\frac{1}{L}$ for all $n\in\N$.\\
Let $x\in C$ and $b>0$ be such that $T$ has approximate fixed points in a $b$-neighborhood of $x$.

Then $\ds \limn d(x_n,Tx_n)=0$  and, moreover, for all $\eps\in (0,1]$ and all $g:\N\to\N$,
\beq
\exists N\le \Phi(K, L, b, \eta, \eps, g)
\forall m\in [N, N+g(N)] \bigg(d(x_m,Tx_m) < \eps \bigg), \label{habil-as-ne-Herbrand-conclusion}
\eeq 
where $\Phi(K, L, b, \eta, \eps, g)=\ds h^M(0)$, with 
\[\ba{l}
h(n) = g(n+1)+n+2,\quad M =\left\lceil\ds\frac {3\left(5KD+D+\frac{11}{2}\right)}{\delta}\right\rceil, \quad  
\ds D = e^K\left(b+ 2\right),\\
\delta = \ds\frac{\eps}{L^2f(K)}\cdot\eta\left((1+K)D+1,\frac{\eps}{f(K)((1+K)D+1)}\right),\quad 
f(K) = 2(1+(1+K)^2(2+K)).
\ea\]
Moreover, $N=h^i(0)+1$ for some $i<M$.\\
\end{theorem}

\begin{fact}\label{habil-as-ne-theorem-tilde-eta}
Assume, moreover, that  $\eta(r,\varepsilon)$ can be written as $\eta(r,\varepsilon)=\varepsilon\cdot\tilde{\eta}(r,\varepsilon)$ such that $\tilde{\eta}$ increases with $\varepsilon$ (for a fixed $r$). Then we can replace $\eta$ with $\tilde{\eta}$ in the bound $\Phi(K, L, b, \eta, \eps, g)$.
\end{fact}

We give now some consequences.  By taking  $g(n)\equiv 0$, we obtain an approximate fixed point bound for the asymptotically nonexpansive mapping $T$.

\bcor\label{habil-as-ne-AsReg-as-ne}
Assume $(X,d,W),\eta, C, T, (k_n), K, (\lambda_n), L$ are  as in the hypotheses of Theorem \ref{habil-as-ne-Herbrand}. Let $x\in C$ and $b>0$ be such that $T$ has approximate fixed points in a $b$-neighborhood of $x$.

Then $\ds \limn d(x_n,Tx_n)=0$  and, moreover, 
\beq
\forall\eps\in (0,1]\,\exists N\le \Phi(K, L, b, \eta, \eps)
\bigg(d(x_N,Tx_N) < \eps \bigg),
\eeq 
where $\Phi(K, L, b, \eta, \eps)=\ds 2M$ and $M, D, \theta, f(K)$ are as in Theorem \ref{habil-as-ne-Herbrand}. 
\ecor

Furthermore, if $C$ is bounded with diameter $d_C$, $C$ has the AFPP for asymptotically nonexpansive mappings by Proposition \ref{habil-AFPP-ass-ne-uchyp}, so $T$ has approximate fixed points in a $d_C$-neighborhood of $x$ for all $x\in C$. Hence, we get asymptotic regularity and an explicit approximate fixed point bound.

\begin{corollary}\label{habil-as-ne-bounded-C}
Let $(X,d,W),\eta, C, T, (k_n), K, (\lambda_n), L$ be  as in the hypotheses of Theorem \ref{habil-as-ne-Herbrand}. Assume moreover that $C$ is bounded with  diameter $d_C$.

Then $T$ is $\lambda_n$-asymptotically regular and  the following holds for all $x\in C$:
\beq
\forall\eps\in (0,1]\exists N\le \Phi(K, L, d_C, \eta, \eps)
\bigg(d(x_N,Tx_N) < \eps \bigg),
\eeq 
where $\Phi(K, L, d_C, \eta, \eps)$ is defined as in Theorem \ref{habil-as-ne-AsReg-as-ne} by replacing $b$ with $d_C$.
\end{corollary}

Finally, in the case of convex bounded subsets of $CAT(0)$ spaces, we get a quadratic (in $1/\varepsilon$) approximate fixed point bound. 

\begin{corollary}\label{habil-as-ne-bounded-C-CAT0}
Let $X$ be a $CAT(0)$ space, $C$ be a convex bounded subset of $X$ with diameter $d_C$ and $T:C\to C$ be asymptotically nonexpansive with sequence $(k_n)$. \\
Assume that $K\in\N,L\in\N, L\geq 2$  are such that $\ds\sum_{n=0}^\infty k_n\le K$ and $\ds\frac{1}{L}\leq \lambda_n \leq 1-\frac{1}{L}$ for all $n\in\N$.

Then $T$ is $\lambda_n$-asymptotically regular and  the following holds for all $x\in C$:
\beq
\forall\eps\in (0,1]\,\exists N\le \Phi(K, L, d_C,\eps)
\bigg(d(x_N,Tx_N) < \eps \bigg),
\eeq 
where $\Phi(K, L, d_C, \eps)=2M$, with 
\[\ba{l}
M =\left\lceil\ds\frac{1}{\eps^2}\cdot 24L^2\left(5KD+D+\frac{11}{2}\right)(f(K))^3((1+K)D+1)^2\right\rceil, \\[0.2cm]
  D = \ds e^K\left(d_C+ 2\right), \quad f(K) = 2(1+(1+K)^2(2+K)).
\ea\]
\end{corollary}

\section{Proof mining in ergodic theory}\label{proof-mining-ergodic}

In this section, we apply proof mining techniques to obtain an explicit uniform bound on the metastability of ergodic averages in uniformly convex Banach spaces. This result was obtained by Kohlenbach and the author in \cite{KohLeu08b}. Our result can also be viewed as a finitary version in the sense of Terence Tao of the mean ergodic theorem for such spaces and so generalizes similar results obtained for Hilbert spaces by Avigad, Gerhardy and  Towsner \cite{AviGerTow08} and Tao \cite{Tao08}.

In the following $\N=\{ 1,2,3,\ldots\}$. Let $X$ be a Banach space and $T:X\to X$ be a self-mapping of $X.$ 
The {\em Cesaro mean} starting with $x\in X$ is the sequence $(x_n)_{n\geq 1}$ defined by $$\ds x_n:=\frac{1}{n}\sum_{i=0}^{n-1}T^ix.$$

Uniformly convex Banach spaces were introduced in 1936 by Clarkson in his seminal paper \cite{Cla36}. A Banach space $X$ is called {\em uniformly convex} if for all $\eps\in (0,2]$ there exists $\delta\in(0,1]$  such that for  all  $x,y\in X$, 
\beq
\| x\| \le 1, \quad \|y\|\le 1 \text{~~and~~}  \|x-y\|\geq \varepsilon  \text{~~imply~~} \left\|\frac12(x+y)\right\|\leq  1-\delta. 
\eeq
A mapping $\eta:(0,2]\to (0,1]$ providing such a $\delta:=\eta(\eps)$ for given $\eps\in(0,2]$ is called a {\em modulus of uniform convexity}. An example of a modulus of uniform convexity is Clarkson's {\em modulus of convexity} \cite{Cla36}, defined for any  Banach space $X$ as the function $\delta_X:[0,2]\to[0,1]$ given by
\beq
\delta_X(\varepsilon)=\inf\left\{1-\left\|\frac{x+y}{2}\right\|: \|x\|\le 1, \|y\|\le 1, \|x-y\|\ge \varepsilon \right\}.
\eeq
It is easy to see that $\delta_X(0)=0$ and that $\delta_X$ is nondecreasing. 
A well-known result is the fact that a Banach space $X$ is uniformly convex 
if and only if $\delta_X(\eps)>0$ for $\eps\in(0,2]$. Note that for uniformly convex Banach spaces $X$, $\delta_X$ is the largest modulus of uniform convexity.

In 1939, Garrett Birkhoff proved the following generalization of von Neumann's mean ergodic theorem.

\bthm\cite{Bir39}\label{habil-edts-MET-Birkhoff}
Let $X$ be a uniformly convex Banach space and $T:X\to X$ be a linear nonexpansive mapping.  Then for any $x\in X$, the Cesaro mean $(x_n)$ is convergent.
\ethm

In \cite{AviGerTow08}, Avigad, Gerhardy and Towsner addressed the issue of finding an effective rate of convergence 
for $(x_n)$ in Hilbert spaces. They showed that even for the separable Hilbert space $L_2$ there are simple computable such operators $T$ and computable points $x\in L_2$ such that there is no computable rate of convergence of $(x_n)$. In such a situation, the best one can hope for is an effective bound on the Herbrand normal form of the Cauchy property of $(x_n)$:  
\beq 
\forall \varepsilon >0 \,\forall g:\N\to\N\,\exists N\in\N\, \forall 
i,j\in [N,N+g(N)] \ 
\big(\| x_i-x_j\| < \varepsilon\big). \label{habil-edts-Herbrand}
\eeq
In \cite{Koh05} (see also \cite[Section 17.3]{Koh08-book}), Kohlenbach obtained general logical metatheorems for (uniformly convex) normed spaces, similar with the ones for metric or $W$-hyperbolic spaces presented in Section \ref{logical-meta}. These metatheorems guarantee, given a proof of (\ref{habil-edts-Herbrand}), 
the extractability of an effective bound $\Phi(\eps,g,b)$ on $\exists N$ in (\ref{habil-edts-Herbrand}) that is 
highly uniform in the sense that it only depends on $g,\varepsilon$ and an upper bound 
$b\ge \| x\|$ but otherwise is independent from $x,X$ and $T$. In fact, by a simple renorming argument one can always achieve to have the bound to depend on $b,\varepsilon$ only 
via $b/\varepsilon$.

Guided by this approach, Avigad, Gerhardy and Towsner \cite{AviGerTow08} extracted such a bound from a standard textbook proof of von Neumann's mean ergodic theorem.  A less direct proof for the existence of a bound with the above mentioned uniformity features is - for a particular finitary dynamical system - also given by Tao \cite{Tao08} as part of his proof of a generalization of the von Neumann mean ergodic theorem to commuting 
families of invertible measure preserving transformations $T_1,\ldots,T_l.$   

In \cite{KohLeu08b}, we apply the same methodology to Birkhoff's proof of  Theorem \ref{habil-edts-MET-Birkhoff} and extract an even easier to state bound for the more general case of uniformly convex Banach spaces. In this setting, the bound additionally depends on a given modulus of uniform convexity $\eta$ for $X$. Despite of our result being significantly more general then the Hilbert space case treated in \cite{AviGerTow08}, the extraction of 
our bound is considerably easier compared to \cite{AviGerTow08} and even numerically better.

\subsection{Logical discussion}

The proof of the above theorem can be formalized in the theory ${\cal A}^{\omega}[X,\|\cdot \|,\eta]$ of uniformly convex normed spaces, defined in \cite{Koh05}. We refer to \cite[Section 17.3]{Koh08-book} for details on this theory and the corresponding logical metatheorems. 

The conclusion of the above theorem is that $(x_n)$ converges for all $x\in C$, that is 
\beq
\exists l\in X\forall \eps>0\exists N\in\N\forall p\in\N\big(\|x_{N+p}-l\|<\eps\big),\label{habil-ergodic-limit}
\eeq
which is a $\exists\forall\exists\forall$-formula, so it has a too complicated logical form. One can cut down the complexity a little bit by considering the equivalent (for Banach spaces) conclusion that for all $x\in C$, $(x_n)$ is Cauchy, i.e.:
\beq
\forall \eps>0\exists N\in\N\forall p\in\N\big(\|x_{N+p}-x_N\|<\eps\big).\label{habil-ergodic-xn-Cauchy}
\eeq
The Cauchy property of $(x_n)$ is a $\forall\exists\forall$-formula, still too complicated. We are in a situation similar with the one in Subsection \ref{habil-app-fpt-as-ne}. The idea is again to consider the Herbrand normal form of the Cauchy property of $(x_n)$. As in the proof of Lemma \ref{habil-lemma-Herbrand-converges-0}, one can easily see that for all $x\in X$, the fact that $(x_n)$ is Cauchy is equivalent to
\beq
\forall \eps>0\, \exists N\in\N\, \forall m\in\N \,\forall i,j\in [N,N+m]\big(\|x_{i}-x_{j}\|<\eps\big), \label{reformulation-Caucy-1}
\eeq
which in turn is equivalent with its Herbrand normal form, given by (\ref{habil-edts-Herbrand}):
\bua
\forall \eps>0\, \forall g:\N\to\N\, \exists N\in\N\,\forall i,j\in [N,N+g(N)]\big(\|x_{i}-x_{j}\|<\eps\big).
\eua
As we have discussed above, the logical metatheorems guarantee the extractability of an effective bound $\Phi(\eps,g,b,\eta)$ on $\exists N$, where $b\ge \| x\|$ and $\eta$ is a modulus of uniform convexity of $X$. 

The only ineffective principle used in Birkhoff's original proof is the fact that any 
sequence $(a_n)$ of nonnegative real numbers has an infimum. We denote it with $(GLB)$.

In our analysis we first replace this analytical existential statement by a purely arithmetical one, namely 
\bua (GLB_{ar}): \quad 
\forall \eps>0\,\exists N\in\N\,\forall m\in\N\ (a_N\le a_m+\eps). 
\eua
For the general underlying facts from logic that guarantee this to be possible, we refer to \cite{Koh98} or to \cite[Chapter 13]{Koh08-book}.
The principle $(GLB_{ar})$ is still ineffective as, in general, there  is no computable bound on $\exists N$, even for computable $(a_n)$.  As above, we consider the equivalent reformulation
\bua 
\forall \eps>0\,\exists N\in\N\,\forall m\in\N\, \forall i\leq m (a_N\le a_i+\eps). 
\eua
and then we take its Herbrand normal form
\bua 
\forall \eps>0\,\forall g:\N\to\N\,\exists N\in\N\, \forall i\leq g(N) (a_N\le a_i+\eps). 
\eua
We carry out informally monotone functional interpretation, by which (GLB$_{ar}$) gets replaced in the proof by the quantitative form provided in Lemma \ref{habil-edts-GLB-lemma}.

\blem\label{habil-edts-GLB-lemma}\cite{KohLeu08b}\\
Let  $(a_n)_{n\geq 0}$ be a sequence of nonnegative real numbers. Then
\be
\item $\ds \forall\eps>0\, \forall g:\N\to\N\,\exists 
N\leq\Theta(b,\eps,g)\,\big(a_N\leq a_{g(N)}+ \eps\big), $\\
where $\ds \Theta(b,\eps,g)= \max_{i\leq K}g^i(1), \,\, 
b \geq  a_0,  \,\, K =\left\lceil \frac{b}{\varepsilon} 
\right\rceil$. \\
Moreover, $N=g^i(1)$ for some $i<K$.
\item  $\ds \forall\eps>0\, 
\forall g:\N\to\N\,\exists N\leq h^K(1)\,\forall 
m\leq g(N)\big(a_N\leq a_m+\eps\big),$\\
where $\ds  h(n)=\max_{i\leq n}g(i)$ and  $b, K$ are as above.
\ee
\elem
\noindent In the above lemma, $h^K$ is the $K$-th iterative of $h:\N\to\N$.

\subsection{Main results}

The  main result of the paper \cite{KohLeu08b} is the following quantitative version of Birkhoff's generalization to uniformly convex Banach spaces of von Neumann's mean ergodic Theorem.

\bthm\label{habil-edts-quant-B-VN}
Assume that $X$ is a uniformly convex Banach space, $\eta$ is a modulus of uniform convexity  and  $T:X\to X$ is a linear nonexpansive mapping. Let $b>0$. Then for all $x\in X$ with $\|x\|\leq b$,
\beq
\forall\eps>0\,\forall g:\N\to\N\,\exists N\leq \Phi(\eps,g,b,\eta)\,\forall i,j\in[N,N+g(N)]\,\big(\|x_i-x_j\|<\eps\big).
\eeq
where $\ds \Phi(\eps,g,b,\eta)=M\cdot\tilde{h}^K(1)$, with
$$\ba{l}
\ds M=\left\lceil\frac{16b}\eps\right\rceil, \quad \quad \gamma=\frac\eps{16}
\eta\left(\frac{\eps}{8b}\right), \quad \quad \ds K=\left\lceil 
\frac{b}{\gamma}\right\rceil,\\[0.1cm]
h,\,\tilde{h}:\N\to\N, \,\,h(n)=2(Mn+g(Mn)),  \quad \ds 
\tilde{h}(n)=\max_{i\leq n}h(i).\\[0.1cm]
 \quad 
\ds 
\ea
$$
If $\eta(\varepsilon)$ can be written as $\ds \varepsilon\cdot 
\tilde{\eta}(\varepsilon)$ with $0<\varepsilon_1\le\varepsilon_2\to 
\tilde{\eta}(\varepsilon_1)
\le \tilde{\eta}(\varepsilon_2),$ then we can replace $\eta$ by 
$\tilde{\eta}$ and the constant `$16$' by `$8$' in the 
definition of $\gamma$  in the bound above.
\ethm

Note that our bound $\Phi$ is independent from $T$ and depends on the space $X$ and the starting 
point $x\in X$ only via the modulus of convexity $\eta$ and the norm upper bound $b\ge \| x\|.$ Moreover, it is easy to see that the bound depends on $b$ and $\varepsilon$ only via $b/\varepsilon$.

It is well-known that as a modulus of uniform convexity of a Hilbert space $X$ one can take $\eta(\varepsilon)=\eps^2/8$ with $\tilde{\eta}(\eps)=\eps/8$ satisfying the requirements in the last claim 
of our main theorem.  As an immediate consequence, we get the following quantitative version of 
von Neumann's mean ergodic theorem.

\bcor\label{habil-edts-quant-Hilbert}
Assume that $X$ is a Hilbert space and  $T:X\to X$ is a $T:X\to X$ is a linear nonexpansive mapping. Let $b>0$. Then for all $x\in X$ with $\|x\|\leq b$,
\beq
\forall\eps>0\,\forall g:\N\to\N\,\exists N\leq 
\Phi(\eps,g,b)\,\forall i,j\in[N,N+g(N)]\,\big(\|x_i-x_j\|<\eps\big).
\eeq
where $(x_n),$ $\Phi$ are defined as above, but with 
$\ds K=\left\lceil \frac{512b^2}{\eps^2}\right\rceil$. 
\ecor

We get a similar result for $L_p$-spaces ($2<p<\infty$), using the fact that $\ds \eta(\eps)=\frac{\varepsilon^p}{p\,2^p}$ is a modulus of uniform convexity for $L_p$ 
(see e.g. \cite{Koh03}). Note that  $\ds \frac{\varepsilon^p}{p\,2^p} =\varepsilon \cdot
\tilde{\eta}_p(\varepsilon)$ with $\ds \tilde{\eta}_p(\varepsilon)=\frac{\varepsilon^{p-1}}{p\,2^p} $
satisfying the monotonicity condition in Theorem \ref{habil-edts-quant-B-VN}.

The bound extracted by Avigad et al. \cite{AviGerTow08} for Hilbert spaces  
is the following one:
\[ \Phi(\eps,g,b)= h^K(1),\]
where $\ds h(n)= n+2^{13}
\rho^4\tilde{g}((n+1)\tilde{g}(2n\rho)\rho^2)$, $\rho=\left\lceil \frac{b}{\varepsilon}\right\rceil$, $K=512\rho^2$ and $\ds \tilde{g}(n)=\max_{i\le n}(i+g(i))$. 
Note that, disregarding the different placement of 
`$\ds\lceil\cdot\rceil$',  the number of iterations $K$ in both this bound and in our 
bound in Corollary 
\ref{habil-edts-quant-Hilbert} coincide, whereas 
the function $h$ being iterated in our bound is much simper than that 
occurring in the above bound from  \cite{AviGerTow08}. 

Avigad et al. \cite{AviGerTow08} have an improved bound (roughly corresponding to our bound for $T$ 
being linear nonexpansive)  only in the special case when the linear mapping $T$ is an 
isometry. For this case, they show that one can take $h$ as  
\[ h(n)= n+2^{13}
\rho^4\tilde{g}\big((n+1)\tilde{g}(1)\rho^2\big), \] 
which still is somewhat more complicated than the function $h$ in our bound for the 
general case of $T$ being nonexpansive.  From this, Avigad et al. \cite{AviGerTow08} obtain in the isometric case that  $\ds\Phi(\eps,g,b)=2^{O(\rho^2\log \rho)}$ for linear functions $g,$ i.e. $g=O(n).$ 

Our bound in Corollary \ref{habil-edts-quant-Hilbert} generalizes this complexity upper bound on $\Phi$ 
to $T$ being nonexpansive rather than being an isometry.

\end{document}